%% file: APD_siam_revision2.tex
\documentclass[onefignum,onetabnum]{siamart171218}

\input{ex_shared}
\ifpdf
\hypersetup{
  pdftitle={\TheTitle},
  pdfauthor={\TheAuthors}
}
\fi

\DeclareMathOperator*{\argmin}{\arg\!\min}
\DeclareMathOperator*{\argmax}{\arg\!\max}

\input defs.tex

\def\fprod#1{\left\langle#1\right\rangle}

\def\ind#1{\mathbb{I}_{#1}}

\def\id{\mathbf{I}}

\def\sa#1{{#1}}
\def\rev#1{\textcolor{black}{#1}}

\definecolor{green}{RGB}{14,163,78}

\def\eyh#1{\textcolor{cyan}{#1}}
\def\eyz#1{\textcolor{teal}{#1}}

\begin{document}

\maketitle
\begin{abstract}
In this paper, we propose a primal-dual algorithm with a {novel momentum term using the partial gradients of the coupling function} that can be viewed as a generalization of the method proposed by Chambolle and Pock in 2016 
to solve saddle point problems defined by a convex-concave function $\cL(x,y)=f(x)+\Phi(x,y)-h(y)$ with a general coupling term $\Phi(x,y)$ that is \emph{not} assumed to be bilinear. {Assuming $\grad_x\Phi(\cdot,y)$ is Lipschitz 
for any fixed $y$, and {$\grad_y\Phi(\cdot,\cdot)$ is Lipschitz}, we show that the iterate sequence converges to a saddle point; and
for any $(x,y)$, 
we derive error bounds in terms of {$\cL(\bar{x}_k,y)-\cL(x,\bar{y}_k)$} for the ergodic sequence $\{\bar{x}_k,\bar{y}_k\}$.} In particular, we show $\cO(1/k)$ rate when the problem is merely convex in $x$. 
Furthermore, assuming $\Phi(x,\cdot)$ is linear 
for each fixed $x$ and $f$ is strongly convex, we obtain the 
ergodic convergence rate of $\cO(1/k^2)$ {-- we are not aware of another single-loop method in the related literature achieving the same rate when $\Phi$ is not bilinear.} Finally, we propose a backtracking technique which does not require the knowledge of Lipschitz constants while ensuring the same convergence results.
\sa{We 
also consider} convex optimization problems with nonlinear functional constraints and we show that using the backtracking scheme, \sa{the} optimal convergence rate can be achieved even when the dual domain is unbounded.
We tested our method against other state-of-the-art first-order algorithms and interior point methods for solving quadratically constrained quadratic problems with synthetic data, the kernel matrix learning and regression with fairness constraints arising in machine learning. 
\end{abstract}
\section{Introduction}
\label{sec:intro}
{Let $(\cX,\norm{\cdot}_{\cX})$ and $(\cY,\norm{\cdot}_{\cY})$ be finite dimensional, normed vector spaces.} In this paper, we study the following saddle point~(SP) problem:\vspace*{-1mm}
\begin{equation}\label{eq:original-problem}
(P):\quad \min_{x\in\cX}\max_{y\in\cY} \cL(x,y)\triangleq f(x)+\Phi(x,y)-h(y),\vspace*{-1mm}
\end{equation}
where $f:\cX\rightarrow \reals\cup\{+\infty\}$ and $h:\cY\rightarrow \reals\cup\{+\infty\}$ are convex functions (possibly nonsmooth) and $\Phi:\cX\times\cY\rightarrow \reals$ is a 
\sa{continuous function with certain differentiability properties}, convex in $x$ and concave in $y$. Our objective is to design an efficient first-order method to compute a saddle point of the structured convex-concave function $\cL$ in~\eqref{eq:original-problem}.
The problem $(P)$ covers a broad class of optimization problems, e.g., convex optimization with nonlinear conic constraints which itself includes LP, QP, QCQP, SOCP, and SDP as its subclasses. Indeed, consider \vspace*{-2mm}
\begin{equation}
\label{eq:conic_problem}
\min_{x\in\reals^n}~ \rho(x)\triangleq f(x)+g(x) \quad \hbox{s.t.}  \quad  G(x)\in -\cK,\vspace*{-2mm}
\end{equation}
{where $\cK\subseteq \cY^*$ is a closed convex cone in the dual space $\cY^*$}, $f$ is convex (possibly nonsmooth), $g$ is convex with a Lipschitz continuous gradient, $G:\cX\rightarrow \cY^*$ is
a smooth $\cK$-convex
, Lipschitz function having \sa{a Lipschitz continuous Jacobian}. 
Various optimization problems that frequently arise in many important 
applications are special cases of the conic problem in \eqref{eq:conic_problem}, e.g., primal or dual formulations of $\ell_1$ or $\ell_2$-norm soft margin SVM, ellipsoidal kernel machines~\cite{shivaswamy2007ellipsoidal}, kernel matrix learning~\cite{gonen2011multiple,lanckriet2004learning} etc.
Using Lagrangian duality, one can equivalently write \eqref{eq:conic_problem} as \vspace*{-1mm}
\begin{equation}\label{eq:conic_problem_equivalent}
\min_{x\in\reals^n}\max_{y\in\cK^*} f(x)+g(x)+\fprod{G(x),y},\vspace*{-1mm}
\end{equation}
which is a special case of \eqref{eq:original-problem}, i.e., $\Phi(x,y)=g(x)+\fprod{G(x),y}$ and $h(y)=\ind{\cK^*}(y)$ is the indicator function of $\cK^*$, where {$\cK^*\subseteq\cY$} denotes the dual cone of $\cK$.

{\bf Related Work.} Constrained convex optimization can be viewed as a special case of SP \sa{problem}~\eqref{eq:original-problem}, and recently some first-order methods and their randomized-coordinate variants are proposed to 
solve $\min\big\{ f(x)+g(x):G(x)\in-\reals^m_+\big\}$. In \cite{lin2017levelset}, a level-set method with iteration complexity guarantees is proposed for nonsmooth/smooth and strongly/merely convex settings. In \cite{xu2017first}, a primal-dual method based on the linearized augmented Lagrangian method (LALM) is proposed with \sa{$\cO(1/k)$} sublinear convergence rate in terms of suboptimality and infeasibility (see also \cite{yu2017primal} for another primal-dual algorithm 
with $\cO(1/k)$ rate. However, none of these methods can solve the more general SP problem we consider in this paper.

{SP problems \sa{have} become popular in recent years due to their generality and ability to directly solve constrained optimization problems with certain special structures. There has been several 
\sa{work on} first-order primal-dual algorithms for 
\eqref{eq:original-problem} when $\Phi(x,y)$ is bilinear, such as~\cite{chambolle2011first,dang2014randomized,chambolle2016ergodic,he2016accelerated,wang2017exploiting,du2018linear}, and few others have considered \sa{a} more general setting similar to this paper~\cite{palaniappan2016stochastic,nemirovski2004prox,juditsky2011first,he2015mirror,kolossoski2017accelerated} -- see Tseng's forward-backward-forward algorithm~\cite{tseng2000modified} for monotone inclusion problems and a projected reflected gradient method for monotone VIs~\cite{malitsky2015projected} which can also be used to solve~\eqref{eq:original-problem}.} Here, we briefly review some recent work that is closely related to ours. In the rest, 
we assume that \eqref{eq:original-problem} has a saddle point $(x^*,y^*)$.

In~\cite{chambolle2011first},
a special case of \eqref{eq:original-problem} with \sa{a} bilinear coupling term is studied:
\begin{align}\label{eq:CP-problem}
\min_{x\in\cX}\max_{y\in\cY} \hat{f}(x)+\fprod{Kx,~y}-h(y),
\end{align}
for some linear operator $K:\cX\rightarrow\cY^*$, where $\hat{f}$ and $h$ are \sa{closed} convex functions with easily computable prox {(Moreau) maps~\cite{hiriart2012fundamentals}}. The authors proposed a primal-dual algorithm which guarantees that \sa{$(x_k,y_k)$ converges to a saddle point $(x^*,y^*)$,} $\cL(\bar{x}_K,y^*)-\cL(x^*,\bar{y}_K)$ converges to $0$ with $\cO(1/K)$ rate when $\hat{f}$ is merely convex and with $\cO(1/K^2)$ rate when $\hat{f}$ is strongly convex, where $\{(\bar{x}_k,\bar{y}_k)\}_k$ is a weighted ergodic average sequence. Later, both Condat~\cite{condat2013primal} and Chambolle \& Pock~\cite{chambolle2016ergodic} studied some related primal-dual algorithms for an SP problem of the form in \eqref{eq:CP-problem} such that $\hat{f}$ has a composite convex structure, i.e., $\hat{f}(x)=f(x)+g(x)$ such that $f$ has an easy prox map and $g$ has Lipschitz continuous gradient -- also see~\cite{vu2013splitting} for a related method. In~\cite{condat2013primal}, convergence of the proposed algorithm is shown without providing any rate statements. In~\cite{chambolle2016ergodic}, it is shown that their previous work in~\cite{chambolle2011first} can be extended to handle non-linear proximity operators based on Bregman distance functions while guaranteing the same rate results -- see also~\cite{chen2014optimal} for an optimal method with $\cO(1/K)$ rate to solve \emph{bilinear} SP problems. \sa{Later Malitsky \& Pock \cite{malitsky2018first} proposed a primal-dual method with linesearch to solve \eqref{eq:CP-problem} with the same rate results as in \cite{chambolle2016ergodic}.}

In a recent work, He and Monteiro~\cite{he2016accelerated} 
considered a \emph{bilinear} SP problem 
from a monotone inclusion perspective. They proposed an accelerated algorithm based on hybrid proximal extragradient~(HPE) method, and showed that an $\epsilon$-saddle point $(x_\epsilon,y_\epsilon)$ can be computed within $\cO(1/\epsilon)$ iterations. More recently, Kolossoski and Monteiro~\cite{kolossoski2017accelerated} proposed another HPE-type method to solve a more general SP problem as in~\eqref{eq:original-problem} over \emph{bounded} sets {-- it is worth emphasizing that for nonlinearly constrained convex optimization, 
the dual optimal solution set may be unbounded and/or it may not be trivial to get an upper bound on a dual solution. Indeed, the method in~\cite{kolossoski2017accelerated} is an inexact proximal point method, each prox subinclusion (outer step) is solved using an accelerated gradient method (inner steps).} This work generalizes the method in~\cite{he2016accelerated} as the new method can deal with SP problems that are not bilinear, and it can use general Bregman distances instead of the Euclidean one.

Nemirovski~\cite{nemirovski2004prox} and {Juditsky \& Nemirovski~\cite{juditsky2011first}} also studied 
a convex-concave 
SP problem with a general coupling,
$\min_{x\in X}\max_{y\in Y} \Phi(x,y)$. Writing it as a variational inequality~\sa{(VI)} problem, they proposed a prox-type method, Mirror-prox. Assuming that $X$ and $Y$ are convex \emph{compact} sets, $\Phi$ 
is differentiable, and $F(x,y)=[\grad_x\Phi(x,y)^\top, -\grad_y\Phi(x,y)^\top]^\top$ is Lipschitz with constant $L$, $\cO(L/K)$ ergodic convergence rate is shown for Mirror-prox where in each iteration $F$ is computed \emph{twice} and a projection onto $X\times Y$ is computed with respect to a general (Bregman) distance. Moreover, \sa{in \cite{juditsky2011first}, for the case $\grad_y\Phi(x,\cdot)$ is linear for all $x$, i.e., $L_{yy}=0$,
assuming $Y$ is compact and $\Phi(\cdot,y)$ is strongly convex for any fixed $y$, 
convergence rate of $\cO(1/K^2)$ is shown for a \emph{multi-stage} method which repeatedly calls Mirror-Prox in each stage}.
More recently, He et al.~\cite{he2015mirror} extended \rev{the Mirror-Prox method in~\cite{nemirovski2004prox}} to 
\sa{handle} 
$\min_{x\in\cX}\max_{y\in\cY} f(x)+\Phi(x,y)-h(y)$ \sa{with the same convergence 
guarantees as in~\cite{nemirovski2004prox}}, where $f$ and $h$ are 
\sa{closed} convex 
with simple prox maps with respect to a general (Bregman) distance. In these papers, \sa{both the primal and dual step-sizes can be at most $1/L$.} \rev{Later, Malitsky~\cite{malitsky2018proximal} also considered} a monotone 
\sa{VI} problem 
of computing $z^*\in \sa{\cZ}$ such that $\fprod{F(z^*),z-z^*}+g(z)-g(z^*)\geq 0$ for all $z\in \sa{\cZ}$, where $F:\cZ\to\cZ$ is a monotone operator, $g$ is a \sa{proper closed} convex function \sa{and $\cZ$ is a finite-dimensional vector space with inner product}. The author proposed a proximal extrapolated gradient method (PEGM) with ergodic convergence rate of $\cO(1/K)$. The proposed method enjoys a backtracking scheme to estimate the local Lipschitz constant of the monotone map $F$ -- see also~\cite{malitsky2018forward} for a related line-search method in a more general setting of monotone inclusion problems.

\rev{Finally, while our paper was under review, we become aware of 
another primal-dual method proposed in~\cite{boob2019stochastic} for solving optimization problems with functional constraints considering convex/nonconvex problems and stochastic/deterministic oracle settings in a unified manner with convergence guarantees -- we compare our results with those in~\cite{boob2019stochastic} for convex and strongly convex minimization problems at the end of the contribution paragraph below.}

{\bf Application.} From the application perspective, there are many real-life problems arising in machine learning, signal processing, image processing, finance, etc. such that they can be formulated as a special case of \eqref{eq:original-problem}. In particular, the following problems arising in machine learning can be efficiently solved using the methodology proposed in this paper: 
\textbf{i)} robust classification under Gaussian uncertainty in feature observations leads to SOCP problems~\cite{bhattacharyya2005second}; \textbf{ii)} distance metric learning formulation proposed in~\cite{xing2003distance} is a convex optimization problem over positive semidefinite matrices subject to nonlinear convex constraints; \textbf{iii)} training ellipsoidal kernel machines~\cite{shivaswamy2007ellipsoidal} requires solving nonlinear SDPs; \textbf{iv)} learning a kernel matrix for transduction problem can be cast as an SDP or a QCQP~\cite{lanckriet2004learning,gonen2011multiple}.

In this paper, following~\cite{lanckriet2004learning}, we implemented our method for learning a kernel matrix to predict the labels of partially labeled data sets. To summarize the problem, {suppose we are given a set of labeled data points consisting of feature vectors $\{\ba_i\}_{i\in\cS}\subset\reals^m$, corresponding labels $\{b_i\}_{i\in\cS}\subset\{-1,+1\}$, and a set of unlabeled test data {$\{\ba_i\}_{i\in\cT}\subset\reals^m$}. Let $n_{tr}\triangleq |\cS|$ and $n_t\triangleq |\cT|$ denote the cardinality of the training and test sets, respectively, and define $n\triangleq n_{tr}+n_t$.} Consider $M$ different embedding of the data corresponding to kernel functions $k_\ell:\reals^m\times\reals^m\rightarrow\reals$ for $\ell=1,...,M$. Let $K_\ell\in\mathbb{S}^n_{+}$ be the kernel matrix such that $[K_\ell]_{ij}=k_\ell(\ba_i,\ba_j)$ for $i,j\in\cS\cup\cT$ and consider the partition of
{\small
$K_\ell=\left(
   \begin{array}{cc}
     K_\ell^{tr} & K_\ell^{tr,t} \\
     K_\ell^{t,tr} & K_\ell^t \\
   \end{array}
 \right)
$}, where $K_\ell^{tr}=[k_\ell(\ba_i,\ba_j)]_{i,j\in\cS}\in\mathbb{S}^{n_{tr}}$, $K_\ell^{t}=[k_\ell(\ba_i,\ba_j)]_{i,j\in\cT}\in\mathbb{S}^{n_{t}}$ and ${K_\ell^{t,tr}}^\top=K_\ell^{tr,t}=[k_\ell(\ba_i,\ba_j)]_{i\in\cS,j\in\cT}\in\reals^{n_{tr}\times n_t}$.

The objective is to learn a kernel matrix $K$ belonging to a class of kernel matrices which is a convex set generated by $\{K_\ell\}_{\ell=1}^M$, such that it minimizes the training error of a kernel SVM as a function of $K$. Skipping the details in~\cite{lanckriet2004learning}, one can study both $\ell_1$- and $\ell_2$-norm soft margin SVMs by considering the following generic formulation:\vspace*{-1mm}
{
\begin{align}
\label{eq:kernel_learn}
\min_{\substack{K\in\cK, \\ \text{trace}(K)=c}}\ \max_{\substack{\rev{\alpha:\ \mathbf{0}\leq \alpha\leq C\be}, \\ \fprod{\bb, \alpha}=0}} 2\be^\top \alpha - \alpha^\top (G(K^{tr})+\lambda\id)\alpha,\vspace*{-8mm}
\end{align}}%
where $c,C>0$ and $\lambda\geq 0$ are model parameters, \rev{$\be\in\reals^{n_{tr}}$ denotes the vector of ones}, $\bb=[b_i]_{i=1}^{n_{tr}}$ and $G(K^{tr})\triangleq\diag(\bb)K^{tr}\diag(\bb)$.
Suppose we want to learn a kernel matrix belonging to \rev{the class}  $\cK=\{\sum_{\ell=1}^M\eta_\ell K_\ell:\ \eta_\ell\geq 0,\ \ell=1,\ldots, M\}$; clearly, $K\in\cK$ implies $K\succeq 0$. For 
kernel class $\cK$,
\eqref{eq:kernel_learn} takes the following form: \vspace*{-2mm}
{
\begin{align}
\label{eq:kernel_learn_simple}
\min_{\substack{\eta:\ \fprod{\br,\eta}=c,\\ \eta\geq 0}}\ \max_{\substack{\rev{\alpha:\ \mathbf{0}\leq \alpha\leq C\be}, \\ \ \ \fprod{\bb,\alpha}=0}} 2\be^\top \alpha - \sum_{\ell=1}^M {\eta_\ell}\alpha^\top G(K^{tr}_\ell)\alpha-\lambda\norm{\alpha}_2^2,\vspace*{-8mm}
\end{align}}%
where $\eta=[\eta_\ell]_{\ell=1}^M$ and $\br=[r_\ell]_{\ell=1}^M$ for $r_\ell=\text{trace}(K_\ell)$. Clearly, \eqref{eq:kernel_learn_simple} is a special case of \eqref{eq:original-problem}.
In~\cite{lanckriet2004learning}, 
\eqref{eq:kernel_learn_simple} is equivalently represented as a QCQP and then solved using MOSEK~\cite{mosek}, a comercial interior-point method~(IPM). \rev{Per-iteration computational complexity} of a generic IPM is $\cO(Mn_{tr}^3)$ for solving the resulting QCQP~\cite{nesterov1994interior}, 
\rev{while it is $\cO(Mn_{tr}^2)$ for the first-order primal-dual method we proposed in this paper.} 
Therefore, when $n_{tr}$ is very large, IPMs are not suitable for solving large-scale problems unless the data matrix has certain sparsity structure; and in practice as $n_{tr}$ grows, the first-order methods with much lower per-iteration complexity will have the advantage over IPMs for computing low-to-medium level accuracy solutions.

{\bf Contribution.} 
We propose an \rev{accelerated primal-dual~(APD)} algorithm with a momentum term that can be viewed as a generalization of the method in \cite{chambolle2016ergodic} to solve SP problems with a more general coupling term $\Phi$ that is \emph{not} bilinear. Assuming {$\grad_y\Phi(\cdot,\cdot)$ is Lipschitz} and $\grad_x\Phi(\cdot,y)$ is Lipschitz for any fixed $y$,
we show that $(x_k,y_k)$ converges to a saddle point $(x^*,y^*)$ and \rev{for any $(x,y)\in\cX\times\cY$ we derive error bounds in terms of $\cL(\bar{x}_K,y)-\cL(x,\bar{y}_K)$} for the ergodic sequence {-- without requiring primal-dual domains to be bounded}; in particular, we show $\cO(1/K)$ rate when the problem is merely convex in $x$ using a constant step-size rule, \sa{where $K$ denotes the number of gradient computations}. Furthermore, assuming $\Phi(x,\cdot)$ is linear 
for each fixed $x$ and $f$ is strongly convex, we obtain the 
ergodic convergence rate of $\cO(1/K^2)$ --
{we are not aware of any other single-loop method with $\cO(1/K^2)$ rate when $\Phi$ is not bilinear.} Moreover, we develop a backtracking scheme, \rev{APDB,} which ensures that above stated rate results 
continue to hold in terms of the total number of gradient computations even though the Lipschitz constants, $L_{xx}$, $L_{yx}$ and $L_{yy}$, are \emph{not} known -- see Assumption~\ref{assum}. To best of our knowledge, for the strongly convex-concave SP problems, a line-search method ensuring $\cO(1/K^2)$ rate is proposed for the first time for when the coupling function $\Phi$ is \emph{not} bilinear.
\rev{In the context of constrained optimization problems, the backtracking scheme helps us demonstrate convergence results even when a dual bound is not available or easily computable. Our results continue to hold when the dual optimal solution set is \emph{unbounded}.}

{The previous art for solving SP problems in the general setting include the Mirror-Prox algorithm in~\cite{nemirovski2004prox,he2015mirror}, the HPE-type method in~\cite{kolossoski2017accelerated} \sa{and the PEGM by Malitsky~\cite{malitsky2018proximal}}.
All these methods including ours have $\cO(1/\epsilon)$ complexity under mere convexity; however, \rev{our APD and APDB methods both have} an improved $\cO(1/\sqrt{\epsilon})$ rate when $f$ is strongly convex.} 
Indeed, all the rates derived here are the optimal rates for \sa{the settings considered in this paper -- see~\cite{ouyang2018lower} for the lower complexity bounds of $\cO(1/K)$ and $\cO(1/K^2)$ associated with first-order primal-dual methods for convex-concave and strongly convex-concave bilinear SP problems, respectively}. {When compared to~\cite{kolossoski2017accelerated}, ours is a simpler one-loop algorithm while HPE~\cite{kolossoski2017accelerated} is a two-loop method, \sa{having outer and inner iterations}, requiring a more stronger oracle for subproblems -- see Remark~\ref{rem:prox} and also requiring a bounded domain. Moreover, while convergence to a \emph{unique} limit point is shown for APD, a limit point result (weaker than ours) is shown in~\cite{kolossoski2017accelerated}, i.e., any limit point is a saddle point -- see end of p.1254 in~\cite{kolossoski2017accelerated}.} Another competitor algorithm, Mirror-Prox, requires computing both the primal and dual gradients \emph{twice} during each iteration while the proposed APD method only needs to compute them once; thus, saving the computation cost by half yet achieving the same iteration complexity \rev{-- see Remark~\ref{rem:MP-APD}}. Moreover, 
under the assumption that $\grad \Phi(\cdot,\cdot)$ is Lipschitz with constant $L$, the method in~\cite{he2015mirror} has a primal-dual step-size less than $1/L$; 
compared to \cite{he2015mirror} our assumption on $\Phi$ is weaker, 
our primal and dual step-sizes are larger than $1/L$ {-- see Remark~\ref{rem:differentiability} for further weakening the assumptions on $\Phi$.}
{Finally, the numerical 
results \rev{also} clearly demonstrate that APD has roughly the same iteration complexity as proximal Mirror-Prox; but, requires half the computational efforts (reflected by the savings in computation time).} \sa{Finally, setting $z=[x^\top, y^\top]^\top$, $F(x,y)=[\grad_x\Phi(x,y)^\top, -\grad_y\Phi(x,y)^\top]^\top$ and $g(x,y)=f(x)+h(y)$ within the VI problem of~\cite{malitsky2018proximal} mentioned above, PEGM can deal with \eqref{eq:original-problem}. It is worth emphasizing that PEGM utilizes a single step-size to update the next iterate and uses $\norm{F(z)-F(\bar{z})}$ to estimate the Lipschitz constant $L=L_{xx}+2L_{yx}+L_{yy}$ locally within the backtracking procedure while our method uses two different step-sizes (one for primal and one for dual updates) 
to locally approximate $L_{xx}+L_{yx}$ and $L_{yx}+2L_{yy}$ for choosing primal and dual step-sizes, respectively -- see Assumption~\ref{assum}. Empirically, we have observed that exploiting the special structure of SP compared to more general VI problems and allowing primal and dual steps chosen separately lead to larger step-sizes, speeding up the convergence in practice.}

\rev{For the composite convex problem in \eqref{eq:conic_problem} 
with $\cK=\reals^m_+$, the 
method in~\cite{boob2019stochastic} can deal with \emph{unbounded} dual domain 
through employing an extra linearization for the constraint function.
Provided that a bound $B\geq\norm{y^*}+1$ is \emph{known}, where $y^*$ denotes an arbitrary optimal dual solution to \eqref{eq:conic_problem}, the deterministic method in~\cite{boob2019stochastic} can compute an $\epsilon$-optimal and $\epsilon$-feasible solution to \eqref{eq:conic_problem} 
with $\cO(1/\epsilon)$ and $\cO(1/\sqrt{\epsilon})$ complexity when $g$ is convex and strongly convex, respectively.
 On the other hand, if $B<\norm{y^*}+1$, then the complexity drops to $\cO(1/\epsilon^2)$ and $\cO(1/\epsilon)$ for convex and strongly convex settings, respectively --see the discussion after \cite[Corollary 2.2 and Theorem 2.3]{boob2019stochastic}. Furthermore, the step-sizes 
 in~\cite{boob2019stochastic} require the knowledge of global Lipschitz constants, and the established convergence rate for the convex setting is non-asymptotic, as the step-size is chosen depending on the tolerance (see $\eta$ choice in~\cite[Theorem 2.3]{boob2019stochastic}); therefore, the sequence won't converge to a primal-dual pair in the limit. In contrast, our 
 backtracking scheme APDB generates an asymptotically optimal 
 sequence, and even if a bound on $\norm{y^*}$ is \emph{not known}, it 
 achieves the optimal complexity guarantees of $\cO(1/\epsilon)$ and $\cO(1/\sqrt{\epsilon})$ oracle calls (in total, including those for line search) for convex and strongly convex settings, respectively
 --see Section~\ref{sec:without_dualbound}, where we show that even if $L_{xx}$ does not exist (possibly due to unbounded dual domain), \rev{APDB, i.e., the APD with backtracking,} generates a convergent primal-dual sequence without knowing any bound or Lipschitz constants of the problem.}

{\bf Organization of the Paper.} In the coming section, we precisely state our assumptions on $\cL$ in \eqref{eq:original-problem}, describe \sa{the proposed algorithms, APD and APD with backtracking~(APDB), and present convergence guarantees for APD and APDB iterate sequences, which are the main results of this paper.} Subsequently, in Section~\ref{sec:methodology}, we provide an easy-to-read convergence analysis proving the main results. \sa{Next, in Section~\ref{sec:constrained}, we discuss how APD and APDB can be implemented for solving constrained convex optimization problems.} Later, in Section~\ref{sec:numeric}, 
we apply our APD and APDB methods to solve the kernel matrix learning and \sa{QCQP problems to numerically compare APD and APDB with the {Mirror-prox} method
~\cite{he2015mirror}, PEGM~\cite{malitsky2018proximal}} and off-the-shelf interior point methods. Finally, Section~\ref{sec:conclude} concludes the paper.
\section{\rev{The} Accelerated Primal-Dual~\rev{(APD)} Algorithm}
\begin{defn}
\label{def:bregman}
{Let $\varphi_{\cX}:\cX\rightarrow\reals$ and $\varphi_{\cY}:\cY\rightarrow\reals$ be differentiable functions on open sets containing $\dom f$ and $\dom h$, respectively. Suppose $\varphi_{\cX}$ and $\varphi_{\cY}$ have closed domains and are 1-strongly convex with respect to $\norm{\cdot}_{\cX}$ and $\norm{\cdot}_{\cY}$, respectively. Let $\bD_{\cX}:\cX\times\cX\rightarrow\reals_+$ and $\bD_{\cY}:\cY\times\cY\rightarrow\reals_+$ be Bregman distance functions corresponding to $\varphi_{\cX}$ and $\varphi_{\cY}$,
i.e., $\bD_{\cX}(x,\bar{x})\triangleq \varphi_{\cX}(x)-\varphi_{\cX}(\bar{x})-\fprod{\grad \varphi_\cX(\bar{x}),x-\bar{x}}$, and $\bD_{\cY}$ \sa{is defined similarly 
using the same form.}
}
\end{defn}
Clearly, $\bD_\cX(x,\bar{x})\geq \tfrac{1}{2}\norm{x-\bar{x}}_\cX^2$ for $x\in\cX$, $\bar{x}\in\dom f$, and $\bD_{\cY}(y,\bar{y})\geq\tfrac{1}{2}\norm{y-\bar{y}}_\cY^2$ for $y\in\cY$ and $\bar{y}\in\dom h$. The dual spaces are denoted by $\cX^*$ and $\cY^*$. 
For $x'\in\cX^*$, we define the dual norm $\norm{x'}_{\cX^*}\triangleq\max\{\fprod{x',x}:\ \norm{x}_\cX\leq 1\}$, and $\norm{\cdot}_{\cY^*}$ is defined similarly.
{We next state {our main assumption and} 
explain the APD algorithm for 
\eqref{eq:original-problem}, and discuss its convergence properties 
as the main results of this paper.}

\begin{assumption}\label{assum} Suppose $\bD_{\cX}$ and $\bD_\cY$ be some Bregman distance functions as in Definition~\ref{def:bregman}.
In case {$f$ is strongly convex, i.e., $\mu>0$,} we fix $\norm{x}_\cX=\sqrt{\fprod{x,x}}$, 
and set $\bD_{\cX}(x,\bar{x})=\frac{1}{2}\norm{x-\bar{x}}_\cX^2$.

Suppose $f$ and $h$ are closed convex, 
and $\Phi$ is \sa{continuous} 
such that\\
{\bf (i)} for any 
\sa{$y\in\dom h\subset \cY$}, \sa{$\Phi(\cdot,y)$} is convex and differentiable, 
and {for some $L_{xx}\geq 0$,}
\begin{equation}
\label{eq:Lxx}
\norm{\grad_x \Phi(x,y)-\grad_x \Phi(\bar{x},y)}_{\cX^*}\leq L_{xx}\norm{x-\bar{x}}_\cX,\quad \forall~x,\bar{x}\in\dom f\subset \cX,
\end{equation}
{\bf (ii)} for any 
$x\in\dom f$, \sa{$\Phi(x,\cdot)$} is concave and differentiable; 
\sa{there exist $L_{yx}>0$ and $L_{yy}\geq 0$ such that for all $x,\bar{x}\in\dom f$ and $y,\bar{y}\in\dom h$, one has}
\begin{align}
\label{eq:Lipschitz_y}
\norm{\grad_y \Phi(x,y)-\grad_y \Phi(\bar{x},\bar{y})}_{\cY^*}\leq L_{yy}\norm{y-\bar{y}}_\cY+{L_{yx}}\norm{x-\bar{x}}_\cX.
\end{align}
\end{assumption}

\sa{
We first analyze the convergence properties of {APD}, displayed in Algorithm~\ref{alg:APD}, which repeatedly calls for the subroutine \textbf{MainStep} stated in Algorithm~\ref{alg:mainstep}.}
\begin{remark}
\label{rem:prox}
\sa{$x$- and $y$-subproblems of APD are generalizations of the Moreau map~\cite{hiriart2012fundamentals}. Compared to ours, i.e., $\argmin_{y \in \cY} \left\{ h(y)-\fprod{s,y}+\tfrac{1}{\sigma}\bD_{\cY}(y,\bar{y}) \right\}$,} HPE-type method in~\cite{kolossoski2017accelerated} requires solving $\argmax_{y\in Y} \Phi(\bar{x},y)-\bD_{\cY}(y,\bar{y})/\sigma$ as the $y$-subproblem for some given $\bar{x}$ and $\bar{y}$ where $Y\subset\cY$ is a bounded convex set. 
This may not be a trivial operation in general.
\vspace*{-2mm}
\end{remark}

\begin{algorithm}[h!]
   \caption{\textbf{MainStep}$(\bar{x},\bar{y},x_p,y_p,\tau,\sigma,\theta)$}
   \label{alg:mainstep}
\begin{algorithmic}[1]
   \STATE{\bfseries input: $\tau,\sigma,\theta>0$, $(\bar{x},\bar{y})\in\cX\times\cY$, $(x_p,y_p)\in\cX\times\cY$}
   		\STATE{$s\gets (1+\theta)\grad_y \Phi(\bar{x},\bar{y})-\theta\grad_y\Phi(x_p,y_p)$} \label{algeq:s}
   		\STATE{$\hat{y}\gets \argmin_{y\in\cY} h(y)-\fprod{s,~ y}+{\frac{1}{\sigma}}\bD_{\cY}(y,\bar{y})$} \label{algeq:y-main}
   		\STATE{$\hat{x}\gets \argmin_{x\in\cX} f({x})+\fprod{\grad_x\Phi(\bar{x},\hat{y}), ~x}+{\frac{1}{\tau}}\bD_{\cX}(x,\bar{x})$} \label{algeq:x-main}
   \RETURN $(\hat{x},\hat{y})$
\end{algorithmic}
\end{algorithm}
\begin{algorithm}[h!]
   \caption{Accelerated Primal-Dual algorithm~(APD)}
   \label{alg:APD}
\begin{algorithmic}[1]
   \STATE{\bfseries Input: 
   $\mu\geq 0$, 
   $\sa{\tau_0,\sigma_0>0}$, 
   $(x_0,y_0)\in\cX\times\cY$}
   \STATE{$(x_{-1},y_{-1})\gets(x_0,y_0)$, \sa{$\sigma_{-1}\gets\sigma_0$, $\gamma_0\gets \sigma_0/\tau_0$}}
   \FOR{$k\geq 0$}
        \STATE{\sa{$\sigma_{k}\gets \gamma_k\tau_{k}$,\quad $\theta_{k}\gets\frac{\sigma_{k-1}}{\sigma_k}$}}
        \STATE $(x_{k+1},y_{k+1})\gets\hbox{\textbf{MainStep}}(x_k,y_k,x_{k-1},y_{k-1},\tau_k,\sigma_k,\theta_k)$
        \STATE{\sa{$\gamma_{k+1}\gets \gamma_k(1+\mu\tau_k)$,\quad  $\tau_{k+1}\gets{\tau_k}\sqrt{\frac{\gamma_{k}}{\gamma_{k+1}}}$,\quad $k\gets k+1$}} \label{algeq:gamma_update_APD}
   \ENDFOR
\end{algorithmic}
\end{algorithm}

Recall that if {$f$ is convex with modulus $\mu\geq 0$}, then
{
\begin{align}\label{assum-sc}
f(x)\geq f(\bar{x})+\fprod{g,~x-\bar{x}}+\frac{\mu}{2}\norm{x-\bar{x}}_\cX^2,\quad \forall~g\in \partial f(\bar{x}),\quad \forall~x,\bar{x}\in\dom f.
\end{align}}%
Note also that \eqref{eq:Lxx} and convexity imply that for any $y\in\dom h$ and  $x,\bar{x}\in\dom f$,
{
\begin{align}
0 &\leq \Phi(x,y) - \Phi(\bar{x},y)-\fprod{\grad_x\Phi(\bar{x},y),~x-\bar{x}} \leq\frac{L_{xx}}{2}\norm{x-\bar{x}}_\cX^2.
\label{eq:Lxx_bound}
\end{align}}%

\begin{theorem}\label{thm:main} \sa{\bf (Main Result I)}
Let $\bD_{\cX}$ and $\bD_\cY$ be some Bregman distance functions as in Definition~\ref{def:bregman}. Step-size update rule in Algorithm~\ref{alg:APD} implies that \vspace*{-3mm}
\begin{align}\label{eq:step-size-rule}
\theta_{k+1}=\frac{1}{\sqrt{1+\mu\tau_k}},\quad \tau_{k+1}=\theta_{k+1}\tau_k,\quad \sigma_{k+1}=\sigma_k/\theta_{k+1},\quad\forall~k\geq 0.
\end{align}%
\vspace*{-3mm}

Suppose Assumption~\ref{assum} holds, and
$\{x_k,y_k\}_{k\geq 0}$ is 
generated by APD, stated in Algorithm~\ref{alg:APD}, starting from $\tau_0,\sigma_0>0$ such that
{\small
\begin{align}
\label{eq:initial_step_condition}
\Big(\frac{1-\delta}{\tau_0}-L_{xx}\Big)\frac{1}{\sigma_0}\geq \frac{{L^2_{yx}}}{c_\alpha},\quad\ 1-\big(\delta+c_\alpha+c_\beta\big)\geq \frac{L_{yy}^2}{c_\beta}\sigma^2_0,
\end{align}}%
for some $\delta, c_\alpha,\ c_\beta\in\reals_+$ such that $c_\alpha+c_\beta+\delta\leq 1$ satisfying $c_\alpha,~c_\beta> 0$ when $L_{yy}>0$, and $c_\alpha>0$, $c_\beta=0$ when $L_{yy}=0$.
Then for any $(x,y)\in\cX\times\cY$,\vspace*{-1mm}
\begin{align}
{\cL(\bar{x}_K,y)-\cL(x,\bar{y}_K) \leq} \frac{1}{T_K}{\Delta(x,y),\quad \Delta(x,y)\triangleq \frac{1}{\tau_0}\bD_{\cX}(x,x_0)+\frac{1}{\sigma_0}\bD_{\cY}(y,y_0),}\label{eq:delta}
\end{align}
holds for all $K\geq 1$, where $\bar{x}_K\triangleq\frac{1}{T_K}\sum_{k=0}^{K-1} t_k x_{k+1}$, $\bar{y}_K\triangleq\frac{1}{T_K}\sum_{k=0}^{K-1}t_ky_{k+1}$ and $T_K\triangleq\sum_{k=0}^{K-1}t_k$ for some \rev{$\{t_k\}_{k\geq 0}\subset\reals_{++}$ as stated in Parts I and II below:} 

\textbf{(Part I.)} {Suppose $\mu=0$, step-size rule in Algorithm~\ref{alg:APD} implies $\tau_k=\tau_0$, $\sigma_k=\sigma_0$ and $\theta_k=1$ for $k\geq 0$. Then \eqref{eq:delta} holds for $\{t_k\}$ such that $t_k=1$ for $k\geq 0$; hence, $T_K=K$.}
If a saddle point for \eqref{eq:original-problem} exists and
{$\delta>0$}, then 
$\{(x_k,y_k)\}_{k\geq 0}$ converges to a saddle point $(x^*,y^*)$ such that $0\leq \cL(\bar{x}_K,y^*)-\cL(x^*,\bar{y}_K) \leq \cO(1/K)$ and
{\small
\begin{align}
\label{eq:xy-bound-I}
{\gamma_0\bD_{\cX}(x^*,x_{K}) +[1-(c_\alpha+c_\beta)]\bD_{\cY}(y^*,y_{K}) \leq \sigma_0\Delta(x^*,y^*)}. 
\end{align}}%

\textbf{(Part II.)} Suppose $\mu>0$ and $L_{yy}=0$, in this setting let $\norm{x}_\cX=\sqrt{\fprod{x,x}}$, 
and $\bD_{\cX}(x,\bar{x})=\frac{1}{2}\norm{x-\bar{x}}_\cX^2$.
{Then \eqref{eq:delta} holds for $\{t_k\}$ such that $t_k=\sigma_k/\sigma_0$ for $k\geq 0$, and $T_K 
=\Theta(K^2)$.}\footnote{$f(k)=\Theta(k)$ means $f(k)=\cO(k)$ and $f(k)=\Omega(k)$.}
{If a saddle point for \eqref{eq:original-problem}
exists,\footnote{\rev{Since $\mu>0$, all saddle points share the same unique $x$-coordinate, say $x^*\in\cX$.}} then $\{x_k\}_{k\geq 0}$ converges to $x^*$ and $\{y_k\}$ has a limit point. Moreover, if $\delta>0$, then any limit point $(x^*,y^*)$ 
 is a saddle point and it holds that $0\leq \cL(\bar{x}_K,y^*)-\cL(x^*,\bar{y}_K) \leq \cO(1/K^2)$ and}
 {\small
 \begin{align}
 \label{eq:xy-bound-II}
 \gamma_K\bD_{\cX}(x^*,x_{K}) +(1-c_\alpha)\bD_{\cY}(y^*,y_{K}) \leq \sigma_0
\Delta(x^*,y^*) 
\end{align}}%
holds with $\gamma_K=\Omega(K^2)$, which implies that $\bD_{\cX}(x^*,x_{K})=\cO(1/K^2)$.
\end{theorem}
\begin{proof}
See Section~\ref{sec:main_thm_proof} for the proof of \rev{the main result I}.
\end{proof}
\begin{remark}
\sa{The particular choice of initial step-sizes, $\tau_0= c_\tau(L_{xx}+L_{yx}^2/\alpha)^{-1}$ and $\sigma_0= c_\sigma(\alpha+2L_{yy})^{-1}$ for any $\alpha>0$ and $c_\tau,\ c_\sigma\in(0,1]$, satisfies \eqref{eq:initial_step_condition}.}
\end{remark}
\begin{remark}
\label{rem:CP-step-size-cond}
\sa{The requirement in \eqref{eq:initial_step_condition} generalizes the step-size condition in~\cite{chambolle2016ergodic} for \eqref{eq:CP-problem} with $\hat{f}(x)=f(x)+g(x)$ such that $f$ is closed convex and $g$ is convex having Lipschitz continuous gradient with constant $L_g$. It is required in~\cite{chambolle2016ergodic} that $\left(\frac{1}{\tau_0}-L_g\right)\frac{1}{\sigma_0}\geq \norm{K}^2$. For \eqref{eq:CP-problem}, $\Phi(x,y)=g(x)+\fprod{Kx,y}$; hence, $L_{xx}=L_g$, $L_{yx}=\norm{K}$ and $L_{yy}=0$. Note when $L_{yy}=0$, the second condition in~\eqref{eq:initial_step_condition} holds for all $\sigma_0>0$; thus, setting $c_\alpha=1$, $c_\beta=0$ and $\delta=0$, \eqref{eq:initial_step_condition} reduces to the condition in~\cite{chambolle2016ergodic}.}

\sa{Moreover, when $f$ is strongly convex, it is shown in \cite{chambolle2016ergodic} that $\tau_0=\frac{1}{2L_g}$ and $\sigma_0=\frac{L_g}{\norm{K}^2}$ can be used to achieve an accelerated rate of $\cO(1/K^2)$ -- see~\cite[Section~5.2]{chambolle2016ergodic}. Note since $L_{yy}=0$, setting $c_\alpha=1$, $c_\beta=0$ and $\delta=0$, we see that $\tau_0=\frac{1}{2L_{xx}}$ and $\sigma_0=\frac{L_{xx}}{L_{yx}^2}$ satisfies \eqref{eq:initial_step_condition}, and these initial step-sizes are the same as those in~\cite[Section~5.2]{chambolle2016ergodic}.}
\end{remark}
\begin{remark}
\label{rem:differentiability}
As in~\cite{kolossoski2017accelerated}, assuming a stronger oracle we can remove 
\sa{the assumption of $\grad_x\Phi(\cdot,y)$ being Lipschitz for each $y$, i.e., if we replace Line~\ref{algeq:x-main} of \textbf{MainStep} with $\argmin_x f(x)+\Phi(x,\hat{y})+D_{\cX}(x,\bar{x})/\tau$,} then we can remove assumption in \eqref{eq:Lxx}; hence, even if $\Phi$ is nonsmooth in $x$, all our rate results will continue to hold. Let $\Phi(x,y)=x^2y$ on $x\in[-1,1]$ and $y\geq 0$, \sa{i.e., $f(x)=\ind{[-1,1]}(x)$ and $h(y)=\ind{\reals_+}(y)$}; the Lipschitz constant $L$ for $\grad \Phi$ would not exist in this case, and it is not clear how one can modify the analysis of~\cite{he2015mirror} to deal with problems when $\Phi$ is not jointly differentiable.\vspace*{-4mm}
\end{remark}
\rev{\begin{remark}
\label{rem:MP-APD}
{At each iteration $k\geq 0$,} our proposed algorithm, APD, {requires computing} only one pair of primal-dual gradients, i.e., $\grad_x\Phi(x_{k+1},y_k)$ and $\grad_y\Phi(x_k,y_k)$ -- note that $\grad_y\Phi(x_{k-1},y_{k-1})$ required for iteration $k$ can be retrieved from iteration $k-1$. However, Mirror-prox~{(MP)} in~\cite{he2015mirror} uses \emph{two} pairs of primal-dual gradients at each iteration $k\geq 0$. 
{Thus, for a total of $K\geq 1$ iterations, while APD uses $K$ pairs of primal-dual gradients, Mirror-prox uses $2K$ pairs.} 
{Next we compare the iteration complexity of the two methods}. 
Consider the problem \eqref{eq:original-problem} and 
suppose $f(\cdot)=\ind{X}(\cdot)$ and $h(\cdot)=\ind{Y}(\cdot)$ for some compact convex sets $X\subset\cX$ and $Y\subset\cY$. Let $Z\triangleq X\times Y$, and $L$ be the global Lipschitz constant of {$\grad\Phi(\cdot,\cdot)$ over $Z$}. Moreover, let $z=[x^\top;y^\top]^\top$ and define $\bD_Z(z,\bar{z})\triangleq \bD_X(x,\bar{x})+\bD_Y(y,\bar{y})$, for any $z,\bar{z}\in Z$. After $K\geq 1$ iterations, MP~\cite{he2015mirror} can bound $\cG(\bar{z}_K)\triangleq\sup_{(x,y)\in X\times Y} (\Phi(\bar{x}_K,y)-\Phi(x,\bar{y}_K))$ as $\cG(\bar{z}_K)\leq R_{\rm MP}(K)\triangleq\frac{L}{K}\sup_{z\in Z} {\bD_Z(z,z_0)}$; in comparison, APD guarantees that
{\small
\begin{align*}
    &\cG(\bar{z}_K)\leq R_{\rm APD}(K)\triangleq\frac{1}{K}\sup_{(x,y)\in X\times Y}\{(L_{xx}+L_{yx}){\bD_X(x,x_0)}+(2L_{yy}+L_{yx}){\bD_Y(y,y_0)}\}.
\end{align*}}%
Note that from  \eqref{eq:Lxx} and  \eqref{eq:Lipschitz_y} one can conclude that 
{$\grad\Phi$ is Lipschitz with $L=L_{xx}+2L_{yx}+L_{yy}$. Thus, 
when $L_{yy}=0$, which is the case for constrained convex optimization problems (as the Lagrangian is an affine function of the dual variable), {MP} bound is larger than the APD bound, i.e., $R_{\rm APD}(K)\leq R_{\rm MP}(K)$ for $K\geq 1$}. Therefore, in this setting, for any $\epsilon>0$, the number of primal-dual gradient calls required by MP to ensure $R_{\rm MP}(K)\leq \epsilon$ is at least \emph{twice} the number of APD primal-dual gradient calls to ensure the same accuracy $R_{\rm APD}(K)\leq \epsilon$. 
\end{remark}}
Our method generalizes the primal-dual method proposed by~\cite{chambolle2016ergodic} to solve SP problems with coupling term $\Phi$ that is \emph{not} bilinear. 
According to Remark~\ref{rem:CP-step-size-cond}, \rev{to solve \eqref{eq:CP-problem},} any $\tau_0,\sigma_0>0$ such that $\left(\frac{1}{\tau_0}-L_g\right)\frac{1}{\sigma_0}\geq \norm{K}^2$ work for both Algorithm~1 in~\cite{chambolle2016ergodic} and APD, and both methods generate the same iterate sequence with same error bounds.
Similarly, from Remark~\ref{rem:CP-step-size-cond}, in case $f$ is strongly convex, 
when $\{(\tau_k,\sigma_k,\theta_k)\}$ is chosen as
\sa{in~\eqref{eq:step-size-rule} starting from $\tau_0=\frac{1}{2L_g}$ and $\sigma_0=\frac{L_g}{\norm{K}^2}$}, our APD algorithm and Algorithm~4 in~\cite{chambolle2016ergodic} again output the same iterate sequence with the same error bounds. Therefore, APD algorithm inherits the already established connections of the primal-dual framework in~\cite{chambolle2016ergodic} to other well-known methods, e.g., {(linearized) ADMM~\cite{shefi2014rate,aybat2018distributed} and Arrow-Hurwicz method \cite{arrow1958studies}.}
\begin{algorithm}[h!]
   \caption{Accelerated Primal-Dual algorithm with Backtracking~(APDB)}
   \label{alg:APDB}
\begin{algorithmic}[1]
   \STATE{\bfseries Input: 
   $(x_0,y_0)\in\cX\times\cY$, $\mu\geq 0$, \sa{$c_\alpha,c_\beta,\delta\geq 0$}, $\eta\in (0,1)$, 
   $\sa{\bar{\tau},\gamma_0>0}$}
   \STATE{$(x_{-1},y_{-1})\gets(x_0,y_0)$, \sa{$\tau_0\gets\bar{\tau}$, $\sigma_{-1}\gets\gamma_0\tau_0$}}
   \FOR{$k\geq 0$}
        \LOOP
        \STATE{\sa{$\sigma_{k}\gets \gamma_k\tau_{k}$,\quad $\theta_{k}\gets\frac{\sigma_{k-1}}{\sigma_k}$},\quad \sa{$\alpha_{k+1}\gets c_\alpha/\sigma_k$,\quad $\beta_{k+1}\gets c_\beta/\sigma_k$}}
        \STATE{$(x_{k+1},y_{k+1})\gets\hbox{\textbf{MainStep}}(x_k,y_k,x_{k-1},y_{k-1},\tau_k,\sigma_k,\theta_k)$}\label{algeq:APDB-mainstep}
        \IF{\sa{$E_k(x_{k+1},y_{k+1})\leq -\frac{\delta}{\tau_k}\bD_\cX(x_{k+1},x_k)-\frac{\delta}{\sigma_k}\bD_\cY(y_{k+1},y_k)$}}\label{algeq:test_function}
        \STATE{\textbf{go to} Line 13}
        \ELSE
   		\STATE{\sa{$\tau_{k}\gets \eta \tau_k$}}
   		\ENDIF
        \ENDLOOP
        \STATE{\sa{$\gamma_{k+1}\gets \gamma_k(1+\mu\tau_k)$},\quad \sa{$\tau_{k+1}\gets{\tau_k}\sqrt{\frac{\gamma_{k}}{\gamma_{k+1}}}$},\quad \sa{$k\gets k+1$}} \label{algeq:gamma_update_APDB} 
   \ENDFOR
\end{algorithmic}
\end{algorithm}
%

\sa{For some problems, it may either be hard to guess\rev{/know} the Lipschitz constants, $L_{xx}$, $L_{yx}$ and $L_{yy}$, or using these constants may well lead to too conservative step-sizes. Next, inspired by the work~\cite{malitsky2018first}, we propose a backtracking scheme to approximate the Lipschitz constants \emph{locally} and incorporate it within the APD framework as shown in Algorithm~\ref{alg:APDB}, which we call {APDB}.}
\sa{To check whether the step-sizes chosen at each iteration $k\geq 0$ \rev{are} in accordance with \emph{local} Lipschitz constants, we define a test function $E_k(\cdot,\cdot)$ that employs linearization of $\Phi$ with respect to both $x$ and $y$.} 

\sa{For some free parameter sequence $\{\alpha_k,\beta_k\}_{k\geq 0}\subseteq \mathbb{R}_{+}$, we define}\vspace*{-3mm}
\sa{\small
\begin{align}\label{eq:Ek}
E_k(x,y)\triangleq &\Phi(x,y)-\Phi(x_k,y)-\fprod{\grad_x\Phi(x_k,y),x-x_k}-\frac{1}{\tau_k}\bD_\cX(x,x_k)\nonumber\\
&+\tfrac{1}{2\alpha_{k+1}}\norm{\grad_y\Phi(x,y)-\grad_y\Phi(x_k,y)}^2+\tfrac{1}{2\beta_{k+1}}\norm{\grad_y\Phi(x_k,y)-\grad_y\Phi(x_k,y_k)}^2 \nonumber\\
&-\Big(\frac{1}{\sigma_k}-\theta_k(\alpha_k+\beta_k)\Big)\bD_\cY(y,y_k),
\end{align}}%
where we set $0^2/0=0$ which may arise when $L_{yy}=0$ and $\beta_k=0$. \sa{For \rev{the} \emph{particular} $\alpha_k,\beta_k\geq 0$ and $\theta_k$ specified as in Algorithm~\ref{alg:APDB}, we get $E_k(x,y)\leq (L_{xx}+\frac{L_{yx}^2}{c_\alpha}\sigma_k-\frac{1}{\tau_k})\bD_\cX(x,x_k)+(\frac{L_{yy}^2}{c_\beta}\sigma_k+\frac{c_\alpha+c_\beta-1}{\sigma_k})\bD_\cY(y,y_k)$; hence, $E_k$ can be bounded by using the global Lipschitz constants, which prescribes how $\{\tau_k,\sigma_k\}$ should be chosen for APDB so that the test condition in Line~\ref{algeq:test_function} of Algorithm~\ref{alg:APDB} is satisfied.}

The rate statement and convergence result of the 
APDB method, displayed in Algorithm~\ref{alg:APDB}, are given in the next theorem.

\begin{theorem}\label{thm:backtrack} \sa{\bf (Main Result II)}
Suppose Assumption~\ref{assum} holds. Let $\delta\in[0,1)$, $c_\alpha>0$ and $c_\beta\geq 0$ are chosen as stated below, and define
{\small
\begin{align}
	\label{eq:tauhat_bound}
	\Psi_1\triangleq\frac{c_\alpha L_{xx}}{2\gamma_0L_{yx}^2}\zeta,\ \Psi_2\triangleq \frac{\sqrt{c_\beta(1-(c_\alpha+c_\beta+\delta))}}{\gamma_0L_{yy}},\ \zeta\triangleq -1+\sqrt{1+\frac{4(1-\delta)\gamma_0}{c_\alpha }\frac{L_{yx}^2}{L_{xx}^2}}.
\end{align}}%
For any given $(x_0,y_0)\in\dom f\times \dom h$ and $\bar{\tau},\gamma_0>0$, APDB, stated in Algorithm \ref{alg:APDB}, is well-defined, i.e., the number of inner iterations is finite and bounded by {$1+\log_{1/\eta}(\frac{\bar{\tau}}{\Psi})$} uniformly for $k\geq 0$ for some $\Psi>0$. Let $\{x_k,y_k\}_{k\geq 0}$ denote the iterate sequence generated by APDB, 
using the test function $E_k$ in \eqref{eq:Ek}. 
Then {for any $(x,y)\in\cX\times\cY$}, \eqref{eq:delta} holds for $\{t_k\}$ such that {$t_k=\sigma_k/\sigma_0$} for $k\geq 0$, where $(\bar{x}_K,\bar{y}_K)$ and $T_K$ are defined for $K\geq 1$ as in Theorem~\ref{thm:main}.\\
\textbf{(Part I.)} Suppose $\mu=0$.
\rev{Let} $c_\alpha+c_\beta+\delta\in(0,1)$ 
if $L_{yy}>0$; and $c_\beta=0$, $c_\alpha+\delta\in(0,1]$, otherwise. 
For this setting, {$\Psi=\Psi_1$ if $L_{yy}=0$ and $\Psi=\min\{\Psi_1,\Psi_2\}$ if $L_{yy}>0$}; moreover, $T_K 
=\Omega(K)$, implying $\cO(1/K)$ sublinear rate for \eqref{eq:delta}.
Moreover, if a saddle point for \eqref{eq:original-problem} exists and
{$\delta>0$ is chosen}, then 
$\{(x_k,y_k)\}_{k\geq 0}$ converges to a saddle point $(x^*,y^*)$ such that \eqref{eq:xy-bound-I} holds.\\
\textbf{(Part II.)} Suppose $\mu>0$ and $L_{yy}=0$. 
Let $\norm{x}_\cX=\sqrt{\fprod{x,x}}$, 
and $\bD_{\cX}(x,\bar{x})=\frac{1}{2}\norm{x-\bar{x}}_\cX^2$.
Assume $c_\alpha+\delta\in(0,1]$ and $c_\beta=0$.
For this setting, {$\Psi=\Psi_1$} and $T_K=\Omega(K^2)$, implying $\cO(1/K^2)$ sublinear rate for \eqref{eq:delta}. If a saddle point 
for \eqref{eq:original-problem}
exists,\footnote{\rev{Since $\mu>0$, all saddle points share the same unique $x$-coordinate, say $x^*\in\cX$.}} then $\{x_k\}$ converges to $x^*$ and $\{y_k\}$ has a limit point. Moreover, if $\delta>0$, then any limit point $(x^*,y^*)$ is a saddle point satisfying $0\leq \cL(\bar{x}_K,y^*)-\cL(x^*,\bar{y}_K) \leq \cO(1/K^2)$ for $K\geq 1$, and \eqref{eq:xy-bound-II} holds with $\gamma_K=\Omega(K^2)$.
\end{theorem}
\section{Methodology}
\label{sec:methodology}
\vspace*{-2mm}
\sa{The most generic form of our method, {GAPD}, is presented in Algorithm~\ref{alg:GAPD} which takes step-size sequences as input and repeatedly calls for the subroutine \textbf{MainStep} stated in Algorithm~\ref{alg:mainstep}.} In this section, we provide a general result \sa{in Theorem~\ref{thm:general-bound} for GAPD} 
unifying \sa{the analyses of merely and strongly {convex} cases described in Theorems~\ref{thm:main} and~\ref{thm:backtrack}}. An easy-to-read convergence analysis is given at the end of this section.
In our analysis, \sa{to show the general result in Theorem~\ref{thm:general-bound}}, we assume some conditions on $\{(\tau_k,\sigma_k,\theta_k)\}_{k\geq 0}$, stated in Assumption~\ref{assum:step}; \rev{furthermore, we also discuss in this section that when the Lipschitz constants are known, one can replace Assumption~\ref{assum:step} with Assumption~\ref{assum:step-2}, as it implies Assumption~\ref{assum:step}.} \sa{Later we show that step-size sequence $\{(\tau_k,\sigma_k,\theta_k)\}_k$ 
generated by APD and APDB, i.e., Algorithms~\ref{alg:APD} and~\ref{alg:APDB}, satisfies these conditions in \rev{Assumptions~\ref{assum:step-2} and~\ref{assum:step}}, respectively.} We define $0^2/0=0$ which may arise when $L_{yy}=0$.

\sa{To make the notation tractable, we define some quantities now:} for $k\geq 0$, let $q_k\triangleq \grad_y\Phi(x_k,y_k)-\grad_y\Phi(x_{k-1},y_{k-1})$ 
and $s_k \sa{\triangleq} \grad_y\Phi(x_k,y_k)+\theta_kq_k$ -- \sa{see Line~\ref{algeq:s} of MainStep in Algorithm~\ref{alg:mainstep} and Line~\ref{algeq:mainstep-GAPD} of GAPD in Algorithm~\ref{alg:GAPD}}.
\sa{
\begin{assumption}
\label{assum:step}
({\bf Step-size Condition I}) 
There exists $\{\tau_k,\sigma_k,\theta_k\}_{k\geq 0}$ such that $\{(x_k,y_k)\}_{k\geq 0}$ generated by GAPD, displayed in Algorithm~\ref{alg:GAPD}, and the step-size sequence together satisfy the following conditions 
for $k\geq 0$:
\begin{subequations}\label{eq:step-size-condition}
{\small
\begin{align}
& E_k(x_{k+1},y_{k+1})\leq \sa{-\delta \Big[\bD_\cX(x_{k+1},x_k)/\tau_k+\bD_\cY(y_{k+1},y_k)/\sigma_k\Big]}, \label{eq:step-size-condition-Ek}\\
& {t_k\big(\frac{1}{\tau_k}+\mu\big)\geq \frac{t_{k+1}}{\tau_{k+1}}},\quad {\frac{t_k}{\sigma_k}\geq \frac{t_{k+1}}{\sigma_{k+1}}},\quad \frac{t_k}{t_{k+1}}=\theta_{k+1}, \label{eq:step-size-condition-theta}
\end{align}}%
\end{subequations}
{for some positive $\{t_k,~\alpha_k\}_{k\geq 0}$ such that $t_0=1$, nonnegative $\{\beta_k\}_{k\geq 0}$} and $\delta\in[0,1)$, where $E_k(\cdot,\cdot)$ is defined in~\eqref{eq:Ek} using $\{\alpha_k,\beta_k,\theta_k\}$ as above. 
\end{assumption}}%
\sa{
\begin{assumption}
\label{assum:step-2}
({\bf Step-size Condition II}) 
For any $k\geq 0$, the step-sizes $\tau_k,\sigma_k$ and 
momentum parameter $\theta_k$ satisfy $\theta_0=1$, \eqref{eq:step-size-condition-theta} and
{\small
\begin{align}
&\frac{1-\delta}{\tau_k}\geq L_{xx}+\frac{{L^2_{yx}}}{\alpha_{k+1}},\quad\ \frac{1-\delta}{\sigma_k} \geq \theta_k(\alpha_k+\beta_k)+\frac{L^2_{yy}}{\beta_{k+1}}, \label{eq:step-size-condition-pd}
\end{align}}%
{for some positive $\{t_k,~\alpha_k\}_{k\geq 0}$ such that $t_0=1$, nonnegative $\{\beta_k\}_{k\geq 0}$, and $\delta\in[0,1)$}. 
\end{assumption}}%
\begin{algorithm}[h!]
	\caption{\sa{Generic Accelerated Primal-Dual algorithm~(GAPD)}}
	\label{alg:GAPD}
	\begin{algorithmic}[1]
		\STATE{\bfseries Input: 
			$\{\tau_k,\sigma_k,\theta_k\}_{k\geq 0}$, $(x_0,y_0)\in\cX\times\cY$}
		\STATE{$(x_{-1},y_{-1})\gets(x_0,y_0)$}
		\FOR{$k\geq 0$}
		\STATE $(x_{k+1},y_{k+1})\gets\hbox{\textbf{MainStep}}(x_k,y_k,x_{k-1},y_{k-1},\tau_k,\sigma_k,\theta_k)$ \label{algeq:mainstep-GAPD}
		\ENDFOR
	\end{algorithmic}
\end{algorithm}
\subsection{Auxiliary Results} \sa{In this section, we investigate some sufficient conditions on step-size sequence $\{\tau_k,\sigma_k,\theta_k\}$ and the parameter sequence $\{\alpha_k,\beta_k,t_k\}$ that can guarantee some desirable convergence properties for GAPD. Both APD and APDB, {with the iterate and the step-size sequences generated 
as in Algorithms~\ref{alg:APD} and~\ref{alg:APDB}, respectively,} are particular cases of the GAPD algorithm; hence, we later establish our main results in Sections~\ref{sec:main_thm_proof} and~\ref{sec:main_thm_proof_backtrack} using the results of this section.}
\begin{theorem}\label{thm:general-bound}
Suppose Assumption~\ref{assum} holds, and
$\{x_k,y_k\}_{k\geq 0}$ is 
generated by GAPD stated in Algorithm~\ref{alg:GAPD}
using a parameter sequence $\{\tau_k,\sigma_k,\theta_k\}_{k\geq 0}$ that satisfies Assumption~\ref{assum:step}.
Then for any $(x,y)\in\cX\times\cY$ and $K\geq 1$, 
\begin{eqnarray}\label{eq:rate}
\lefteqn{\cL(\bar{x}_{K},y)-\cL(x,\bar{y}_{K})\leq}  \\
& & 
\frac{1}{T_K}{\Delta(x,y)}-\frac{t_K}{T_K}\Big[\frac{1}{\tau_{K}}\bD_{\cX}(x,x_{K}) +\Big(\frac{1}{\sigma_K}-\theta_K(\alpha_K+\beta_K)\Big)\bD_{\cY}(y,y_{K}) \Big], \nonumber
\end{eqnarray}
where $\Delta(x,y)$ is defined in \eqref{eq:delta}, {$T_K$ and $(\bar{x}_K,\bar{y}_K)$ are defined in Theorem~\ref{thm:main}}.
\end{theorem}
\begin{proof}
For $k\geq0$, using Lemma \ref{lem_app:prox} in the appendix for the $y$- and $x$-subproblems in Algorithm~\ref{alg:GAPD} we get two inequalities that hold for any $y\in\cY$ and $x\in\cX$:
\begin{align}
&h(y_{k+1})-\fprod{s_k,~ y_{k+1}-y}\label{eq:sc-h}\\
&\qquad\leq h(y)
+\frac{1}{\sigma_k}\Big[\bD_{\cY}(y,y_k)-{\bD_{\cY}(y,y_{k+1})}-\bD_{\cY}(y_{k+1},y_k)\Big], \nonumber
\end{align}
\vspace{-5mm}
\begin{align}
&f(x_{k+1})+\fprod{\grad_{{x}}\Phi(x_k,y_{k+1}),x_{k+1}-x}+\frac{\mu}{2}\norm{x-x_{k+1}}_\cX^2 \label{eq:sc-f}\\
&\qquad\leq f(x)
+\frac{1}{\tau_k}\Big[\bD_{\cX}(x,x_k)-{\bD_{\cX}(x,x_{k+1})}-\bD_{\cX}(x_{k+1},x_k)\Big]. \nonumber
\end{align}

For all $k\geq 0$, let $A_{k+1}\triangleq \frac{1}{\sigma_k}\bD_{\cY}(y,y_k)-\frac{1}{\sigma_k}{\bD_{\cY}(y,y_{k+1})} -\frac{1}{\sigma_k}\bD_{\cY}(y_{k+1},y_k)$ and $B_{k+1}\triangleq \frac{1}{\tau_k}\bD_{\cX}(x,x_k)-\frac{1}{\tau_k}{\bD_{\cX}(x,x_{k+1})}-\frac{1}{\tau_k}\bD_{\cX}(x_{k+1},x_k)-{\frac{\mu}{2}\norm{x-x_{k+1}}_\cX^2}$.
The inner product in \eqref{eq:sc-f} can be lower bounded using convexity of $\Phi(x,y_{k+1})$ in $x$
as follows:
\begin{align*}
\fprod{\grad_{{x}}\Phi(x_k,y_{k+1}),x_{k+1}-x}=&\fprod{\grad_{{x}}\Phi(x_k,y_{k+1}),x_k-x}+\fprod{\grad_{{x}}\Phi(x_k,y_{k+1}),x_{k+1}-x_k}\\
\geq& \Phi(x_k,y_{k+1})-\Phi(x,y_{k+1})+\fprod{\grad_{{x}}\Phi(x_k,y_{k+1}),x_{k+1}-x_k}.
\end{align*}
Using this inequality after adding $\Phi(x_{k+1},y_{k+1})$ to both sides of \eqref{eq:sc-f}, we get
\begin{align}\label{eq:convex-Lip-f}
f(x_{k+1})+\Phi(x_{k+1},y_{k+1})\leq& f(x)+\Phi(x,y_{k+1})+B_{k+1}+\sa{\Lambda_k},
\end{align}%
where $\Lambda_k \triangleq \Phi(x_{k+1},y_{k+1})-\Phi(x_k,y_{k+1})-\fprod{\grad_{{x}}\Phi(x_k,y_{k+1}),x_{k+1}-x_k}$ for $k\geq 0$. Now, \sa{for $k\geq 0$,} summing \eqref{eq:sc-h} and \eqref{eq:convex-Lip-f} and rearranging the terms lead to
\begin{align}\label{eq:concave}
&\cL(x_{k+1},y)-\cL(x,y_{k+1})  \\
& = f(x_{k+1})+\Phi(x_{k+1},y)-h(y)-f(x)-\Phi(x,y_{k+1})+h(y_{k+1})\nonumber\\
& \leq \Phi(x_{k+1},y)-\Phi(x_{k+1},y_{k+1}) +\fprod{s_k,y_{k+1}-y}+{\Lambda_k}+A_{k+1}+B_{k+1} \nonumber \\
& \leq-\fprod{q_{k+1},y_{k+1}-y}+\theta_k\fprod{q_k,y_{k+1}-y} +{\Lambda_k}+A_{k+1}+B_{k+1}, \nonumber
\end{align}
where in the last inequality we use the concavity of \sa{$\Phi(x_{k+1},\cdot)$}. \sa{To obtain a telescoping sum later,} we can rewrite the bound in \eqref{eq:concave} as \vspace*{-2mm}
\begin{align}\label{eq:rearrange}
&{\cL(x_{k+1},y)-\cL(x,y_{k+1})\leq}  \\
& \Big[\frac{1}{\tau_k}\bD_{\cX}(x,x_k)+\frac{1}{\sigma_k}\bD_{\cY}(y,y_k)+\theta_k\fprod{q_k,y_k-y}\Big]  \nonumber \\
& -\Big[\frac{1}{\tau_{k}}\bD_{\cX}(x,x_{k+1})+\frac{\mu}{2}\norm{x-x_{k+1}}_\cX^2+\frac{1}{\sigma_{k}}\bD_{\cY}(y,y_{k+1})+\fprod{q_{k+1},y_{k+1}-y}\Big] \nonumber \\
& +\Lambda_k-\frac{1}{\tau_{k}}\bD_{\cX}(x_{k+1},x_k) -\frac{1}{\sigma_k}\bD_{\cY}(y_{k+1},y_k)+\theta_k \fprod{q_k,y_{k+1}-y_k}.\nonumber
\end{align}%

\vspace*{-1mm}
\noindent Next, we bound the term \sa{$\fprod{q_k,y_{k+1}-y_k}$} in \eqref{eq:rearrange}. Indeed, one can bound $\fprod{q_k,y-y_k}$ for any given $y\in\cY$ as follows. Let $p_k^x\triangleq \grad_y\Phi(x_{k},y_k)-\grad_y\Phi(x_{k-1},y_k)$ and $p_k^y\triangleq \grad_y\Phi(x_{k-1},y_k)-\grad_y\Phi(x_{k-1},y_{k-1})$ which immediately implies that $q_k=p_k^x+p_k^y$. Moreover, for any $y\in\cY$, $y'\in\cY^*$, and \sa{$a>0$, we have $|\fprod{y',y}|\leq \frac{a}{2}\norm{y}_\cY^2+\frac{1}{2a}\norm{y'}_{\cY^*}^2$}. Hence, using this inequality twice, once for $\fprod{p^x_k,y-y_k}$ and once for $\fprod{p^y_k, y-y_k}$, and the fact that $\bD_{\cY}(y,\bar{y})\geq\frac{1}{2}\norm{y-\bar{y}}_\cY^2$, we obtain for all $k\geq 0$ that
\begin{align}
\label{eq:q_k}
|\fprod{q_k,y-y_k}|\leq &\alpha_k\bD_\cY(y,y_k)+\frac{1}{2\alpha_k}\norm{p_k^x}_{\cY^*}^2 
+\beta_k\bD_\cY(y,y_k)+\frac{1}{2\beta_k}\norm{p_k^y}_{\cY^*}^2, 
\end{align}
which holds for any $\alpha_k,\beta_k>0$. Moreover, if $L_{yy}=0$, then $\norm{p_k^y}_{\cY^*}=0$; hence, $|\fprod{q_k,y-y_k}|\leq \alpha_k\bD_{\cY}(y,y_k)+\frac{1}{2\alpha_k}\norm{p_k^x}_{\cY^*}^2$ for any $\alpha_k>0$. Since we define $0^2/0=0$, \eqref{eq:q_k} holds for any $\alpha_k>0$ and $\beta_k=0$ when $L_{yy}=0$.
Therefore, using \eqref{eq:q_k} within \eqref{eq:rearrange} \sa{with $\{\alpha_k,\beta_k\}$ satisfying Assumption~\ref{assum:step},} we get 
for $k\geq 0$,
\begin{subequations}\label{eq:bound-lagrange}
\begin{align}
&{\cL(x_{k+1},y)-\cL(x,y_{k+1})\leq} Q_k(z) - R_{k+1}(z) + E_k, \label{eq:bound-lagrange-main}\\
& Q_k(z) \triangleq \frac{1}{\tau_k}\bD_{\cX}(x,x_k)+\frac{1}{\sigma_k}\bD_{\cY}(y,y_k)+\theta_k\fprod{q_k,y_k-y} \label{eq:bound-lagrange-Q}\\
&\qquad \qquad \quad +\frac{\theta_k}{2\alpha_k}\norm{p_k^x}_{\cY^*}^2+\frac{\theta_k}{2\beta_k}\norm{p_k^y}_{\cY^*}^2, \nonumber\\
& R_{k+1}(z)\triangleq \frac{1}{\tau_{k}}\bD_{\cX}(x,x_{k+1})+\frac{\mu}{2}\norm{x-x_{k+1}}_\cX^2+\frac{1}{\sigma_{k}}\bD_{\cY}(y,y_{k+1}) \label{eq:bound-lagrange-R}\\
&\qquad \qquad \quad+\fprod{q_{k+1},y_{k+1}-y} +\frac{1}{2\alpha_{k+1}}\norm{p_{k+1}^x}_{\cY^*}^2+\frac{1}{2\beta_{k+1}}\norm{p_{k+1}^y}_{\cY^*}^2,\nonumber \\
& \sa{E_k\triangleq}
\Lambda_k +\frac{1}{2\alpha_{k+1}}\norm{p_{k+1}^x}_{\cY^*}^2 -\frac{1}{\tau_{k}}\bD_{\cX}(x_{k+1},x_k)\nonumber\\
&\qquad \qquad \quad  +\frac{1}{2\beta_{k+1}}\norm{p_{k+1}^y}_{\cY^*}^2-\Big(\frac{1}{\sigma_k}-\theta_k(\alpha_k+\beta_k)\Big)\bD_{\cY}(y_{k+1},y_k). \nonumber
\end{align}
\end{subequations}
\sa{Note that $E_k=E_k( x_{k+1}, y_{k+1})$ where $E_k(\cdot,\cdot)$ is defined in \eqref{eq:Ek} for $k\geq 0$.}
All the derivations until here, including \eqref{eq:bound-lagrange}, hold for any Bregman distance function $\bD_\cX$. Recall that if $\mu>0$, then we set $\bD_{\cX}(x,\bar{x})=\frac{1}{2}\norm{x-\bar{x}}_\cX^2$ and $\norm{x}_\cX=\sqrt{\fprod{x,x}}$. Now, multiplying both sides by $t_k>0$, summing over $k=0$ to $K-1$, and then using Jensens's inequality\footnote{For any $x\in\cX$ and $y\in\cY$, $\cL(\cdot,y)-\cL(x,\cdot)$ is convex.}, we obtain \vspace*{-4mm}
{\small
\begin{align}\label{eq:sum-bound}
T_K(\cL(\bar{x}_{K},y)-\cL(x,\bar{y}_{K}))&\leq\sum_{k=0}^{K-1}t_k\big(Q_k(z)-R_{k+1}(z)+E_k\big)\\
&\leq t_0Q_0(z)-t_{K-1}R_{K}(z)+\sum_{k=0}^{K-1}t_kE_k,\nonumber
\end{align}}%
where $T_K=\sum_{k=0}^{K-1}t_k$ and the last inequality follows from the step-size conditions in~\eqref{eq:step-size-condition-theta}, which imply that $t_{k+1}Q_{k+1}(z)-t_k R_{k+1}(z)\leq 0$ for $k=0$ to $K-2$.

\sa{According to Assumption~\ref{assum:step}, $\tau_k,\sigma_k$ and $\theta_k$ are chosen such that ${E}_k\leq 0$ for $k=0,\hdots, K-1$, then} 
\eqref{eq:sum-bound} implies that
{\small
\begin{eqnarray}\label{eq:bound-backtracking}
\lefteqn{T_K(\cL(\bar{x}_{K},y)-\cL(x,\bar{y}_{K}))\leq t_0Q_0(z)-t_{K-1}R_{K}(z)} \\
&&\leq \frac{t_0}{\tau_0}\bD_{\cX}(x,x_0)+\frac{t_0}{\sigma_0}\bD_{\cY}(y,y_0)+{t_K\theta_K}\fprod{q_{K},y-y_{K}}
\nonumber\\
&& \mbox{ }\ -t_K\Big[\frac{1}{\tau_{K}}\bD_{\cX}(x,x_{K})+\frac{1}{\sigma_{K}}\bD_{\cY}(y,y_{K})+\frac{\theta_K}{2\alpha_K}\norm{p_K^x}_{\cY^*}^2 + \frac{\theta_K}{2\beta_K}\norm{p_K^y}_{\cY^*}^2\Big],\nonumber
\end{eqnarray}}%
where \sa{in the last inequality we used $t_KQ_K(z)\leq t_{K-1}R_{K}(z)$ and 
$q_0=p_0^x=p_0^y=\mathbf{0}$ (which holds due to the initialization $x_0=x_{-1}$ and $y_0=y_{-1}$).} 
One can upper bound \sa{$\fprod{q_{K},y-y_{K}}$} in~\eqref{eq:bound-backtracking} 
  using \eqref{eq:q_k} for $k=K$. After plugging this bound in~\eqref{eq:bound-backtracking} and dividing both sides by $T_K$, we obtain the desired result in~\eqref{eq:rate}.
\end{proof}
Lemma~\ref{lem:supermartingale} 
is used to establish the convergence of the primal-dual iterate sequence.
\begin{lemma}~\cite{Robbins71}\label{lem:supermartingale}
{Let $\{a_k\}$, $\{b_k\}$, and $\{c_k\}$ be non-negative real sequences such that $a_{k+1}\leq a_k - b_k + c_k$ for all $k\geq 0$, and $\sum_{k=0}^\infty c_k<\infty$. Then $a=\lim_{k\rightarrow \infty}a_k$ exists, and $\sum_{k=0}^\infty b_k<\infty$.}
\end{lemma}\vspace*{-4mm}
\sa{\begin{theorem}\label{thm:convergence}
Suppose a saddle point for \eqref{eq:original-problem} exists, and Assumption~\ref{assum:step} holds for some $\delta>0$ and $\{\alpha_k,\beta_k,t_k\}$ such that $\inf_{k\geq 0}t_k\min\{\frac{1}{\tau_k},~\frac{1}{\sigma_k}-\theta_k(\alpha_k+\beta_k)\}\geq\delta'$ for some $\delta'>0$.
\begin{enumerate}[(i)]
\item \textbf{Case 1: $\lim_{k\to \infty}\min\{\tau_k,~\sigma_k\}>0$}. Then any limit point of $\{(x_k,y_k)\}_{k\geq 0}$ is a saddle point. In addition, suppose \rev{$\liminf_{k\to \infty}\min\{\alpha_k,\beta_k\}>0$} when $L_{yy}>0$, and \rev{$\liminf_{k\to \infty}\alpha_k>0$} when $L_{yy}=0$, if $\inf_{k\geq 0}t_k>0$ and $\sup_{k\geq 0}t_k<\infty$ hold, then $\{(x_k,y_k)\}$ has a unique limit point.
\item \textbf{Case 2: $\tau_k\to 0$ and $\lim_{k\to\infty}\sigma_k>0$.}\footnote{Similar conditions can also be given for the case $\max\{\tau_k,\sigma_k\}\to 0$ or for the case $\sigma_k\to 0$ and $\inf_k \tau_k>0$; however, the two cases considered in Theorem~\ref{thm:general-bound} are sufficient to analyze APD and APDB, displayed in Algorithm~\ref{alg:APD} and Algorithm~\ref{alg:APDB}, respectively.} If $\varphi_{\cX}$ defining $\bD_{\cX}$ is Lipschitz differentiable {and $t_k=\Omega(\frac{1}{\tau_k})$,} then any limit point of $\{(x_k,y_k)\}_{k\geq 0}$ is a saddle point.
\end{enumerate}
\end{theorem}}%
\begin{proof}
Suppose $z^\#=(x^\#,y^\#)$ is a saddle point $\cL$ in~\eqref{eq:original-problem}. Let $x=x^\#$ and $y=y^\#$ in \eqref{eq:bound-lagrange}, \sa{then \eqref{eq:bound-lagrange-Q}, \eqref{eq:bound-lagrange-R} and the
assumption on $\delta'>0$ imply that}\vspace*{-2mm}
{\small
\begin{subequations}
\label{eq:Qk}
\begin{align}
\sa{t_k}Q_k(z^\#) &\geq \frac{\sa{t_k}}{\tau_k}\bD_\cX(x^\#,x_k)+\sa{t_k}(\frac{1}{\sigma_k}-\theta_k(\alpha_k+\beta_k))\bD_\cY(y^\#,y_k)\nonumber \\
    &\geq \delta'\bD_\cX(x^\#,x_k)+\delta'\bD_\cY(y^\#,y_k)> 0, \label{Qk-b}\\
\sa{t_k} R_{k+1}(z^\#) &\geq \sa{t_{k+1}Q_{k+1}(z^\#).} \label{Qk-a}
\end{align}
\end{subequations}}%
Multiplying the inequality \eqref{eq:bound-lagrange-main} by $t_k$, then using $\cL({x}_{k+1},y^\#)-\cL(x^\#,{y}_{k+1})\geq 0$, and \eqref{Qk-a}, the following can be obtained
{
\begin{align}\label{eq:bound-iterates}
0\leq &t_kQ_k(z^\#)-t_{k+1}Q_{k+1}(z^\#)+t_kE_k({x}_{k+1},{y}_{k+1})\\
\leq & t_kQ_k(z^\#)-t_{k+1}Q_{k+1}(z^\#)-\delta t_k[\bD_\cX(x_{k+1},x_k)/\sa{\tau_k}+\bD_\cY(y_{k+1},y_k)/\sigma_k]\nonumber.
\end{align}}%
Let $a_k=t_kQ_k(z^\#)$, \sa{$b_k=\delta t_k[\bD_\cX(x_{k+1},x_k)/\sa{\tau_k}+\bD_\cY(y_{k+1},y_k)/\sigma_k]$}, and $c_k=0$ for $k\geq 0$, then Lemma~\ref{lem:supermartingale} implies that $a\triangleq \lim_{k\rightarrow \infty} a_k$ exist. Therefore, \eqref{Qk-b} implies that $\{z_k\}$ is a bounded sequence, where $z_k\triangleq(x_k,y_k)$; hence, it has a convergent subsequence $z_{k_n}\rightarrow z^*$ as $n\rightarrow \infty$ for some $z^*\in\cX\times\cY$ where $z^*=(x^*,y^*)$. \sa{Note that since $t_k,\theta_k,\alpha_k,\beta_k\geq 0$ for $k\geq 0$, we have $\inf_{k\geq 0}\min\{t_k/\tau_k,~t_k/\sigma_k\}\geq \delta'>0$. Moreover, Lemma~\ref{lem:supermartingale} also implies that $\sum_{k=0}^\infty b_k <\infty$; hence, we also have $\sum_{k\geq 0}\norm{z_{k+1}-z_k}^2<\infty$. Thus,} for any $\epsilon>0$ there exists $N_1$ such that for any $n\geq N_1$, $\max\{\norm{z_{k_n}-z_{k_n-1}},~\norm{z_{k_n}-z_{k_n+1}}\}< \frac{\epsilon}{2}$. Convergence of $\{z_{k_n}\}$ sequence also implies that there exists $N_2$ such that for any $n\geq N_2$, $\norm{z_{k_n}-z^*}< \frac{\epsilon}{2}$. Therefore, letting $N\triangleq \max\{N_1,N_2\}$, we conclude 
$\norm{z_{k_n\pm 1}-z^*}< \epsilon$, i.e., $z_{k_n\pm 1} \rightarrow z^*$ as $n\rightarrow \infty$.

Now we show that $z^*$ is indeed
a saddle point of \eqref{eq:original-problem} by considering the optimality conditions for the 
$y$- and $x$-subproblems 
of Algorithm~\ref{alg:GAPD}, i.e., Lines~\ref{algeq:y-main} and~\ref{algeq:x-main} of \textbf{MainStep}. 
For all $n\in\mathbb{Z}_+$, let $u_n\triangleq \left(\grad\sa{\varphi_{\cX}}(x_{k_n})-\grad\sa{\varphi_{\cX}}(x_{k_n+1})\right)/\tau_{k_n}-\grad_x\Phi(x_{k_n},y_{k_n+1})$ and $v_n\triangleq s_{k_n}+\left(\grad\sa{\varphi_{\cY}}(y_{k_n})-\grad\sa{\varphi_{\cY}}(y_{k_n+1})\right)/\sigma_{k_n}$; one has $u_n\in \partial f(x_{k_n+1})$ and $v_n\in \partial h(y_{k_{n}+1})$ for $n\geq 0$. Since \sa{$\grad\varphi_{\cX}$ and $\grad\varphi_{\cY}$ are continuously differentiable on $\dom f$ and $\dom h$, respectively, whenever $\lim_{k\to \infty}\min\{\tau_k,~\sigma_k\}>0$,} it follows from Theorem 24.4 in \cite{rockafellar2015convex} that $\partial f(x^*) \ni \lim_{n\rightarrow \infty}u_n=-\grad_x\Phi(x^*,y^*)$, $\partial h(y^*) \ni \lim_{n\rightarrow \infty}v_n={\grad_y\Phi(x^*,y^*)}$, which implies that $z^*$ is a saddle point of \eqref{eq:original-problem}. In addition, if $\sup_{k\geq 0}t_k<\infty$, we show that $z^*$ is the unique limit point. Since \eqref{eq:Qk} and \eqref{eq:bound-iterates} are true for any saddle point $z^\#$, letting $z^\#=z^*$ and invoking Lemma~\ref{lem:supermartingale} again, one can conclude that $w^*=\lim_{k\rightarrow \infty} w_k \geq 0$ exists, where $w_k\triangleq t_kQ_k(z^*)$. \sa{Because $\{t_k\}\subset\reals_+$ is a bounded sequence, $\{\theta_k\}\subset\reals_+$ is also a bounded sequence. Therefore, using $\lim_{k\to \infty}\min\{\tau_k,\sigma_k\}>0$ {together with $\liminf_{k\to \infty}\min\{\alpha_k,\beta_k\}>0$} when $L_{yy}>0$, it follows from $z_{k_n}\rightarrow z^*$ and $z_{k_n-1}\rightarrow z^*$ that we have} $\lim_{n\rightarrow \infty} w_{k_n}=0$;
henceforth, $w^*=\lim_{k\rightarrow \infty} w_k=\lim_{n\rightarrow \infty} w_{k_n}=0$, and \eqref{Qk-b} evaluated at $z^*$ implies that $z_k\rightarrow z^*$. \sa{For $L_{yy}=0$, $\lim_{k\to \infty}\min\{\tau_k,\sigma_k\}>0$ and {$\liminf_{k\to \infty}\alpha_k>0$} is enough since $\norm{p_k^y}=0$ for $k\geq 0$.}

\sa{On the other hand, for the case $\tau_k\to 0$ and $\lim_k \sigma_k>0$, we require that $\grad\varphi_{\cX}$ is Lipschitz on $\dom f$ with constant $L_{\varphi_{\cX}}$ and that there exist $C>0$ and $K\geq 1$ such that $t_k\geq C\frac{1}{\tau_k}$ for $k\geq K$. Clearly, $\norm{\grad\sa{\varphi_{\cX}}(x)-\grad\sa{\varphi_{\cX}}(x')}_{\cX^*}\leq \sqrt{2L_{\varphi_{\cX}}^2D_{\cX}(x',x)}$ for $x,x'\in\dom f$; hence, to argue that $\lim_{n\rightarrow \infty}u_n=-\grad_x\Phi(x^*,y^*)$, one needs {$\lim_n \bD_{\cX}(x_{k_n+1},x_{k_n})/{\tau_{k_n}^2}=0$}. Indeed, since $\sum_k b_k<\infty$, $t_k\bD_\cX(x_{k+1},x_k)/\sa{\tau_k}\to 0$ holds; therefore, $\bD_\cX(x_{k+1},x_k)/\tau_k^2\to 0$ as $t_k\geq C\frac{1}{\tau_k}$ for $k\geq K$, which implies $\lim_{n\rightarrow \infty}u_n=-\grad_x\Phi(x^*,y^*)$. Finally, $\lim_{n\rightarrow \infty}v_n={\grad_y\Phi(x^*,y^*)}$ as discussed above for the previous case since $\lim_k \sigma_k>0$ and $\grad\varphi_{\cY}$ is continuous. Again invoking \cite[Theorem 24.4]{rockafellar2015convex}, one can establish that $z^*$ is a saddle point of \eqref{eq:original-problem}.}
\end{proof}

\sa{Next, we provide some useful results on $\{\tau_k,\sigma_k,\theta_k\}_k$ of the algorithms APD and APDB which will 
help us derive the rate results in Theorems~\ref{thm:main} and~\ref{thm:backtrack}. }
\begin{lemma}
\label{lem:stronger-step-condition}
\sa{Suppose the sequence $\{\tau_k,\sigma_k,\theta_k\}_{k\geq 0}$ satisfy 
\eqref{eq:step-size-condition-pd} for some positive $\{t_k,~\alpha_k\}_{k\geq 0}$, nonnegative $\{\beta_k\}_{k\geq 0}$, and $\delta\in[0,1)$. Let $\{x_k,y_k\}$ be the GAPD iterate sequence corresponding to $\{\tau_k,\sigma_k,\theta_k\}$. Then $\{x_k,y_k\}$ and $\{\tau_k,\sigma_k,\theta_k\}$ satisfy 
\eqref{eq:step-size-condition-Ek} with the same $\{t_k,\alpha_k,\beta_k\}$ and $\delta$.}
\end{lemma}
\begin{proof}
From Assumption \ref{assum}, 
for any $k\geq 0$ and $(x,y)\in\cX\times\cY$, we have
		\begin{align*}
		&\Phi(x,y)-\Phi(x_k,y)-\fprod{\grad_x\Phi(x_k,y),x-x_k}\leq L_{xx}\norm{x-x_k}_{\cX}^2/2\leq L_{xx}\bD_{\cX}(x,x_k),\\
		&\tfrac{1}{2}\norm{\grad_y\Phi(x,y)-\grad_y\Phi(x_k,y)}_{\cY^*}^2\leq L_{yx}^2\norm{x-x_k}_{\cY}^2/2\leq L_{yx}^2\bD_{\cX}(x,x_k) , \\
		&\tfrac{1}{2}\norm{\grad_y\Phi(x_k,y)-\grad_y\Phi(x_k,y_k)}_{\cY^*}^2\leq L_{yy}^2\norm{y-y_k}_{\cY}^2/2\leq L_{yy}^2\bD_{\cY}(y,y_k).
		\end{align*}
		Above inequalities evaluated at $(x,y)=(x_{k+1},y_{k+1})$ imply that
		\begin{align}\label{eq:upper-bound-E}
		E_k\triangleq E_k(x_{k+1},y_{k+1})\leq&~ (L_{yx}^2/\alpha_{k+1}+L_{xx}-1/\tau_k)\bD_\cX(x_{k+1},x_k) \nonumber\\
		&+(L_{yy}^2/\beta_{k+1}+\theta_k(\alpha_k+\beta_k)-1/\sigma_k)\bD_\cY(y_{k+1},y_k) \nonumber \\
		&\leq  -\delta[\tfrac{1}{\tau_k}\bD_\cX(x_{k+1},x_k)+\tfrac{1}{\sigma_k}\bD_\cY(y_{k+1},y_k)],
		\end{align}
		where in the last inequality we used \eqref{eq:step-size-condition-pd} and non-negativity of Bregman functions.
\end{proof}
\begin{lemma}
\label{lem:step-size-seq}
\sa{Given {$\{\tau_k\}_{k\geq 0}\subset\reals_{++}$ and $\bar{\tau},\gamma_0>0$, let $\sigma_{-1}=\gamma_0\bar{\tau}$} and  $\sigma_k=\gamma_k \tau_k$, $\theta_k=\sigma_{k-1}/\sigma_k$ and $\gamma_{k+1}=\gamma_k (1+\mu \tau_k)$ for $k\geq 0$. $\{\tau_k,\sigma_k,\theta_k\}$ satisfies \eqref{eq:step-size-condition-theta} for $\{t_k\}$ such that $t_k=\sigma_k/\sigma_0$ for $k\geq 0$.}
\end{lemma}
\begin{proof}
Since $t_k=\sigma_k/\sigma_0$ and $\tau_k>0$ for $k\geq 0$, \eqref{eq:step-size-condition-theta} can be written as $(1+\mu\tau_k)\geq \frac{\sigma_{k+1}\tau_k}{\sigma_k\tau_{k+1}}$ and $\frac{\sigma_k}{\sigma_{k+1}}=\theta_{k+1}$. The latter condition clearly holds due to our choice of $\theta_k$. Moreover, from $\sigma_k= \gamma_k\tau_k$ and $\gamma_{k+1}=\gamma_k(1+\mu\tau_k)$, we conclude that the former condition holds with equality by observing that $(1+\mu\tau_k)=\frac{\gamma_{k+1}}{\gamma_k}=\frac{\sigma_{k+1}\tau_k}{\sigma_k\tau_{k+1}}$.
\end{proof}
\begin{lemma}\label{lem:parameter}
\sa{Under Assumption~\ref{assum}, consider APDB displayed in Algorithm~\ref{alg:APDB} for any given $\delta\in[0,1)$ and $c_\alpha,c_\beta\geq 0$ such that $c_\alpha+c_\beta+\delta\leq 1$. When $L_{yy}>0$, set $c_\alpha,~c_\beta> 0$; otherwise, when $L_{yy}=0$, set $c_\alpha>0$ and $c_\beta=0$.
The APDB iterate and step-size sequences, i.e., $\{x_k,y_k\}_{k\geq 0}$ and $\{\tau_k,\sigma_k,\theta_k\}_{k\geq 0}$, are well-defined; more precisely, 
for any $k\geq 0$, the backtracking condition, i.e., $E_k(x_{k+1},y_{k+1})\leq -\delta[\tfrac{1}{\tau_k}\bD_\cX(x_{k+1},x_k)+\tfrac{1}{\sigma_k}\bD_\cY(y_{k+1},y_k)]$, holds after finite number of inner iterations.}

\sa{For $k\geq 0$, $\tau_k\geq\eta\hat{\tau}_k$ for some positive $\{\hat{\tau}_k\}$: when $L_{yy}=0$ and $\mu\geq 0$, $\hat{\tau}_k\geq \Psi_1\sqrt{\gamma_0/\gamma_k}$ for $k\geq 0$; on the other hand, when $L_{yy}>0$ and $\mu=0$,
$\hat{\tau}_k\geq \min\{\Psi_1,\Psi_2\}\sqrt{\gamma_0/\gamma_k}$ 
for $k\geq 0$, where $\Psi_1$ and $\Psi_2$ are defined in \eqref{eq:tauhat_bound}.\footnote{$\mu=0$ implies $\gamma_k=\gamma_0$ for $k\geq 0$, while $\mu>0$ implies $\gamma_{k+1}>\gamma_k$ for $k\geq 0$.}}
\end{lemma}
\begin{proof}
{Fix arbitrary $k\geq 0$. Note that since APDB is a special GAPD corresponding to a particular $\{\tau_k,\sigma_k,\theta_k\}$, Lemma~\ref{lem:stronger-step-condition} implies that whenever \eqref{eq:step-size-condition-pd} holds, \eqref{eq:step-size-condition-Ek} holds as well. Next, we will show that there exists $\hat{\tau}_k>0$ such that \eqref{eq:step-size-condition-pd} is true for all $\tau_k\in(0,\hat{\tau}_k]$. Since $\sigma_k=\gamma_k\tau_k$ and $\theta_k=\sigma_{k-1}/\sigma_k$, \eqref{eq:step-size-condition-pd} is equivalent to 
\begin{align}
\label{eq:pd-equiv}
0\geq -(1-\delta)+L_{xx}\tau_k+\frac{L_{yx}^2}{c_\alpha}\gamma_k\tau_k^2,\qquad 1-(\delta+c_\alpha+c_\beta)\geq\frac{L_{yy}^2}{c_\beta}\gamma_k^2\tau_k^2.
\end{align}}%
Suppose $L_{yy}>0$. Then \eqref{eq:pd-equiv} holds for all $\tau_k\in(0,\hat{\tau}_k]$, where
{\small
\begin{equation}
\label{eq:tauhat}
\hat{\tau}_k\triangleq\min\left\{\frac{-L_{xx}+\sqrt{L^2_{xx}+4(1-\delta)L^2_{yx}\gamma_k/c_\alpha}}{2L^2_{yx}\gamma_k/c_\alpha},\quad \frac{\sqrt{c_\beta(1-(c_\alpha+c_\beta+\delta))}}{\gamma_k L_{yy}}\right\}.
\end{equation}}%
On the other hand, when $L_{yy}=0$, the second inequality in~\eqref{eq:pd-equiv} always holds; hence, $\hat{\tau}_k$ is defined by the first term in \eqref{eq:tauhat}. Since in each step of backtracking, $\tau_k$ is decreased by a factor of $\eta\in(0,1)$, when the backtracking terminates, $\tau_k\geq \eta\hat{\tau}_k$.

{Consider the case $\mu>0$ and $L_{yy}=0$. Since $L_{yy}=0$, second term in~\eqref{eq:tauhat} is not binding and we have $\hat{\tau}_k\triangleq\frac{-L_{xx}+\sqrt{L^2_{xx}+4(1-\delta)L^2_{yx}\gamma_k/c_\alpha}}{2L^2_{yx}\gamma_k/c_\alpha}$. Moreover, $\mu>0$ and Line~\ref{algeq:gamma_update_APDB} in APDB imply that $\gamma_{k+1}\geq\gamma_k\geq\gamma_0$ for $k\geq 0$.} Next, to analyze this case, we show a useful inequality: for any $a\geq 0$  and $b,c>0$, there exists $d\in (0,1]$ such that $\sqrt{a^2+c b^2}\geq a+\sqrt{c}b d$. In fact, the inequality can be written equivalently as $d^2+\frac{2a}{b\sqrt{c}}d-1\leq 0$, which holds if $d^2+\frac{2a}{b\sqrt{\bar{c}}}d-1\leq 0$ has a solution for any $0<\bar{c}\leq c$. Given such $\bar{c}$, $d= -\frac{a}{b\sqrt{\bar{c}}}+\sqrt{\frac{a^2}{b^2\bar{c}}+1}>0$ solves this tighter quadratic inequality. Employing this result within the definition of $\hat{\tau}_k$ for the case $L_{yy}=0$, i.e., setting $a=L_{xx}$, $b=2\sqrt{\tfrac{1-\delta}{c_\alpha}}L_{yx}$, $c=\gamma_k$, and $\bar{c}=\gamma_0$, implies $\hat{\tau}_k\geq \Psi_1\sqrt{\gamma_0/\gamma_k}$ for $k\geq 0$.

{Now, suppose $\mu=0$. Line~\ref{algeq:gamma_update_APDB} in APDB implies that $\gamma_k=\gamma_0$ for $k\geq 0$. Hence, from \eqref{eq:tauhat}, we have $\hat{\tau}_k=\hat{\tau}_0$ for $k\geq 0$; thus, {when $L_{yy}=0$, we get $\hat{\tau}_0\geq \Psi_1$, and when $L_{yy}>0$, we get $\hat{\tau}_0\geq \min\{\Psi_1,\Psi_2\}$}.}
\end{proof}
\begin{lemma}
\label{lem:k2-rate}
\sa{Suppose $\mu>0$, and $L_{yy}=0$. \rev{The} step-size sequences generated by both APD and APDB, displayed in Algorithms~\ref{alg:APD} and~\ref{alg:APDB}, respectively, satisfy $\sigma_k=\Omega(k)$, $\tau_k=\Omega(1/\sigma_k)$, and $\tau_k/\sigma_k=\cO(1/k^2)$ for $k\geq 0$. Indeed, $\sigma_k \geq \frac{\Gamma^2}{3\mu} k$, $\tau_k\sigma_k\geq \Gamma^2/\mu$ and $\gamma_k^{-1}=\tau_k/\sigma_k\leq 9/(\Gamma^2 k^2)$ for $k\geq 0$, where $\Gamma=\mu\tau_0\sqrt{\gamma_0}$ for APD and 
{$\Gamma=\mu\eta\Psi_1\sqrt{\gamma_0}$} for APDB with $\Psi_1$ as defined in~\eqref{eq:tauhat_bound}. Furthermore, for all $\epsilon>0$, $\sigma_k\geq \frac{\Gamma^2}{(2+\epsilon)\mu}k$ and $\tau_k/\sigma_k\leq (2+\epsilon)/(\Gamma^2 k^2)$ for $k\geq \lceil\frac{1}{\epsilon}\rceil$.}
\end{lemma}
\begin{proof}
First, consider $\{\tau_k,\sigma_k,\theta_k\}_k$ generated by APD as shown in Algorithm~\ref{alg:APD}. Note, $\tau_{k+1}=\tau_k\sqrt{\frac{\gamma_k}{\gamma_{k+1}}}$ implies that $\tau_k=\tau_0\sqrt{\frac{\gamma_0}{\gamma_k}}$, for all $k\geq 0$; therefore, using the update rule for $\gamma_{k+1}$ in Line~\ref{algeq:gamma_update_APD} of Algorithm~\ref{alg:APD}, we conclude that for $k\geq 0$,
\begin{align}\label{gamma-k-apd}
\gamma_{k+1}=\gamma_k(1+\mu\tau_k)\geq \gamma_k+\mu\tau_0\sqrt{\gamma_0\gamma_k}.
\end{align}

Second, consider $\{\tau_k,\sigma_k,\theta_k\}_k$ generated by APDB as shown in Algorithm~\ref{alg:APDB}.
{Since $\tau_k\geq \eta\hat{\tau}_k$ for all $k\geq 0$, the update rule for $\gamma_{k+1}$ and Lemma~\ref{lem:parameter} 
imply that
\begin{align}\label{gamma-k}
\gamma_{k+1}=\gamma_k(1+\mu\tau_k)\geq \gamma_k(1+\mu\eta\hat{\tau}_k)\geq \gamma_k+{\mu\eta\Psi_1\sqrt{\gamma_0\gamma_k}},\quad\forall k\geq 0.
\end{align}}%
{Next, we will give a unified analysis for APD and APDB step-size sequences as the bounds in \eqref{gamma-k-apd} and \eqref{gamma-k} have the same form: $\gamma_{k+1}\geq \gamma_k+\Gamma\sqrt{\gamma_k}$, where $\Gamma$, defined in the statement, depends on the algorithm implemented.}
Using induction one can show that {$\gamma_k\geq  \frac{\Gamma^2}{(2+\epsilon)^2} k^2$ for $k\geq \lceil\frac{1}{\epsilon}\rceil$; hence, setting $\epsilon=1$, we get $\gamma_k\geq  \frac{\Gamma^2}{9} k^2$ for $k\geq 0$. Note $\sigma_k=\gamma_k\tau_k$ and $\gamma_{k+1}=\gamma_k(1+\mu\tau_k)$ imply that $\sigma_{k}=\frac{\gamma_{k+1}- \gamma_k}{\mu}$.} Therefore, since $\gamma_{k+1}- \gamma_k\geq \Gamma\sqrt{\gamma_k}\geq \frac{\Gamma^2}{3} k$, we have $\sigma_{k} 
\geq \frac{\Gamma^2}{3\mu} k$, and $\tau_k\sigma_k=\frac{(\gamma_{k+1}-\gamma_k)^2}{\mu\gamma_k}\geq \Gamma^2/\mu$ for $k\geq 0$.
Moreover, 
$\tau_k/\sigma_k=1/\gamma_k=\cO(1/k^2)$.
\end{proof}
\subsection{\sa{Proof of Theorem~\ref{thm:main}}}
\label{sec:main_thm_proof}
Below we establish 
the 
results of Theorem~\ref{thm:main}.
\subsubsection{Rate Analysis}
We first show that $\{\tau_k,\sigma_k,\theta_k\}_{k\geq 0}$ generated by APD in Algorithm~\ref{alg:APD} satisfies \eqref{eq:step-size-rule}. The step-size update rule of APD implies that for $k\geq 0$,
{\small
\begin{align}
\label{eq:step-size-rule-proof}
\theta_{k+1}=\frac{\sigma_k}{\sigma_{k+1}}=\frac{\tau_k\gamma_k}{\tau_{k+1}\gamma_{k+1}}=\sqrt{\frac{\gamma_k}{\gamma_{k+1}}}=\frac{1}{\sqrt{1+\mu\tau_k}},\quad
\tau_{k+1}=\tau_k\sqrt{\frac{\gamma_k}{\gamma_{k+1}}}=\theta_{k+1}\tau_k.
\end{align}}%
\vspace*{-3mm}

Given $\tau_0,\sigma_0>0$ satisfying \eqref{eq:initial_step_condition} for some particular $\delta, c_\alpha, c_\beta\in\reals_+$ as stated in Theorem~\ref{thm:main}, we will consider $\{\alpha_k,\beta_k,t_k\}_{k\geq 0}$ chosen as
\begin{align}
\label{eq:abt_seq}
t_k=\sigma_k/\sigma_0,\quad \alpha_k=c_\alpha/\sigma_{k-1},\quad \beta_k=c_\beta/\sigma_{k-1},\quad \forall k\geq 0.
\end{align}
We next show that $\{\tau_k,\sigma_k,\theta_k\}$ satisfies Assumption~\ref{assum:step-2}.
Indeed, since $\theta_k=\sigma_{k-1}/\sigma_k$, for the choice of $\{\alpha_k,\beta_k\}$ in~\eqref{eq:abt_seq},  \eqref{eq:step-size-condition-pd} can be written as
\begin{align}
\label{eq:step_cond_induction}
\frac{1-\delta}{\tau_k}\geq L_{xx}+\frac{L_{yx}^2}{c_\alpha}\sigma_k,\quad 1-(\delta+c_\alpha+c_\beta)\geq \frac{L_{yy}^2}{c_\beta}\sigma_k^2.
\end{align}
\rev{Clearly,} \eqref{eq:initial_step_condition} implies that \eqref{eq:step_cond_induction} holds for $k=0$. When $\mu=0$, i.e., Part I, we have $\gamma_k=\gamma_0$ and $\theta_k=1$ for $k\geq 0$; hence, $\tau_k=\tau_0$ and $\sigma_k=\sigma_0$ for $k\geq 0$. Thus, \eqref{eq:step-size-condition-pd} holds for all $k\geq 0$ for $\{\tau_k, \sigma_k,\theta_k\}$ produced by APD. For the case $\mu>0$, i.e., Part II, we will use induction to show that \eqref{eq:step-size-condition-pd} holds. Recall that for this case, we assume $L_{yy}=0$; hence, the second condition in \eqref{eq:step_cond_induction} holds for any $\sigma_k$ as long as $1\geq \delta+c_\alpha+c_\beta$. Now suppose the first condition in \eqref{eq:step_cond_induction} holds for some $k\geq 0$, using $\sigma_{k+1}=\sigma_k\sqrt{\gamma_{k+1}/\gamma_k}$ and $\gamma_{k+1}/\gamma_k\geq 1$, we get
\begin{align}
\frac{1-\delta}{\tau_{k+1}}=\frac{1-\delta}{\tau_k}\sqrt{\frac{\gamma_{k+1}}{\gamma_k}}\geq L_{xx}+\frac{L_{yx}^2}{c_\alpha}\sigma_{k+1}.
\end{align}
This completes the induction. Moreover, Lemma~\ref{lem:step-size-seq} implies that $\{\tau_k,\sigma_k,\theta_k\}$ generated by APD satisfies \eqref{eq:step-size-condition-theta} for $\{t_k\}$ such that $t_k=\sigma_k/\sigma_0$ for $k\geq 0$. Thus, Assumption~\ref{assum:step-2} holds for $\{\alpha_k,\beta_k,t_k\}_{k\geq 0}$ shown in \eqref{eq:abt_seq}. Lemma~\ref{lem:stronger-step-condition} implies that $\{x_k,y_k\}$ and $\{\tau_k,\sigma_k,\theta_k\}$ generated by APD satisfies Assumption~\ref{assum:step}. Therefore, \eqref{eq:delta} of \rev{the main result I} directly follows from Theorem~\ref{thm:general-bound}.

Consider the setting in Part I of Theorem~\ref{thm:main}, i.e., $\mu=0$. Since $\mu=0$, clearly APD step-size sequence satisfies $\tau_k=\tau_0$, $\sigma_k=\sigma_0$ for $k\geq 0$; hence, $\theta_k=1$ and $t_k=1$ for $k\geq 0$, which implies $T_K=\sum_{k=0}^{K-1}t_k=K$ for $K\geq 1$. Moreover, for any $k\geq 0$, we have $\frac{1}{\sigma_k}-\theta_k(\alpha_k+\beta_k)=\frac{1-(c_\alpha+c_\beta)}{\sigma_0}\geq 0$; thus, \eqref{eq:rate} implies
{\small
\begin{align}\label{eq:rate_I}
\frac{1}{\tau_0}\bD_{\cX}(x,x_{K}) +\frac{1-(c_\alpha+c_\beta)}{\sigma_0}\bD_{\cY}(y,y_{K}) +K[\cL(\bar{x}_{K},y)-\cL(x,\bar{y}_{K})]\leq \Delta(x,y).
\end{align}}%
The rate result in \eqref{eq:delta} for Part I follows from dropping the non-negative terms on the left hand-side of \eqref{eq:rate_I}. %
Finally, in case a saddle point $(x^*,y^*)$ exists, letting $x=x^*$ and $y=y^*$ in \eqref{eq:rate_I} and using the fact that $\cL(\bar{x}_K,y^*)-\cL(x^*,\bar{y}_K)\geq 0$, we obtain the bound on iterates given in the theorem. 

To show Part II of Theorem~\ref{thm:main}, now suppose $\mu>0$ and $L_{yy}=0$. %
Since $\theta_k=\sigma_{k-1}/\sigma_k$, 
$c_\beta=0$, we have $\frac{1}{\sigma_k}-\theta_k(\alpha_k+\beta_k)=\frac{1-c_\alpha}{\sigma_k}\geq0$ for all $k\geq 0$. Hence, using $t_k=\sigma_k/\sigma_0$ for $k\geq 0$, \eqref{eq:rate} implies that
{\small
\begin{align}\label{eq:rate_II}
\frac{\sigma_K}{\tau_{K}}\frac{1}{\sigma_0}\bD_{\cX}(x,x_{K}) +\frac{1-c_\alpha}{\sigma_0}\bD_{\cY}(y,y_{K}) + T_K [\cL(\bar{x}_{K},y)-\cL(x,\bar{y}_{K})]\leq
\Delta(x,y).
\end{align}}%
Lemma~\ref{lem:k2-rate} shows that $\sigma_k=\Omega(k)$; hence, $T_K=\sum_{k=0}^{K-1}t_k=\sum_{k=0}^{K-1}\sigma_k/\sigma_0=\Omega(K^2)$. Thus, the rate result in \eqref{eq:delta} for Part II follows from dropping the non-negative terms on the left hand-side of \eqref{eq:rate_II}.
Finally, if a saddle point $(x^*,y^*)$ exists, then letting $x=x^*$ and $y=y^*$ in \eqref{eq:rate_II} and using the fact that $\cL(\bar{x}_K,y^*)-\cL(x^*,\bar{y}_K)\geq 0$ implies the bound on iterates. Moreover, Lemma~\ref{lem:k2-rate}  shows that $\tau_k/\sigma_k=\cO(1/k^2)$; hence, we get $\bD_{\cX}(x^*,x_{k})=\cO(1/k^2)$.

Next, we discuss the convergence properties of the APD iterate sequence. Indeed, the result follows directly from Theorem~\ref{thm:convergence}; hence, we only need to verify the assumptions of the theorem.

\subsubsection{Non-ergodic Convergence Analysis}
\label{sec:nonergodic-APD}
Suppose a saddle point exists. We set $\delta, c_\alpha>0$ and $c_\beta\geq 0$ such that $\delta+c_\alpha+c_\beta\leq 1$. Due to our choice of $\{\alpha_k,\beta_k,t_k\}$ in~\eqref{eq:abt_seq}, using $\gamma_k\geq \gamma_0$ for $k\geq 0$, the general assumption of Theorem~\ref{thm:convergence} holds, i.e., $\inf_{k\geq 0}t_k\min\{\frac{1}{\tau_k},~\frac{1}{\sigma_k}-\theta_k(\alpha_k+\beta_k)\}\geq\delta'$ for  $\delta'=\min\{\frac{1}{\tau_0},~\frac{1-(c_\alpha+c_\beta)}{\sigma_0}\}>0$.

For Part I, i.e., $\mu=0$, since $\tau_k=\tau_0$ and $\sigma_k=\sigma_0$ for $k\geq 0$, $\lim_{k\to \infty}\min\{\tau_k,~\sigma_k\}>0$; hence, from Theorem~\ref{thm:convergence}, any limit point of $\{(x_k,y_k)\}_{k\geq 0}$ is a saddle point. In addition, since $t_k=1$ for $k\geq 0$, $\lim_{k\to \infty}\alpha_k=\frac{c_\alpha}{\sigma_0}>0$, and since for the case $L_{yy}>0$ we have $\lim_{k\to \infty}\beta_k=\frac{c_\beta}{\sigma_0}>0$, we conclude that $\{(x_k,y_k)\}$ has a unique limit point.

For Part II, i.e., $\mu>0$ and $L_{yy}=0$, it holds that $\tau_k\to 0$ and $\lim_{k\to\infty}\sigma_k>0$. Note for this setting, $\varphi_{\cX}(\cdot)=\norm{\cdot}_{\cX}^2$ defining $\bD_{\cX}$ is Lipschitz differentiable. Moreover, since $t_k=\sigma_k/\sigma_0$ and Lemma~\ref{lem:k2-rate} shows that $\tau_k=\Omega(1/\sigma_k)$, we have $t_k=\Omega(\frac{1}{\tau_k})$; thus any limit point of $\{(x_k,y_k)\}_{k\geq 0}$ is a saddle point.
\subsection{\sa{Proof of Theorem~\ref{thm:backtrack}}}
\label{sec:main_thm_proof_backtrack}
Below we provide rate and non-ergodic convergence analyses for APDB.
\subsubsection{Rate Analysis}
\label{sec:rate_APDB}
We first show that $\{x_k,y_k\}$ and $\{\tau_k,\sigma_k,\theta_k\}$ generated by APDB, displayed in Algorithm~\ref{alg:APDB}, satisfies Assumption~\ref{assum:step}. Indeed, Lemma~\ref{lem:step-size-seq} implies that $\{\tau_k,\sigma_k,\theta_k\}$ generated by APDB satisfies \eqref{eq:step-size-condition-theta} for $\{t_k\}$ such that $t_k=\sigma_k/\sigma_0$ for $k\geq 0$. Moreover, APDB is a special GAPD corresponding to a particular $\{\tau_k,\sigma_k,\theta_k\}$, and Lemma~\ref{lem:parameter} shows that for any $k\geq 0$, the backtracking condition in Line~\ref{algeq:test_function} of Algorithm~\ref{alg:APDB} holds after finite number of inner iterations. Thus, Assumption~\ref{assum:step} clearly holds, and \eqref{eq:delta} 
directly follows from Theorem~\ref{thm:general-bound} and observing that $\alpha_k$ and $\beta_k$ choice in APDB implies that $\frac{1}{\sigma_k}-\theta_k(\alpha_k+\beta_k)\geq \frac{1-(c_\alpha+c_\beta)}{\sigma_k}\geq 0$. 
Next, we show the number of inner iterations for each outer iteration $k\geq 0$ can be uniformly bounded by {$1+\log_{1/\eta}(\frac{\bar{\tau}}{\Psi})$}, where $\Psi=\Psi_1$ when $L_{yy}=0$, and $\Psi=\min\{\Psi_1,\Psi_2\}$ when $L_{yy}>0$.

For Part I, since $\mu=0$, 
it is clear that $\gamma_k=\gamma_0>0$ for all $k\geq 0$. From Lemma~\ref{lem:parameter}, $\tau_k\geq\eta\hat{\tau}_k$ for
some {$\hat{\tau}_k\geq \Psi\sqrt{\gamma_0/\gamma_k}=\Psi$} for all $k\geq 0$. Since $\{\tau_k\}_k$ is a diminishing sequence, we also have {$\tau_k\leq \tau_0\leq \bar{\tau}$} 
which implies that the number of backtracking steps is at most 
$1+\log_{1/\eta}(\frac{\bar{\tau}}{\Psi})$. Furthermore, since $\sigma_k=\gamma_0\tau_k$ for $k\geq 0$, we conclude that  \eqref{eq:delta} holds with $T_K=\sum_{k=0}^{K-1}\sigma_k/\sigma_0 \geq {\frac{\eta\Psi}{\tau_0} K}$. {Finally, if a saddle point exists, then using $\cL(\bar{x}_K,y^*)-\cL(x^*,\bar{y}_K)\geq 0$, \eqref{eq:rate} implies \eqref{eq:xy-bound-I} for any saddle point $(x^*,y^*)$.}

Consider Part II, i.e., $\mu>0$ and $L_{yy}=0$. Since {$\tau_{k+1}\leq\tau_k\sqrt{\gamma_k/\gamma_{k+1}}$}, we get $\tau_{k}\leq {\bar{\tau}}\sqrt{\gamma_0/\gamma_k}$ for all $k\geq 0$; moreover, according to Lemma \ref{lem:parameter}, we have that {$\tau_k\geq \eta\hat{\tau}_k\geq\eta\Psi\sqrt{\gamma_0/\gamma_k}$} for $k\geq 0$. Therefore, we conclude that the number of backtracking steps is at most {$1+\log_{1/\eta}(\frac{\bar{\tau}}{\Psi})$}. Moreover, Lemma~\ref{lem:k2-rate} shows that $\sigma_k=\Omega(k)$; hence, \eqref{eq:delta} for Part II holds with $T_K=\sum_{k=0}^{K-1}t_k=\sum_{k=0}^{K-1}\sigma_k/\sigma_0=\Omega(K^2)$. {Finally, if a saddle point exists, then  \eqref{eq:rate} implies \eqref{eq:xy-bound-II} for any saddle point $(x^*,y^*)$.}
\subsubsection{Non-ergodic Convergence Analysis} Using the same arguments in Section~\ref{sec:nonergodic-APD}, the general assumption of Theorem~\ref{thm:convergence} holds for $\delta'$ given in Section~\ref{sec:nonergodic-APD}.

For Part I, since $\mu=0$, for $k\geq 0$, $\gamma_k=\gamma_0$; hence, $\sigma_k=\gamma_0\tau_k$. Thus, $t_k=\sigma_k/\sigma_0=\tau_k/\tau_0$. As discussed in Section~\ref{sec:rate_APDB}, for Part I we have ${\eta\Psi}\leq \tau_k\leq\tau_0$ for $k\geq 0$, which implies that {$\inf_{k\geq 0}t_k\geq {\eta\Psi/\tau_0}$ and $\sup_{k\geq 0} t_k\leq 1$}. Furthermore, since $\sigma_k\leq\gamma_0\tau_0=\sigma_0$, we also get $\rev{\liminf_{k\to \infty}}\min\{\alpha_k,\beta_k\}>0$ when $L_{yy}>0$, and $\rev{\liminf_{k\to \infty}}\alpha_k>0$ when $L_{yy}=0$. Therefore,  the assumptions for Case 1 of Theorem~\ref{thm:convergence} are satisfied, and we have convergence to a unique saddle point.

Part II follows from the same arguments given in Section~\ref{sec:nonergodic-APD}.

\begin{remark}\label{rem:larger-step}
\sa{Selecting larger 
step-sizes may improve overall practical behavior of the algorithm. To this aim, one can adopt non-monotonic $\{\tau_k\}$ within APDB to possibly select larger steps -- this might increase the number of backtracking steps as \sa{the outer iteration counter $k\geq 0$ increases}. For example, given any $\tau_{\max}>0$, setting $\tau_{k+1}=\min\{\tau_k\sqrt{\frac{\gamma_k}{\gamma_{k+1}}(1+\frac{\tau_k}{\tau_{k-1}})},\tau_\max\}$ in Line~\ref{algeq:gamma_update_APDB} of APDB implies that the number of backtracking steps at iteration $k$ is bounded by $N_k\triangleq 1+\log_{1/\eta}(\frac{\tau_\max}{\Psi}\sqrt{\gamma_k/\gamma_0})$. When $\mu=0$, $N_k=1+\log_{1/\eta}(\frac{\tau_\max}{\Psi})$, while $N_k=\cO(\log(k))$ when $\mu>0$. Thus, given any $(x,y)\in\cX\times\cY$, to guarantee $\cL(\bar{x}_K,y)-\cL(x,\bar{y}_K)\leq \epsilon$, one needs $\cO(1/\epsilon)$ inner iterations in total when $\mu=0$ compared to $\cO(\frac{1}{\sqrt{\epsilon}}\log(1/\epsilon))$ inner iterations when $\mu>0$. It is worth reemphasizing that $\cO(1/\epsilon)$ and $\cO(1/\sqrt{\epsilon})$ are the lower complexity bounds associated with first-order primal-dual methods for convex-concave and strongly convex-concave bilinear SP problems, respectively~\cite{ouyang2018lower}.}
\end{remark}

\section{Application to {the} Constrained Convex Optimization}
\label{sec:constrained}
An important special case of \eqref{eq:original-problem} is the convex optimization problem with a nonlinear conic constraint, formulated as in~\eqref{eq:conic_problem}. 
Indeed, 
\eqref{eq:conic_problem} can be reformulated as a saddle point problem as shown in \eqref{eq:conic_problem_equivalent}, which is in the form of \eqref{eq:original-problem}. Clearly, $L_{yy}=0$, and $L_{yx}>0$ exists if $G$ is Lipschitz. Moreover, for any fixed $y\in\cY$, a bound on $L_{xx}$, \sa{i.e.,} the Lipschitz constant of $\grad_x\Phi(x,y)$ as a function of $x$, can be computed as \vspace*{-2mm}
{\small
\begin{align}\label{eq:L_xx-co}
\norm{\grad_x \Phi(x,y)-\grad_x \Phi(\bar{x},y)}_{\cX^*}&\leq \norm{\grad g(x)-\grad g(\bar{x})}_{\cX^*}+\norm{\grad G(x)^\top y-\grad G(\bar{x})^\top y}_{\cX^*} \nonumber \\
&\leq (L_g+L_G\norm{y}_{\cY})\norm{x-\bar{x}}_\cX ,\quad \forall~x,\bar{x}\in
\rev{\dom f}.\vspace*{-3mm}
\end{align}}%

Now we customize our algorithm and state its convergence result for \eqref{eq:conic_problem}. 
\begin{assumption}
\label{assum:optimization}
Suppose $(\cX,~\norm{\cdot}_{\cX})=(\reals^n,~\norm{\cdot})$ and $(\cY,~\norm{\cdot}_{\cY})=(\reals^m,~\norm{\cdot})$ are Euclidean spaces. 
We assume that a dual optimal solution $y^*\in\cY$ exists. Consider the objective $\rho(x)\triangleq f(x)+g(x)$ in \eqref{eq:conic_problem}, suppose $f:\reals^n\rightarrow\reals\cup\{\infty\}$ is convex (possibly nonsmooth), $g:\reals^n\rightarrow\reals$ is convex with a Lipschitz continuous gradient with constant $L_g$, and $\cK\subset\reals^m$ is a closed convex cone. Moreover, {$G:\reals^n\rightarrow \reals^m$} is  $\cK$-convex~\cite{boyd2004convex}, Lipschitz continuous 
with constant $C_G>0$ and it has a Lipschitz continuous {Jacobian}, denoted by $\grad G:\reals^n\rightarrow\reals^{m\times n}$, with constant $L_G\geq 0$.
\end{assumption}
\rev{Assumption \ref{assum:optimization} 
ensures that $\Phi(\cdot,\cdot)$ is convex-concave, and adopting the Euclidean metric will help us to convert our SP rate results into suboptimality and infeasibility rate results for the problem \eqref{eq:conic_problem}.}
{In the rest, 
let $\cP_{\cK}(w) \triangleq \argmin_{y\in\cK}\norm{y-w}$, $d_{\cK}(w)\triangleq \norm{\cP_{\cK}(w)-w}=\norm{\cP_{\cK^{\circ}}(w)}$ where $\cK^\circ=-\cK^*$ denotes the polar cone of $\cK$.}

We next consider two scenarios: \textbf{i)} a dual bound is {available}, \textbf{ii)} a dual bound is not {available} -- \rev{we allow the dual solution set to be possibly \emph{unbounded}.}
\subsection{A dual bound is available} For any \emph{given} $\kappa, B>0$ such that $\norm{y^*}\leq B$ for some dual optimal solution $y^*$, let
{\small
\begin{align}
\label{eq:B}
    \cB\triangleq \{y\in\reals^m \ :\ \norm{y}\leq B+\kappa\}.
\end{align}}%
Thus, the Lipschitz constant $L_{xx}$ for \eqref{eq:L_xx-co} can be chosen as $L_{xx}=L_g+(B+\kappa)L_G$ and we set $h(y)=\mathbb{I}_{\cK^*\cap \cB}$.
Such a bound $B$ can be computed if a slater point for \eqref{eq:conic_problem} is available. 
Using the following lemma one can compute a dual bound efficiently.
\begin{lemma}\cite{aybat2016distributed}\label{dual-bound}
Let $\bar{x}$ be a Slater point for \eqref{eq:conic_problem}, i.e., $\bar{x}\in\relint(\dom \rho)$ such that $G(\bar{x})\in\intr(-\cK)$, and $q:\reals^m\rightarrow\reals\cup\{-\infty\}$ denote the 
dual function, i.e.,
{\small
$$q(y)\triangleq
       \begin{cases}
         \inf_{x}\rho(x)+ \fprod{G(x),~y}, & \hbox{if $y\in\cK^*$;} \\
         -\infty, & \hbox{o.w.}
       \end{cases}
$$}%
For any $\bar{y}\in \dom q$, let $Q_{\bar{y}}\triangleq\{y \in \dom q:\ q(y)\geq q(\bar{y})\}\subset\cK^*$ denote the corresponding superlevel set. Then for all $\bar{y}\in \dom q$, 
$Q_{\bar{y}}$ can be bounded as follows:
\vspace*{-4mm}
{\small
\begin{equation}
\label{eq:dual-bound-radius}
\norm{y}\leq \frac{\rho(\bar{x})-q(\bar{y})}{r^*},\quad \forall y \in Q_{\bar{y}},
\end{equation}}%
where $0<r^*\triangleq \min_w\{-\fprod{G(\bar{x}),~w}:\ \|w\|= 1,\ w\in \cK^*\}$.

Although this is not a convex problem due to the nonlinear equality constraint, one can upper bound \eqref{eq:dual-bound-radius} using $0<\tilde{r}\leq r^*$, which can be efficiently computed by solving a convex problem
$\tilde{r}\triangleq\min_w\{-\fprod{G(\bar{x}),~w}:\ \|w\|_1= 1,\ w\in \cK^*\}$.
\end{lemma}
\begin{corollary}\label{cor:bound_known}
Consider the convex optimization
problem in~\eqref{eq:conic_problem}. \rev{Suppose Assumption \ref{assum:optimization} holds and} let 
$\{(x_k,y_k)\}_{k\geq 0}$ be the iterate
sequence when APD is applied to the following SP problem with $h(y)=\mathbb{I}_{\cK^*\cap \cB}(y)$, \rev{where $\cB$ is defined in~\eqref{eq:B},}
\begin{align}
\label{eq:constrained-problem-thm}
\min_{x\in\reals^n}\max_{y\in\reals^m} f(x)+g(x)+\fprod{G(x),y}-h(y).
\end{align}
Let $\tau_0=c_\tau\big(L_g+(B+\kappa)L_G+\tfrac{1}{\alpha}C_G^2\big)^{-1}$ and $\sigma_0=c_\sigma\alpha^{-1}$ for some $\alpha>0$ and $c_\tau,c_\sigma\in(0,1]$. Then, for all $K\geq 1$,
\begin{align}
\label{eq:optimization-bounds}
\max\Big\{|\rho(\bar{x}_K)-\rho(x^*)|,~\kappa~d_{-\cK}\big(G(\bar{x}_K)\big)\Big\}\leq \frac{1}{T_K}\Delta(x^*,y_K^\dag)=\cO(1/T_K),
\end{align}
where $y_K^\dag=(\norm{y^*}+\kappa)\cP_{\cK^*}\big(G(\bar{x}_K)\big) \norm{\cP_{\cK^*}\big(G(\bar{x}_K)\big)}^{-1}$, $\bar{x}_K$ and $T_K$ are defined in Theorem~\ref{thm:main}, and \rev{$\Delta(\cdot,\cdot)$} is defined in \eqref{eq:delta}. Note that $\sup_{K\geq 1}\|y_K^\dag\|=\|y^*\|+\kappa$.

\textbf{(Part I.)} Suppose the objective 
in \eqref{eq:conic_problem} is merely convex. Then \eqref{eq:optimization-bounds} holds with $T_K=K$ for all $K\geq 1$ when $\theta_k=1$, $\tau_k=\tau_0$, $\sigma_k=\sigma_0$ and $t_k=1$ for all $k\geq 0$.

\textbf{(Part II.)} Suppose the objective 
in \eqref{eq:conic_problem} is strongly convex with $\mu>0$. Then \eqref{eq:optimization-bounds} holds with $T_K=\Theta(K^2)$ for all $K\geq 1$ when $\{\tau_k,\sigma_k,\theta_k\}_{k\geq 0}$ sequence is chosen as in \eqref{eq:step-size-rule}, and $t_k=\sigma_k/\sigma_0$ for $k\geq 0$. Moreover, for all $K\geq 1$,
\begin{align*}
\bD_\cX(x^*,x_K)\leq \frac{\tau_K}{\sigma_K}\sigma_0\Delta(x^*,y^*)=\cO(1/K^2).
\end{align*}
\end{corollary}
\begin{proof}
It is easy to verify that $\fprod{G(\bar{x}_K),\sa{y_K^\dag}}= (\norm{y^*}+\kappa) d_{-\cK}(G(\bar{x}_K))$ as for any $w\in\reals^m$ we have $w=\cP_{-\cK}(w)+\cP_{\cK^*}(w)$ and $\fprod{\cP_{-\cK}(w),~\cP_{\cK^*}(w)}=0$. Hence, $\cL(\bar{x}_K,\sa{y_K^\dag})=\rho(\bar{x}_K)+(\norm{y^*}+\kappa) d_{-\cK}(G(\bar{x}_K))$ since $y_K^\dag\in\cK^*$. Note that $\rho(x^*)=\cL(x^*,y^*)\geq \cL(x^*,\bar{y}_K)$. Therefore, \eqref{eq:delta} implies that 
{\small
\begin{align}\label{eq:L_upper_bound}
 \rho(\bar{x}_K)-\rho(x^*)+(\norm{y^*}+\kappa) d_{-\cK}(G(\bar{x}_K))\leq \cL(\bar{x}_K,y_K^\dag)-\cL(x^*,\bar{y}_K)\leq \sa{\frac{1}{T_K}}\Delta(x^*,y_K^\dag).
\end{align}}%
On the other hand, we also have
\begin{align}\label{eq:L_lower_bound}
0\leq \cL(\bar{x}_K,y_K^\dag)-\cL(x^*,y^*) &= \rho(\bar{x}_K)-\rho(x^*)+\fprod{G(\bar{x}_K),y^*} \nonumber \\
&\leq  \rho(\bar{x}_K)-\rho(x^*)+\norm{y^*} d_{-\cK}(G(\bar{x}_K)),
\end{align}
where we used the fact that for any $y\in\reals^m$, $\fprod{y^*,y}\leq \fprod{y^*,\cP_{\cK^*}(y)}\leq \norm{y^*} d_{-\cK}(y)$. Combining \eqref{eq:L_upper_bound} and \eqref{eq:L_lower_bound} gives the desired result.
\end{proof}
\begin{remark}
\label{rem:delta}
\sa{
Since $\|y_K^\dag\|=\norm{y^*}+\kappa$ for all $K\geq 1$, one has $\sup_{K\geq 1}\Delta(x^*,y^\dag_K)\leq \frac{1}{2\tau_0}\norm{x^*-x_0}^2+\frac{1}{\sigma_0}((\norm{y^*}+\kappa)^2+\norm{y_0}^2)$. In practice, $\kappa=B$ can be used.}
\end{remark}
\subsection{\sa{A dual bound is not available}}
\label{sec:without_dualbound}
Here we consider the situation where the dual bound for \eqref{eq:conic_problem} is not known and hard to compute. We consider two subcases.\\
\indent {\bf CASE 1: {$L_{xx}$ exists.}} 
Suppose $L_{xx}$ \emph{exists}, but \emph{not} known, then one can {\emph{immediately}} implement APDB which locally estimates the Lipschitz constants.
\begin{corollary}\label{cor:bound_unknown}
\sa{Consider 
\eqref{eq:conic_problem} \rev{under Assumption \ref{assum:optimization}.} Let $\{(x_k,y_k)\}_{k\geq 0}$ be the APDB iterate sequence when APDB is applied to \eqref{eq:constrained-problem-thm} with $h(y)=\mathbb{I}_{\cK^*}(y)$. The bounds in Part I and Part II of Corollary~\ref{cor:bound_known} continue to hold {for any $\kappa>0$} when the step-sizes are adaptively updated as described in Algorithm \ref{alg:APDB}.}
\end{corollary}

\indent {\bf CASE 2: {$L_{xx}$ does not {exist}}.} {Suppose $L_{xx}$ does not exist -- possibly when $\dom h$ is unbounded, e.g., see the example in Remark~\ref{rem:differentiability}.} In this case, one can still implement the algorithm APDB with a slight modification {and still guarantee convergence under Assumption~\ref{assum:optimization}}. In fact, \rev{since the global constant $L_{xx}$ does not exist for all $y\in\dom h$, the challenge in this scenario is to guarantee that there exists a \emph{bounded} set $\bar{\cB}\subset\reals^m$ such that APDB dual iterate sequence lies in it, i.e., $\{y_k\}\subset\bar{\cB}$, which would imply the existence of a constant $L_{xx}$ such that \eqref{eq:Lxx} holds for all $y\in\bar{\cB}$.
More precisely, we use induction to show the boundedness of $\{y_k\}$, and this result implies} that $\Phi$ satisfies \eqref{eq:L_xx-co} for $y=y_k$ with $L_{xx}=L_g+L_G\sup_k\norm{y_k}_{\cY}$ for all $k\geq 0$ -- this is what we need for the proof of Theorem~\ref{thm:backtrack} to hold if we relax \textbf{(i)} in Assumption~\ref{assum}.
\rev{However, naively using induction 
to construct a uniform bound on $\{y_k\}_k$ fails 
as} one needs the Lipschitz constant of $\grad_x\Phi(\cdot,y_{k+1})$ to bound $\norm{y_{k+1}}$ which depends on $y_{k+1}$ at iteration $k$. A remedy to this circular {argument} is to perform the $x$-update first, followed by the $y$-update; this way, at iteration $k\geq 0$, one needs to bound the Lipschitz constant of $\grad_x\Phi(\cdot,y_{k})$, instead of $\grad_x\Phi(\cdot,y_{k+1})$, which is now possible as a bound on $\norm{y_k}$ is available through the induction hypothesis.

\sa{To solve \eqref{eq:conic_problem}, we will implement a modified version of APDB on \eqref{eq:constrained-problem-thm}, which is a special case of \eqref{eq:original-problem} with $\Phi(x,y)=g(x)+\fprod{G(x),y}$. In the rest, we consider running a variant of Algorithm~\ref{alg:APDB} with $(x_{k+1},y_{k+1})\gets\hbox{\textbf{MainStep}}(x_k,y_k,x_{k-1},y_{k-1},\tau_k,\sigma_k,\theta_k)$ step modified as follows:}
\sa{
\begin{subequations}\label{eq:switchAPD}
\begin{align}
&{s_k} \gets (1+\theta_k)\grad_x\Phi(x_k,y_k)-\theta_k\grad_x\Phi(x_{k-1},y_{k-1}), \\
&x_{k+1}\gets \argmin_{x\in\cX} f({x})+\fprod{{s_k}, ~x}+{\frac{1}{\tau_k}}\bD_{\cX}(x,x_k),\\
& y_{k+1}\gets \argmin_{y\in\cY} h(y)-\fprod{\grad_y\Phi(x_{k+1},y_k),~ y}+{\frac{1}{\sigma_k}}\bD_{\cY}(y,y_k).
\end{align}
\end{subequations}}%
\sa{This switch of $x$- and $y$-updates within Algorithm~\ref{alg:APDB} also requires modifying the test function $E_k$ in \eqref{eq:Ek}. Now we redefine $E_k$ used in Line~\ref{algeq:test_function} as follows:}
\sa{\small
\begin{align}
\label{eq:modified_Ek}
E_k(x,y)\triangleq &\frac{1}{2\alpha_{k+1}}\norm{\grad_x\Phi(x,y)-\grad_x\Phi(x,y_k)}^2-\frac{1}{\sigma_k}\bD_\cY(y,y_k)\nonumber\\
&+\frac{1}{2\beta_{k+1}}\norm{\grad_x\Phi(x,y_k)-\grad_x\Phi(x_k,y_k)}^2-\Big(\frac{1}{\tau_k}-\theta_k(\alpha_k+\beta_k)\Big)\bD_\cX(x,x_k).
\end{align}}%
\sa{Next, we state the convergence result of the proposed method.} 
\begin{corollary}\label{cor:bound_not_exist}
\sa{Consider a variant of APDB where Line~\ref{algeq:APDB-mainstep} of Algorithm~\ref{alg:APDB} is replaced by the update-rule in \eqref{eq:switchAPD} and the test function in Line~\ref{algeq:test_function} is set as in~\eqref{eq:modified_Ek} with $\{\alpha_k,\beta_k\}$ chosen as $\alpha_{k+1}=c_\alpha/\tau_k$ and $\beta_{k+1}=\gamma_0 c_\beta/\sigma_k$ for $k\geq 0$.}

\rev{Suppose Assumption~\ref{assum:optimization} holds and }\sa{let $(x^*,y^*)$ be a primal-dual optimal solution to \eqref{eq:conic_problem}. Consider $\{(x_k,y_k)\}_{k\geq 0}$ generated by the modified APDB when applied to \eqref{eq:constrained-problem-thm} with $h(y)=\mathbb{I}_{\cK^*}(y)$. Assuming either $\cK=\reals^m_+$ or $\grad G$ is bounded on $\dom f$, one has 
\begin{align}
\label{eq:y_bound}
\norm{y_k}\leq 
 \bar{B}\triangleq \norm{y^*}+\sqrt{\gamma_0\norm{x^*-x_0}^2+\norm{y^*-y_0}^2},\quad \forall k\geq 0.
\end{align}
Moreover, for any $\kappa>0$,
\eqref{eq:optimization-bounds} continue to hold for $\{t_k\}$ such that $t_k=\sigma_k/\sigma_0$ for $k\geq 0$. Finally, for $\mu=0$, $T_K=\Omega(K)$ and for $\mu>0$, $T_K=\Omega(K^2)$.}
\end{corollary}
\begin{proof}
{Suppose not only a dual bound for \eqref{eq:conic_problem} is \emph{not} known and hard to compute; but also $L_{xx}$ for \eqref{eq:constrained-problem-thm} does not exist -- possibly when $\dom h$ is unbounded.} Here the main idea is to show that when APDB with update-rules of \eqref{eq:switchAPD}, applied to \eqref{eq:constrained-problem-thm} with $h(y)=\mathbb{I}_{\cK^*}(y)$, APDB generates $\{y_k\}$ that is bounded. Then the convergence results can be proved similar to the previous results using the dual iterate bound.

{To this end, we use induction, i.e., we assume that for some $K\geq 1$, 
$\norm{y_k}\leq \bar{B}$ 
for $k= 0,\hdots,K-1$, and we prove that $\norm{y_K}\leq \bar{B}$.} Note that the basis of induction clearly holds for $K=1$ as we have $\bar{B}\rev{\geq \norm{y^*}+\norm{y^*-y_0}\geq} \norm{y_0}$.

Given $\gamma_0,\bar{\tau}>0$, let $\sigma_{-1}=\gamma_0\bar{\tau}$ 
and consider $\{t_k,\alpha_k,\beta_k\}$ chosen as
\begin{equation}\label{eq:parameter-t-alpha-beta}
t_k=\sigma_k/\sigma_0,\quad \alpha_{k+1}=c_\alpha/\tau_k,\quad \beta_{k+1}=\gamma_0c_\beta/\sigma_k.
\end{equation}
We first verify that $\{x_k,y_k\}_{k=0}^{K-1}$ and $\{\tau_k,\sigma_k,\theta_k\}_{k=0}^{K-1}$ together with $\{\alpha_k,\beta_k,t_k\}$ as stated in~\eqref{eq:parameter-t-alpha-beta} satisfy \eqref{eq:step-size-condition} in Assumption~\ref{assum:step} for $k=0,\hdots,K-1$. To show that \eqref{eq:step-size-condition-Ek} and \eqref{eq:step-size-condition-theta} hold for $k=0,\hdots,K-1$, we revisit the proofs of Lemmas~\ref{lem:step-size-seq} and \ref{lem:parameter}. 
Note that from the  step-size update rules we have that $\sigma_k=\gamma_k\tau_k$, {$\theta_k=\sigma_{k-1}/\sigma_k$} and $\gamma_{k+1}=\gamma_k(1+\mu\tau_k)$, for $k=0,\hdots,K-1$. Therefore, Lemma \ref{lem:step-size-seq} implies that \eqref{eq:step-size-condition-theta} holds for $k=0,\hdots,K-1$.

{Note that for any $x\in\dom f$ and $y,y'\in\cY$, we have
\begin{align}
\label{eq:Lxy}
\norm{\grad_x\Phi(x,y)-\grad_x\Phi(x,y')}=\norm{\grad G(x)^\top(y-y')}\leq\norm{\grad G(x)}\norm{y-y'}.
\end{align}
Moreover, using Lipschitz continuity of $G$ \rev{--see Assumption \ref{assum:optimization}}, one can easily show that if $\cK=\reals^m_+$, then $\grad G$ is bounded.\footnote{\rev{This is an immediate extension of the real-valued case, i.e., if a differentiable, convex function $g:\reals^n\to\reals$ is $L$-Lipschitz with respect to $\norm{\cdot}$, then $\norm{\grad g(\cdot)}_{*}\leq L$.}} Thus, it  follows from \eqref{eq:Lxy} that whenever $\grad G$ is bounded on $\dom f$, the Lipschitz constant $L_{xy}$ exists. Next, define $\bar{L}_{xx}\triangleq {L_g+L_G \bar{B}}$ and note that 
$\norm{\grad_x\Phi(x_{k+1},y_k)-\grad_x\Phi(x_k,y_k)}\leq \bar{L}_{xx}\norm{x_{k+1}-x_k}$, for $k=0,\hdots,K-1$,} where we used \eqref{eq:L_xx-co} and the induction hypothesis, i.e., $\norm{y_k}\leq \bar{B}$ for $k= 0,\hdots,K-1$.
Therefore, for $E_k(\cdot,\cdot)$ defined in~\eqref{eq:modified_Ek}, the following upper bound on $E_k\triangleq E_k(x_{k+1},y_{k+1})$ holds for $k=0,\hdots,K-1$:
\begin{align*}
E_k\leq & \Big(\frac{L_{xy}^2}{\alpha_{k+1}}-\frac{1}{\sigma_k}\Big)\bD_\cY(y_{k+1},y_k) +\Big(\frac{{\bar{L}_{xx}^2}}{\beta_{k+1}}+\theta_k(\alpha_k+\beta_k)-\frac{1}{\tau_k}\Big)\bD_\cX(x_{k+1},x_k)\\
\leq &\Big(\frac{1}{c_\alpha}{L_{xy}^2}\tau_k-\frac{1}{\gamma_k\tau_k}\Big)\bD_\cY(y_{k+1},y_k) +\Big({\bar{L}_{xx}^2}\tau_k\frac{\gamma_k}{c_\beta\gamma_0}-\frac{1-(c_\alpha+c_\beta)}{\tau_k}\Big)\bD_\cX(x_{k+1},x_k),
\end{align*}
where in the last inequality, we used $\sigma_k=\gamma_k\tau_k$ and {the fact that $\{\gamma_k\}_{k\geq 0}$ is a non-decreasing sequence such that $\gamma_k\geq \gamma_0$, for $k\geq 0$. Therefore, given $\delta\in [0,1)$, one can conclude that for all $k=0,\hdots,K-1$, $E_k\leq -\delta[\bD_\cX(x_{k+1},x_k)/\tau_k+\bD_\cY(y_{k+1},y_k)/\sigma_k]$ holds for any $\tau_k\in(0,\Psi_3\sqrt{\gamma_0/\gamma_k}]$
where $\Psi_3\triangleq \min\{\frac{\sqrt{c_\alpha(1-\delta)}}{L_{xy}\sqrt{\gamma_0}},\frac{\sqrt{c_\beta(1-(c_\alpha+c_\beta+\delta))}}{L_{xx}}\}$;} hence, after a finite number steps the backtracking terminates and we have $\tau_k\geq \eta\Psi_3\sqrt{\gamma_0/\gamma_k}$. At this point, we verified that \eqref{eq:step-size-condition} holds for $k=0,\hdots,K-1$.

Following the same proof lines as in Theorem \ref{thm:general-bound} with $x$ and $y$ being switched one can easily derive the following result:
{\small
\begin{align}\label{eq:bound-lagrange-switch}
{\cL(x_{k+1},y)-\cL(x,y_{k+1})\leq} Q_k(z) - R_{k+1}(z) + E_k,
\end{align}
which holds for $k=0,\ldots,K-1$, where $Q_k(z)$, $R_k(z)$ and $E_k$ are defined similarly:
\begin{align*}
& Q_k(z) \triangleq \frac{1}{\tau_k}\bD_{\cX}(x,x_k)+\frac{1}{\sigma_k}\bD_{\cY}(y,y_k)+\frac{\theta_k}{2\alpha_k}\norm{\grad_x\Phi(x_k,y_{k})-\grad_x\Phi(x_{k},y_{k-1})}^2 \nonumber\\
&+\theta_k\fprod{{q_k},x-x_k}+\frac{\theta_k}{2\beta_k}\norm{\grad_x\Phi(x_k,y_{k-1})-\grad_x\Phi(x_{k-1},y_{k-1})}^2, \nonumber\\
& R_{k+1}(z)\triangleq \frac{1}{\tau_{k}}\bD_{\cX}(x,x_{k+1})+\frac{\mu}{2}\norm{x-x_{k+1}}^2_{\cX}+\frac{1}{2\beta_{k+1}}\norm{{\grad_x\Phi(x_{k+1},y_{k})}-\grad_x\Phi(x_k,y_k)}^2\nonumber\\
&+\frac{1}{\sigma_{k}}\bD_{\cY}(y,y_{k+1}) +\fprod{{q_{k+1}},x-x_{k+1}}+\frac{1}{2\alpha_{k+1}}\norm{\grad_x\Phi(x_{k+1},y_{k+1})-{\grad_x\Phi(x_{k+1},y_{k})}}^2,\nonumber \\
& E_k = E_k(x_{k+1},y_{k+1}), \nonumber 
\end{align*}}%
and {$q_k\triangleq\grad_x\Phi(x_k,y_k)-\grad_x\Phi(x_{k-1},y_{k-1})$ for $k=0,\hdots,K-1$}. It follows from \eqref{eq:step-size-condition-theta} that multiplying \eqref{eq:bound-lagrange-switch} with $t_k$ and summing it over $k=0,\hdots,K-1$, we get
\rev{
\begin{eqnarray}\label{eq:final-bound-backtracking-switch}
\lefteqn{T_K(\cL(\bar{x}_{K},y)-\cL(x,\bar{y}_{K}))\leq t_0Q_0(z)-t_{K-1}R_K(z)\leq}  \\
& & 
{\Delta(x,y)}-t_{K-1}\Big[\Big(\frac{1}{\tau_{K-1}}-(\alpha_K+\beta_K)\Big)\bD_{\cX}(x,x_{K}) +\frac{1}{\sigma_{K-1}}\bD_{\cY}(y,y_{K}) \Big], \nonumber
\end{eqnarray}}%
\rev{where $\Delta(x,y)$ is defined in \eqref{eq:delta}.} Using \eqref{eq:parameter-t-alpha-beta} at $k=K-1$ we get
\rev{
\begin{align}
\frac{1}{\tau_{K-1}}-(\alpha_K+\beta_K)=\frac{\gamma_{K-1}}{\sigma_{K-1}}-\Big(\frac{\gamma_{K-1}c_\alpha+\gamma_0c_\beta}{\sigma_{K-1}}\Big)\geq\frac{\gamma_{K-1}}{\sigma_{K-1}}(1-c_\alpha-c_\beta) \geq 0,
\end{align}}%
where \rev{the equality follows from $\sigma_{K-1}=\gamma_{K-1}\tau_{K-1}$, and} in the first inequality we used the fact that 
{$\gamma_{k+1}\geq\gamma_k>0$ for $k\geq 0$} and $c_\alpha+c_\beta+\delta\in (0,1]$. Therefore, \eqref{eq:final-bound-backtracking-switch} implies that
{\small
\begin{equation}\label{eq:final-bound-rate}
\frac{\rev{\gamma_{K-1}}(1-(c_\alpha+c_\beta))}{\sigma_0}\bD_{\cX}(x,x_{K})+\frac{1}{\sigma_0}\bD_\cY(y,y_K)+T_K(\cL(\bar{x}_{K},y)-\cL(x,\bar{y}_{K}))\leq \Delta(x,y).
\end{equation}}%
Evaluating \eqref{eq:final-bound-rate} at $(x,y)=(x^*,y^*)$ and using 
 $\cL(\bar{x}_{K},y^*)-\cL(x^*,\bar{y}_{K})\geq 0$, we obtain that $\frac{1}{\sigma_0}\bD_\cY(y^*,y_K)\leq \Delta(x^*,y^*)$. Moreover, using $\bD_\cY(y^*,y_K)\geq \frac{1}{2}\norm{y^*-y_K}^2$ one can easily verify that $\norm{y_K}\leq \bar{B}$;
hence, the induction is complete.

Consider Part I, since $\mu=0$ it is clear that $\gamma_k=\gamma_0$, for $k\geq 0$; hence, we get $T_K= \sum_{k=0}^{K-1}\sigma_k/\sigma_0\geq \eta \frac{\Psi_3}{{\tau_0}}K$. Consider Part II, i.e., $\mu>0$ and $L_{yy}=0$, following the same proof lines of Lemma \ref{lem:k2-rate}, $\gamma_{k+1}=\gamma_k(1+\mu\tau_k)$ implies that $\gamma_{k+1}\geq \gamma_k+\mu\eta\Psi_3\sqrt{\gamma_0\gamma_k}$ which implies that $\gamma_k\geq (\frac{\Gamma}{3})^2k^2$ and $\sigma_k\geq \frac{\Gamma^2k}{3\mu}$ for $k\geq 0$, where $\Gamma=\mu\eta\Psi_3\sqrt{\gamma_0}$. Moreover, $\tau_k\sigma_k\geq \Gamma^2/\mu$, for $k\geq 1$. Therefore, $T_K=\sum_{k=0}^{K-1}\sigma_k/\sigma_0=\Omega({K^2})$, $\tau_K/\sigma_K=1/\gamma_K=\cO(1/K^2)$, and $\sigma_K=\Omega(1/\tau_K)$ for $K\geq 1$.

Similar to the proof in section \ref{sec:rate_APDB}, one can observe that for both Part I and II, $\tau_k\geq \eta\Psi_3\sqrt{\gamma_0/\gamma_k}$ and $\tau_k\leq \bar{\tau}\sqrt{\gamma_0/\gamma_k}$ for all $k\geq 0$, which implies a uniform bound $1+\log_{1/\eta}(\frac{\bar{\tau}}{\Psi_3})$ on the number of inner iterations. Moreover, the rate results follow from \eqref{eq:final-bound-rate} similar to the proof of Theorem~\ref{thm:main}.
\end{proof}
\section{Numerical Experiments}\label{sec:numeric}
In this section, we implement both APD and APDB for solving \rev{quadratically constrained quadratic problems with synthetic data, and two problems arising in machine learning, namely, the kernel matrix learning and regression with fairness constraints.} We compare them with other state-of-the-art methods. \sa{All experiments are performed on a machine running 64-bit Windows 10 with Intel i7-8650U @2.11GHz and 16GB RAM.}
\begin{remark}
Using convexity of $\Phi(\cdot,y)$, one can define another test function $\widetilde{E}_k(x,y)$ that upper bounds $E_k(x,y)$. Indeed, we define $\widetilde{E}_k(x,y)$ by replacing the term \sa{$\Phi(x,y)-\Phi(x_k,y)-\fprod{\grad_x\Phi(x_k,y),x-x_k}$ in the first line of \eqref{eq:Ek} with the inner product $\fprod{\grad_x\Phi(x,y)-\grad_x\Phi(x_k,y),x-x_k}$.} 
\sa{For backtracking, using $\widetilde{E}_k(x_{k+1},y_{k+1})$ in Line~\ref{algeq:test_function} of APDB instead of $E_k(x_{k+1},y_{k+1})$ leads to a stronger condition; but, in practice, we have found this condition numerically more stable.}
\end{remark}
\subsection{Kernel Matrix Learning} We test the implementation of our method for solving the kernel matrix learning problem discussed in Section~\ref{sec:intro} for classification.
In particular, given a set of kernel matrices $\{K_\ell\}_{\ell=1}^M\subset\mathbb{S}^n_{+}$, consider the problem in~\eqref{eq:kernel_learn_simple}. When $\lambda>0$ and $C=\infty$, the objective is to find a kernel matrix $K^*\in\cK\triangleq\{\sum_{\ell=1}^M\eta_\ell K_\ell:\ \eta\geq \mathbf{0}\}$ that achieves the best training error for $\ell_2$-norm soft margin SVM, and when $\lambda=0$ and $C>0$, the objective is to find a kernel matrix $K^*\in\cK$ that gives the best performance for $\ell_1$-norm soft margin SVM. Once $(\alpha^*,\eta^*)$, a saddle point for \eqref{eq:kernel_learn_simple}, is computed using the training set $\cS$, one can construct $K^*=\sum_{\ell=1}^M{\eta_\ell^*}K_\ell$ and predict unlabeled data in the test set $\cT$ using the model $\cM:\reals^m\rightarrow\{-1,+1\}$ such that the predicted label of $\ba_i$ is $
\cM(\ba_i)={\rm sign}\Big(\sum_{j\in\cS}b_j\alpha^*_j K^*_{ji}+\gamma^*\Big),
$
for all $i\in\cT$ where for $\ell_1$ soft margin SVM, $\gamma^*=b_{i^*}-\sum_{j\in\cS}b_j\alpha^*_jK^*_{ji^*}$ for some $i^*\in\cS$ such that $\alpha^*_{i^*}\in (0,C)$, and for $\ell_2$ soft margin SVM, $\gamma^*=b_{i^*}(1-\lambda\alpha^*_{i^*})-\sum_{j\in\cS}b_j\alpha^*_jK^*_{ji^*}$ for some $i^*\in\cS$ such that $\alpha^*_{i^*}>0$. Note that \eqref{eq:kernel_learn_simple} is a special case of \eqref{eq:original-problem} for $f$, $\Phi$ and $h$ chosen as follows: let $y_\ell\triangleq\frac{\eta_\ell r_\ell}{c}$ for each $\ell$ and define $h(y)=\ind{\Delta}(y)$ where $y=[y_\ell]_{\ell=1}^M\in\reals^M$ and $\Delta$ is an $M$-dimensional unit simplex; $\Phi(x,y)= -2\be^\top x + \sum_{\ell=1}^M \frac{c}{r_\ell}y_\ell x^\top G(K^{tr}_\ell)x+\lambda\norm{x}_2^2$, and $f(x)=\ind{X}(x)$ where $X=\{x\in\reals^{n_{tr}}:~0\leq x\leq C,~ \fprod{\bb, x}=0\}$.
\subsubsection{{APD} vs \textbf{Mirror-prox} for soft-margin SVMs}
In this experiment, we 
compared our method against Mirror-prox, the primal-dual algorithm proposed by He et al.~\cite{he2015mirror}.
We used four different data sets available in UCI repository: \texttt{Ionosphere} (351 observations, 33 {features}), \texttt{Sonar} (208 observations, 60 {features}), \texttt{Heart} (270 observations, 13 {features}) and \texttt{Breast-Cancer} (608 observations, 9 {features}) with three given kernel functions ($M=3$); polynomial kernel function $k_1(\ba,\bar{\ba})=(1+\ba^\top\bar{\ba})^2$, Gaussian kernel function $k_2(\ba,\bar{\ba})=\exp(-0.5(\ba-\bar{\ba})^\top(\ba-\bar{\ba})/0.1)$, and linear kernel function $k_3(\ba,\bar{\ba})=\ba^\top\bar{\ba}$ to compute $K_1, K_2, K_3$ respectively. All the data sets are normalized such that each feature column is mean-centered and divided by its standard deviation. 
For $\ell_2$-norm soft margin we set $\lambda=1$, for $\ell_1$-norm soft margin SVM we set $C=1$ and for both SVMs $c=\sum_{\ell=1}^3 r_\ell$, where $r_\ell={\rm trace}(K_\ell)$ for $\ell=1,2,3$. 
The kernel matrices are normalized as in \cite{lanckriet2004learning}; {thus, $\diag(K_\ell)=\ones$, $r_\ell=n_{tr}+n_t$ and $c/r_\ell=3$ for $\ell=1,2,3$.}

We tested four different implementations of the APD algorithm: 
we will refer to the constant step version of APD, stated in Part I of the main result in Theorem~\ref{thm:main}, as {\bf APD1};
and we refer to the adaptive step version of APD, stated in Part II of the main result, as {\bf APD2}. Finally, we also implemented a variant of {\bf APD2} with periodic restarts, and we call it {\bf APD2-restart}.
The {\bf APD2-restart} method is implemented simply by restarting the algorithm periodically after every 500 iterations and using the most current iterate as the initial solution for the next call of {\bf APD2}. {All the algorithms are initialized from $x_0=\mathbf{0}$ and $y_0=\frac{1}{M}\ones$.}

The 
results reported are the average values over 10 random replications. In each replication, \%80 of the dataset is selected uniformly at random, and used for training; the rest of data set (\%20) is reserved as test data to calculate the test set accuracy (TSA), which is defined as the fraction of the correctly labeled data in the test set. The algorithms are compared in terms of relative error for the function value ($|\cL(x_k,y_k)-\cL^*|/|\cL^*|$) and for the solution ($\norm{x_k-x^*}_2/\norm{x^*}_2$), where $(x^*,y^*)$ denotes a saddle point for the problem of interest, i.e., \eqref{eq:kernel-l1} or \eqref{eq:kernel-l2}, and $\cL^*\triangleq \cL(x^*,y^*)$. To compute $(x^*,y^*)$, we called MOSEK through CVX \cite{grant2008cvx}.

{\bf $\ell_1$-norm Soft Margin SVM:}
Consider the following equivalent reformulation of $\ell_1$-norm soft margin problem:\vspace*{-4mm}
\begin{align}
&\min_{\substack{\rev{x:\ 0\leq x\leq C\be} \\ \fprod{\bb,x}=0}}\max_{y\in\Delta}
-2x^\top \be + \sum_{\ell=1}^M \frac{c}{r_\ell}~y_\ell~ x^\top G(K_\ell^{tr})x, \label{eq:kernel-l1}
\end{align}
Note that \eqref{eq:kernel-l1} is merely convex in $x$ and linear in $y$; therefore, we only used {\bf APD1}. 
Let $\norm{\cdot}$ denote the spectral norm; the Lipschitz constants defined in \eqref{eq:Lxx} and \eqref{eq:Lipschitz_y} can be set as $L_{xx}\triangleq 6\max_{\ell=1,2,3}\{\norm{G(K_\ell)}\}$, $L_{y y}=0$, $L_{y x}\triangleq 6\sqrt{3}C\max_{\ell=1,2,3}\{\norm{G(K_\ell)}\}$. Recall that the step-size of {\bf Mirror-prox} is determined by Lipschitz constant $L$ of $\grad \Phi$, and which can be set as $L=\sqrt{L_{xx}^2+L_{x y}^2+L_{yx}^2+L_{yy}^2}$, where $L_{xy}$ is defined similarly as $L_{yx}$ in~\eqref{eq:Lipschitz_y}, and for \eqref{eq:kernel-l1} one can take $L_{xy}=L_{yx}$.

In these experiments on $\ell_1$ soft margin problems, {\bf APD1} outperformed {\bf Mirror-prox} on all four data sets. In particular, {\bf APD1} and {\bf Mirror-prox} are compared in terms of relative errors for function value and for solution in Figures~\ref{fig:l1-obj} and \ref{fig:l1-sol} respectively. In these figures, relative 
errorr are plotted against the number of iterations.
In this experiment we observed that for fixed number of iterations \rev{$k$}, the run time for {\bf Mirror-prox} is at least twice of {\bf APD1} run time -- while { APD} requires one primal-dual prox operation, 
{\bf Mirror-prox} needs \emph{two} primal-dual prox operations at each iteration. 
It is worth mentioning that for \texttt{Ionosphere} and \texttt{Sonar} data sets, 1500 iterations of {\bf APD1} took 
roughly the same time as MOSEK required to solve the problem, and within 1500 iterations {\bf APD1} was able to generate a decent approximate solution with relative error less than $10^{-4}$ and with a high TSA value; 
on the other hand, {\bf Mirror-prox} was not able to produce such good quality solutions in a similar 
amount of time.
The average TSA of the optimal solution $(x^*,y^*)$, computed by MOSEK, are 93.81, 84.76, 84.07, 96.79 percent for \texttt{Ionosphere}, \texttt{Sonar}, \texttt{Heart} and \texttt{Breast-Cancer} data sets, respectively. Note that TSA is not necessarily increasing in the number of iterations \sa{$k$}, e.g., {\bf APD1} iterates at \sa{$k=1000$ and $k=2000$} have 85.95\% and 84.76\% TSA values, respectively, for \texttt{Sonar} data set -- note the optimal solution $(x^*,y^*)$ has 84.76\% {TSA -- see Table~\ref{table:l1} for details}. This is a well-known phenomenon and is related to over fitting; in particular, the model’s ability to generalize can weaken as it begins to overfit the training data.
\begin{table*}[h]
\centering
\resizebox{\columnwidth}{!}{%
\begin{tabular}{|c|c|c c c|c c c|c c c|c c c|}
\hline
\multicolumn{2}{| c |}{Iteration~\#}
& \multicolumn{3}{| c |}{k=1000} & \multicolumn{3}{| c |}{k=1500} & \multicolumn{3}{| c |}{k=2000} & \multicolumn{3}{| c |}{k=2500} \\
\hline
\mbox{} & \mbox{} &
& Rel. & & & Rel.& & & Rel.& & & Rel. & \\
Method & Data Set
& Time & subopt. & TSA & Time & error & TSA & Time & error & TSA & Time & error & TSA \\
\hline
\multirow{2}{*}{\bf APD1}
&Ionosphere
&   0.34& 5.6e-05   &	93.80	
&	0.50& 9.3e-06	&	93.80	
&	0.67& 1.6e-06	&	93.80	
&	0.84& 3.6e-07	&	93.80	\\
& Sonar
&	0.16& 4.6e-04	&	85.95	
&	0.24& 4.1e-05	&	84.52	
&	0.32& 2.1e-06	&	84.76	
&	0.41& 9.7e-08	&	84.76	\\
& Heart
&	0.31& 1.1e-06	&	83.89	
&	0.46& 3.6e-07	&	83.89	
&   0.60& 1.1e-07   &	83.89	
&	0.75& 3.6e-08	&	83.88	\\
&	Breast-cancer
&	0.98& 5.5e-03 	&	96.86	
&	1.59& 1.0e-03 	&	96.79	
&	2.19& 2.2e-04 	&	96.72	
&	2.79& 6.3e-05 	&   96.71\\
\hline
\multirow{ 2}{*}{\bf Mirror-}
&Ionosphere
&	0.84& 1.3e-04 	&	93.80	
&	1.30& 2.6e-05 	&	93.80	
&	1.80& 6.3e-06 	&	93.80	
&	2.28& 1.5e-06 	&	93.80	\\
& Sonar
& 0.47	& 4.3e-03 &	85.24	
& 0.71	& 3.4e-04 &	85.48	
& 0.96	& 2.9e-05 &	84.52	
& 1.24	& 2.9e-06 &	84.76	\\
{\bf prox}& Heart
&   0.75 & 1.9e-06	&	83.89	
&	1.16 & 7.5e-07	&	83.89	
&	1.60 & 2.9e-07	&	83.89	
&	2.04 & 1.2e-07	&	83.88	\\
& Breast-cancer
&  2.56 & 1.1e-02	&	96.86	
&  4.24 & 2.6e-03	&	96.86	
&  5.91 & 6.8e-04	&	96.72	
&  7.54	& 2.0e-04   &	96.71	\\
\hline
\end{tabular}%
}
\caption{$\ell_1$-norm soft margin: runtime (sec), relative error for $\cL(x_k,y_k)$ and TSA (\%) for {\bf APD1} and {\bf Mirror-prox} at iteration $k\in\{1000, 1500, 2000, 2500\}$.}
\vspace*{-5mm}
\label{table:l1}
\end{table*}

\begin{figure}[h!]
\centering
\includegraphics[scale=0.16]{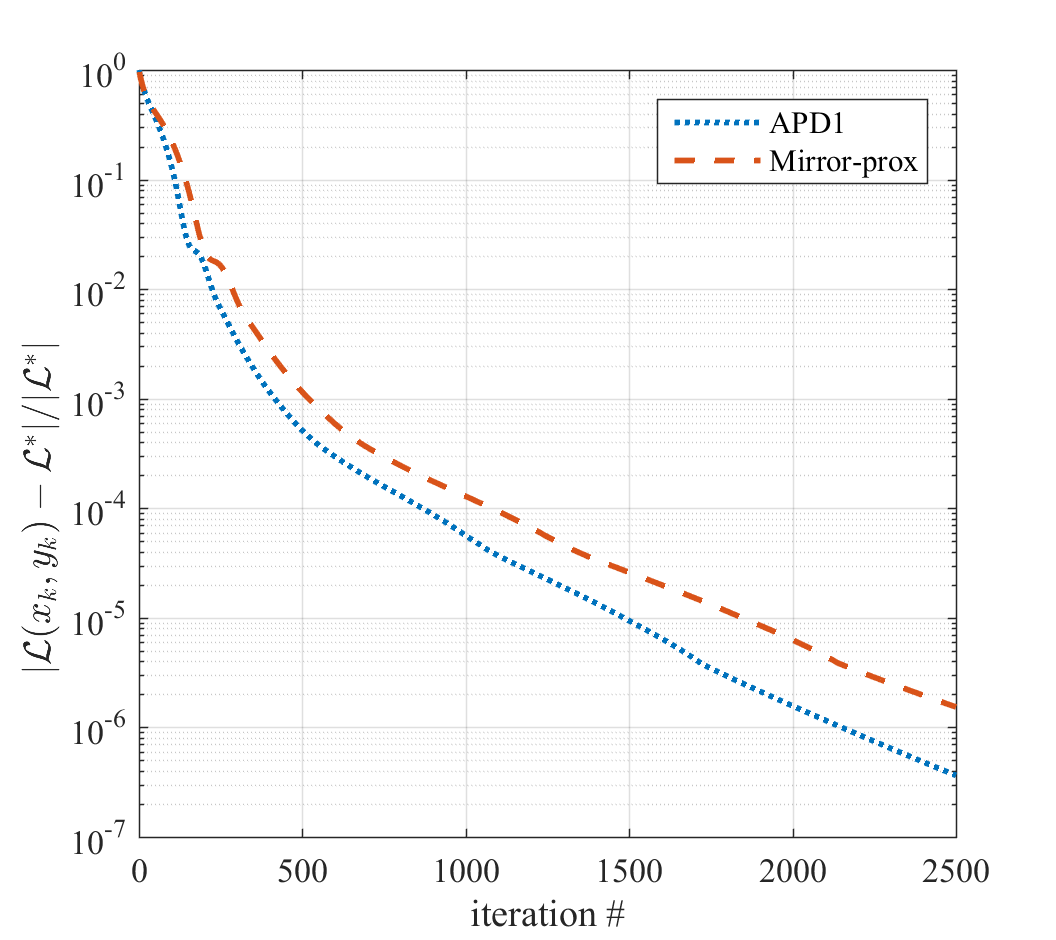}
\includegraphics[scale=0.16]{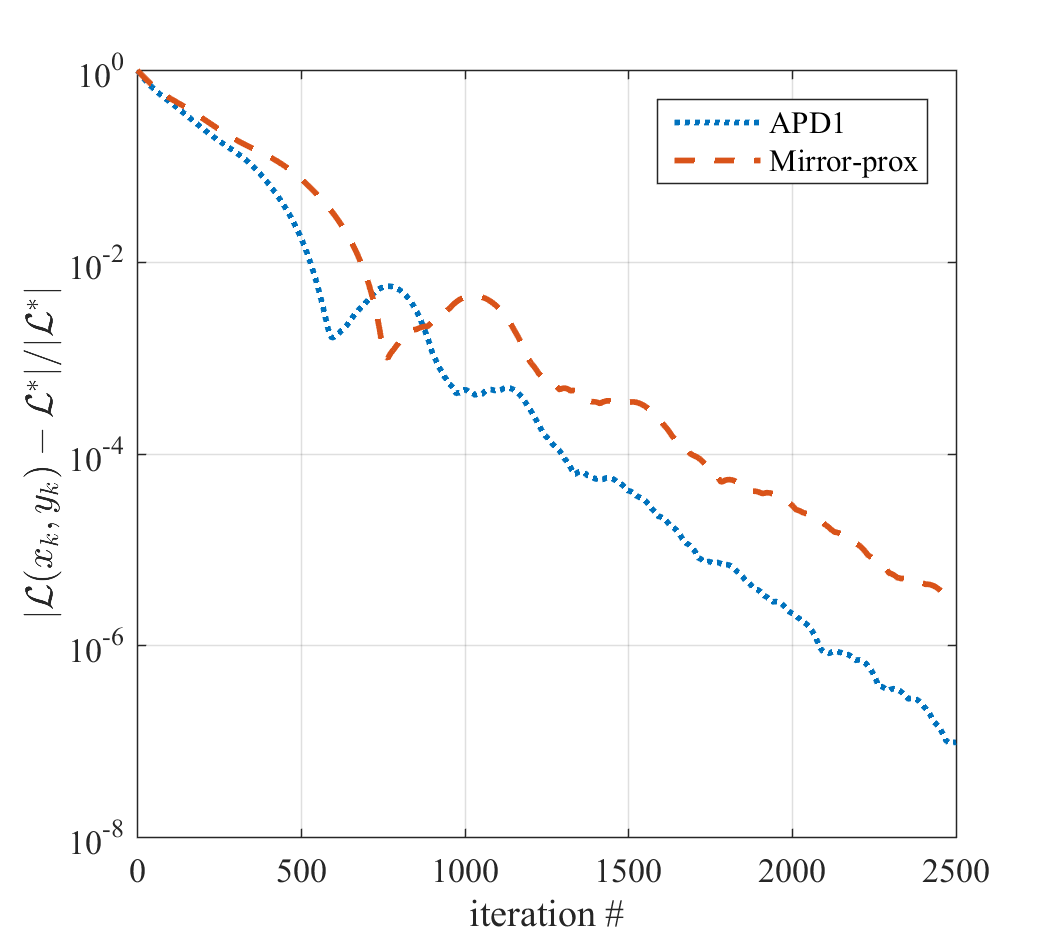}
\includegraphics[scale=0.16]{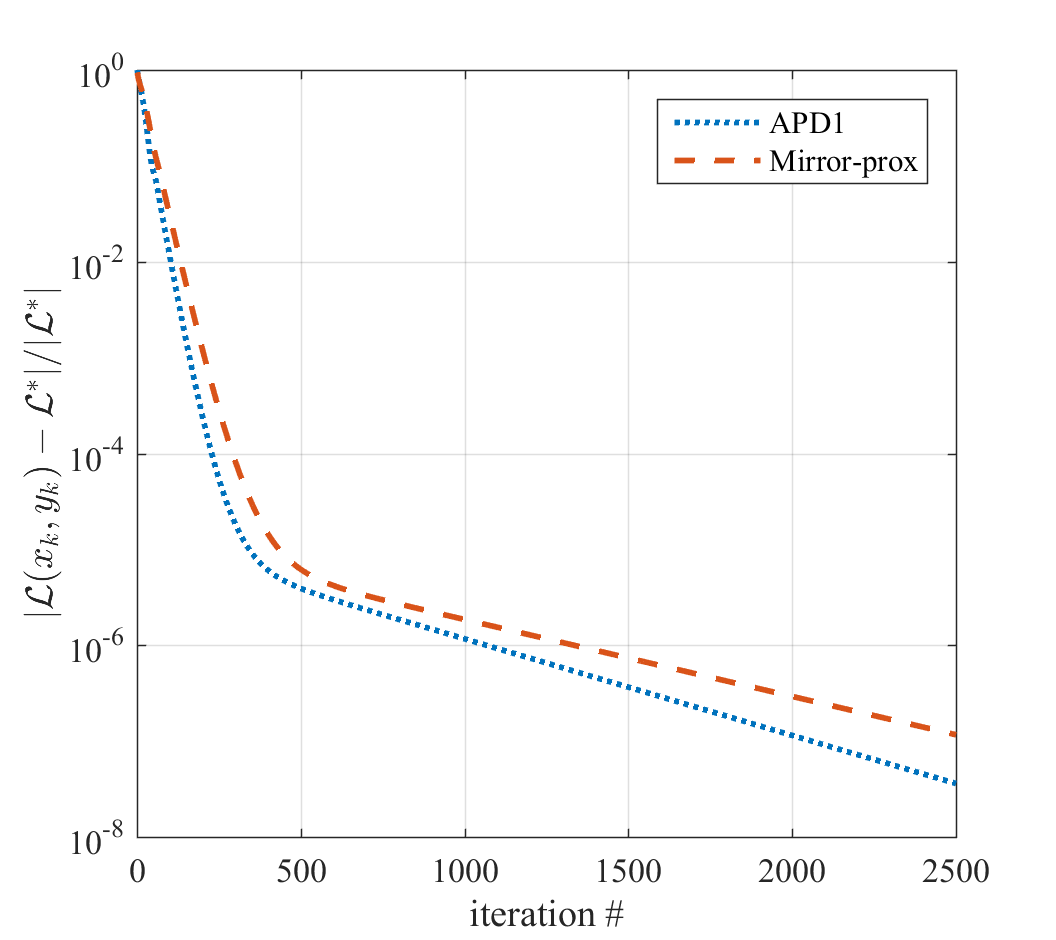}
\includegraphics[scale=0.16]{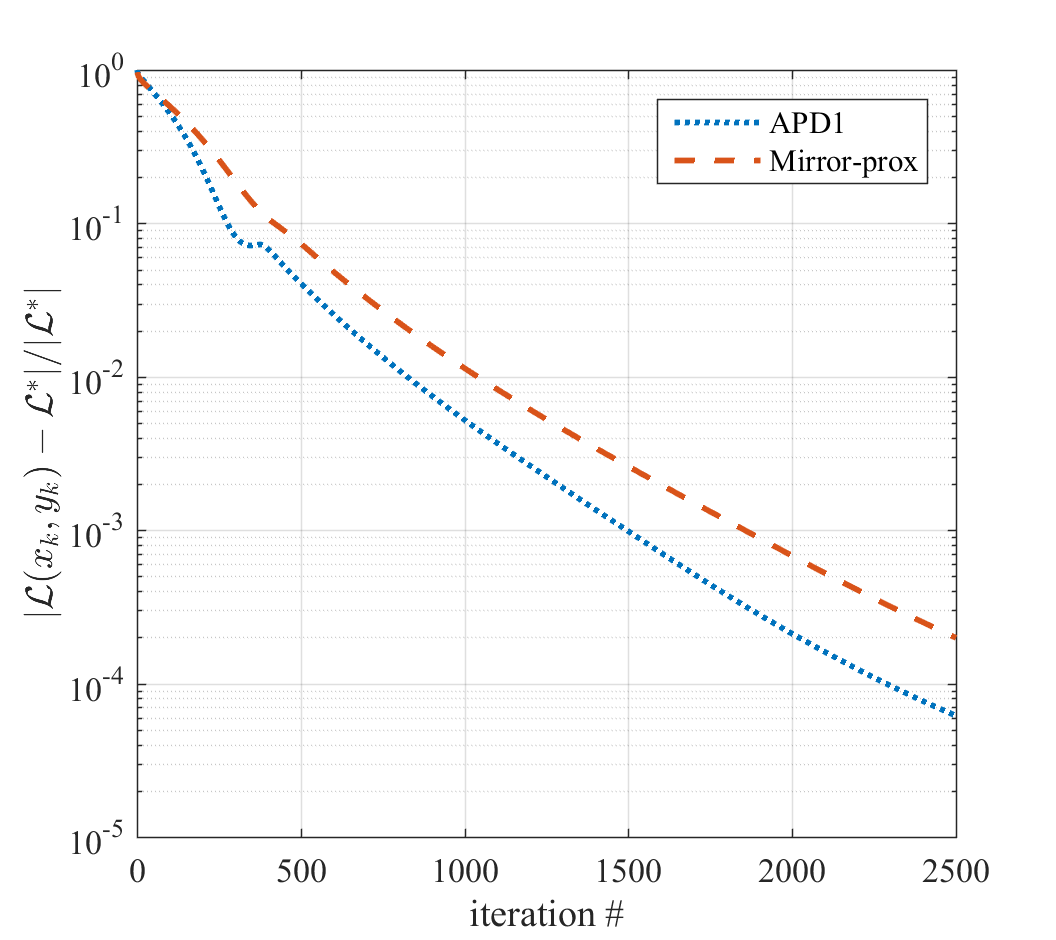}
\caption{$\ell_1$-norm soft margin: {\bf APD1} vs {\bf Mirror-prox} in terms of relative error for $\cL(x_k,y_k)$. The plots from left to right: 
\texttt{Ionosphere}, \texttt{Sonar}, \texttt{Heart}, \texttt{Breast-Cancer}.}
\vspace*{-2mm}
\label{fig:l1-obj}
\end{figure}
\begin{figure}[h!]
\centering
\includegraphics[scale=0.16]{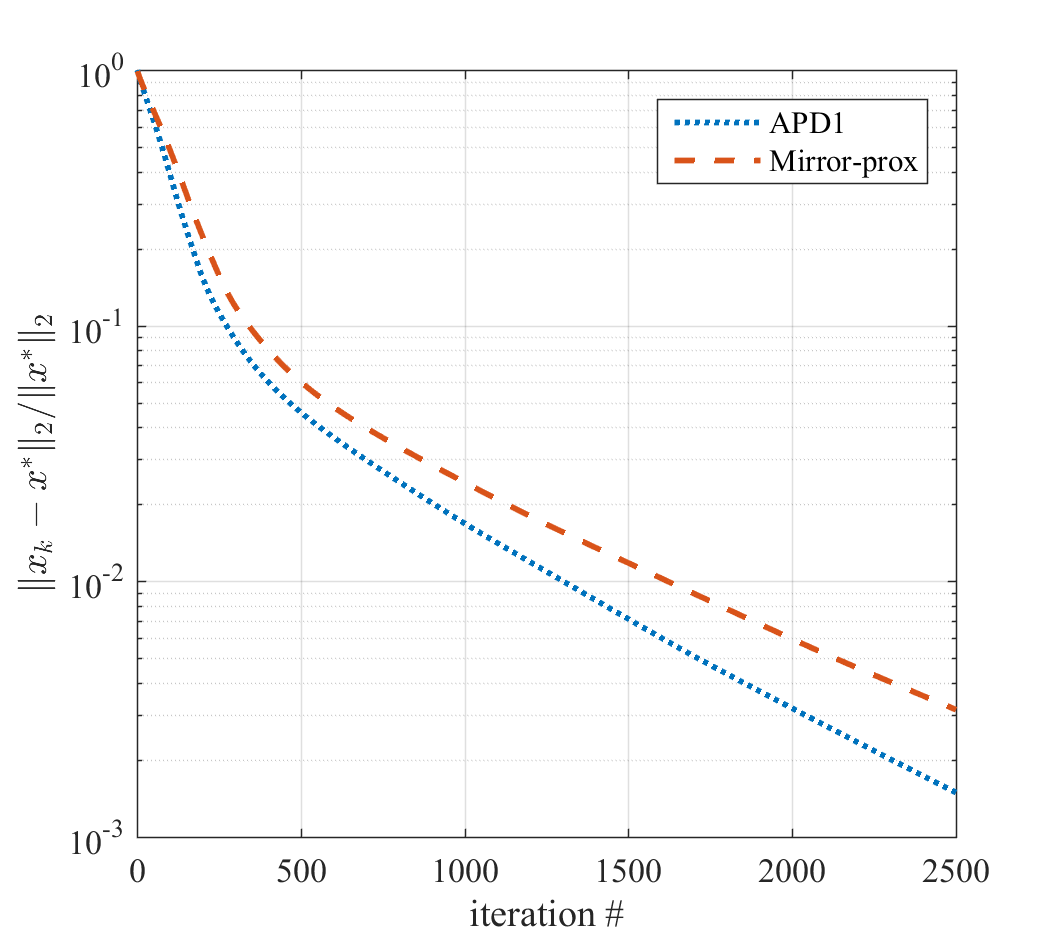}
\includegraphics[scale=0.16]{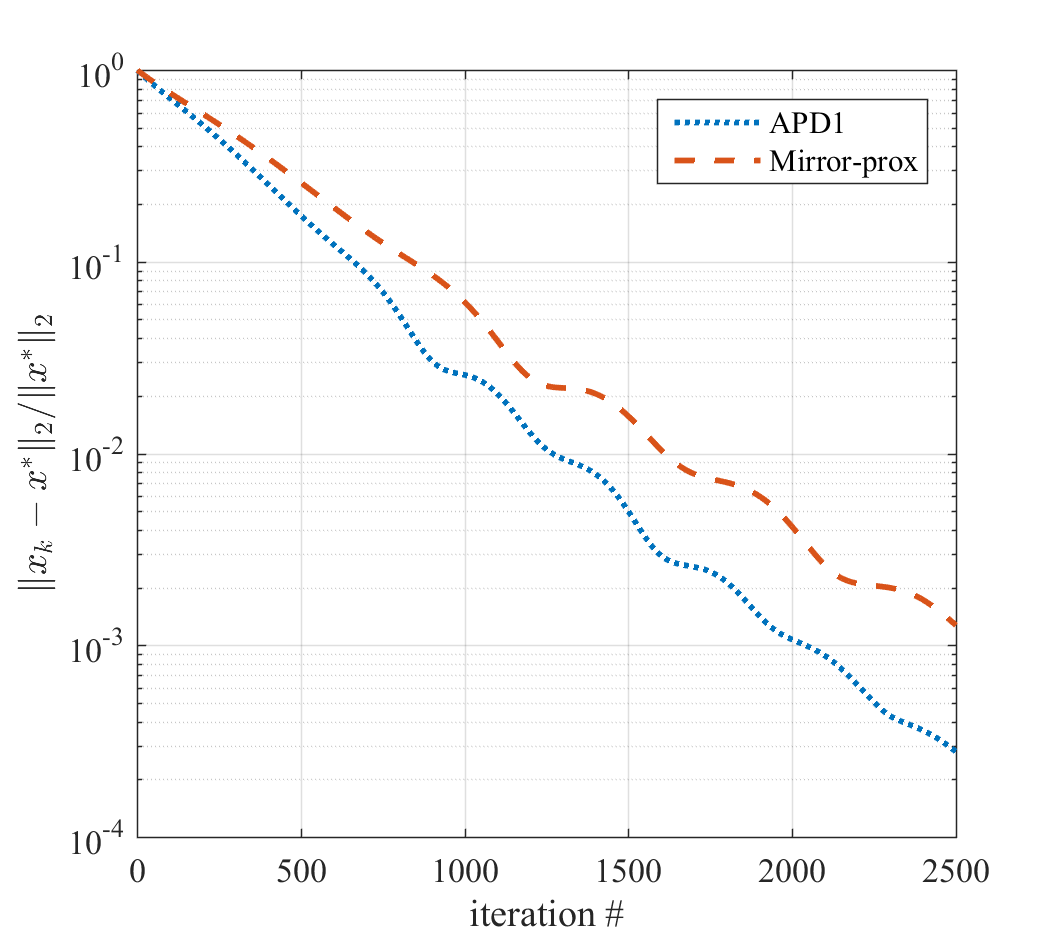}
\includegraphics[scale=0.16]{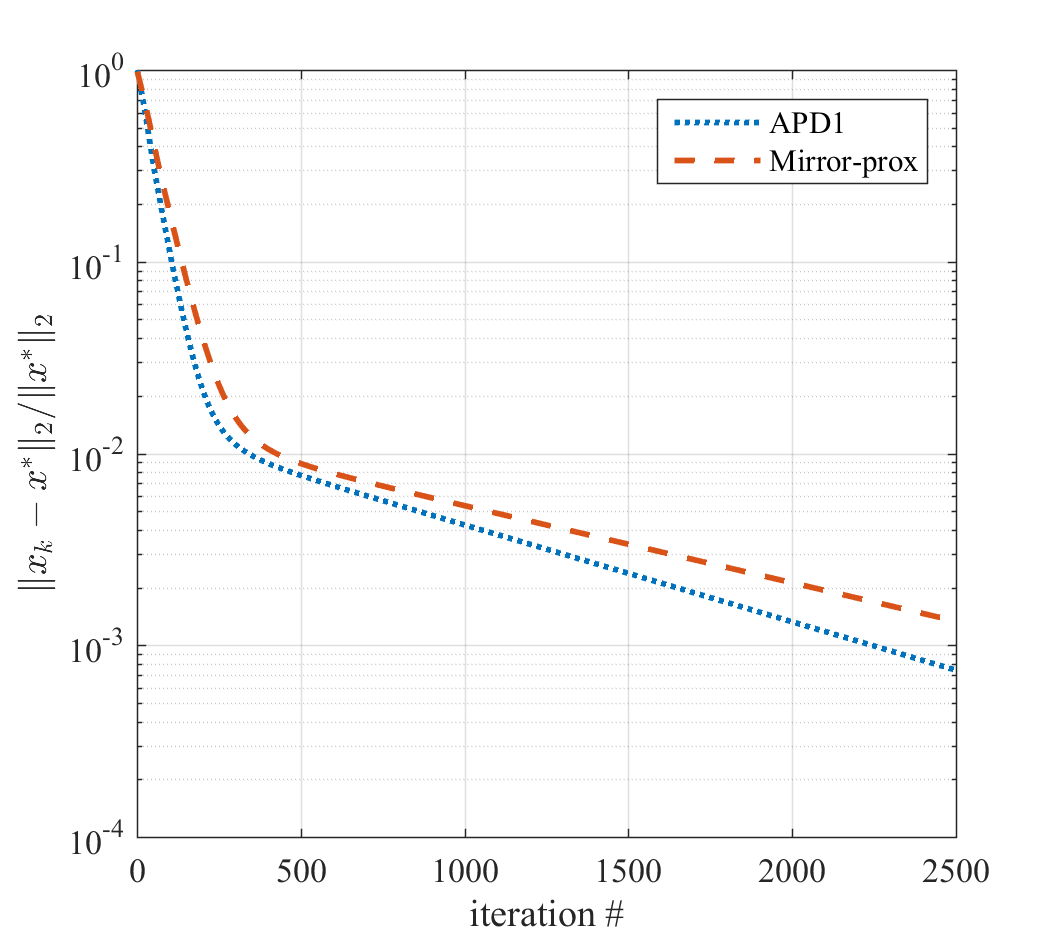}
\includegraphics[scale=0.16]{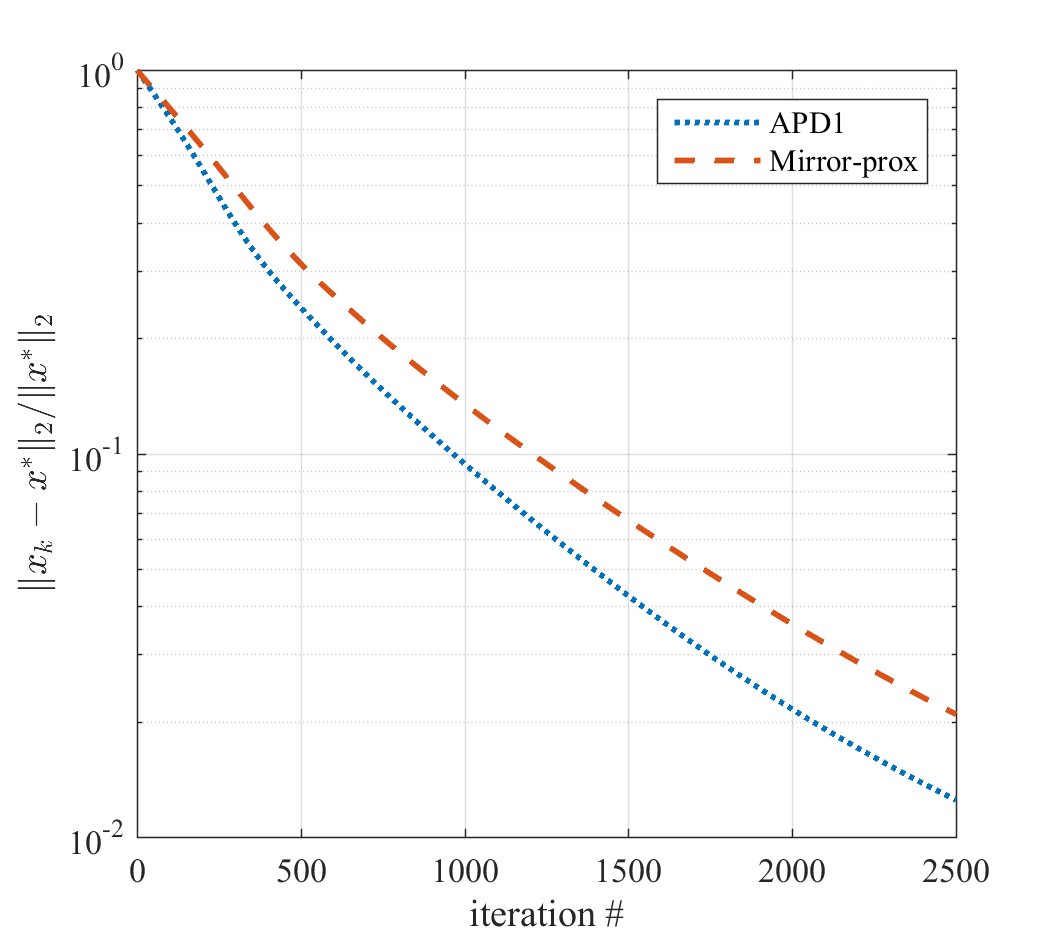}
\caption{$\ell_1$-norm soft margin: {\bf APD1} vs {\bf Mirror-prox} in terms of relative error for $x_k$. The plots from left to right: 
\texttt{Ionosphere}, \texttt{Sonar}, \texttt{Heart}, \texttt{Breast-Cancer}.}
\vspace*{-5mm}
\label{fig:l1-sol}
\end{figure}
{\bf $\ell_2$-norm Soft Margin SVM:}
Consider the following equivalent reformulation of $\ell_2$-norm soft margin problem:
\begin{align}
&\min_{\substack{x\geq 0\\ \fprod{\bb,x}=0}}\max_{y\in\Delta}
-2x^\top \be + \sum_{\ell=1}^M \frac{c}{r_\ell}~y_\ell~x^\top G(K_\ell^{tr})x+\lambda\norm{x}_2^2. \label{eq:kernel-l2}
\vspace*{-2mm}
\end{align}
Since \eqref{eq:kernel-l2} is strongly convex in $x$ and linear in $y$, we implement both {\bf APD2} and {\bf APD2-restart} methods in addition to {\bf APD1}. 
Due to strong convexity, $\norm{x^*}_2$ can be bounded depending on $\lambda>0$ and the Lipschitz constants can be computed similarly as in $\ell_1$-norm soft margin problem. In these experiments on $\ell_2$ soft margin problems, {\bf APD2-restart} outperformed all other methods
on all four data sets. In particular, {\bf APD1}, {\bf APD2}, {\bf APD2-restart} and {\bf Mirror-prox} are compared in terms of relative errors for function value and for solution in Figures~\ref{fig:l2-obj} and \ref{fig:l2-sol} respectively. Given an accuracy level, both {\bf APD1} and {\bf APD2-restart} can compute a solution with a given accuracy requiring much fewer iterations than {\bf Mirror-prox} needs.
In addition, 
we observed that for fixed number of iterations the run time for {\bf Mirror-prox} is almost twice the run time for any {APD} implementation -- see Table~\ref{table:l2}, {and for runtime comparison on larger size problems, see Section~\ref{sec:experiment-ipm}}.
Interpreting the results in Figures~\ref{fig:l2-obj} and \ref{fig:l2-sol}, and computational time, we conclude that {APD} implementations can compute a solution with a given accuracy in a significantly lower time than {\bf Mirror-prox} requires. For instance, consider the results for \texttt{Sonar} data set in Figure~\ref{fig:l2-obj}, to compute a solution with $|\cL(x_k,y_k)-\cL^*|/|\cL^*|<10^{-6}$, {\bf APD2-restart} requires 1000 {iterations}, on the other hand, {\bf Mirror-prox} needs around 2000 iterations; hence, {\bf APD2-restart} can compute it in 1/4 of the run time for {\bf Mirror-prox}. This effect is more apparent when these methods are compared on larger scale problems, e.g., see Figure~\ref{fig:largedata-sido}.

\begin{table*}[h]
\centering
\resizebox{\columnwidth}{!}{%
\begin{tabular}{|c|c|c c c|c c c|c c c|c c c|}
\hline
\multicolumn{2}{| c |}{Iteration~\#}
& \multicolumn{3}{| c |}{k=1000} & \multicolumn{3}{| c |}{k=1500} & \multicolumn{3}{| c |}{k=2000} & \multicolumn{3}{| c |}{k=2500} \\
\hline
\mbox{} & \mbox{} &
& Rel. & & & Rel. & & & Rel. & & & Rel. & \\
Method & Data Set
& Time & error & TSA & Time & error & TSA & Time & error & TSA & Time & error & TSA \\
\hline
\multirow{ 2}{*}{\bf APD1}
&	Ionosphere	
&	0.29    &	6.2e-07	&   93.9	
&   0.45    &	1.6e-06	&	93.9
&	0.60    &	1.6e-06	&   93.9	
&	0.75    &	1.6e-06	&	93.9	\\
&	Sonar	
&	0.15    &	8.3e-05	&   82.6	
&   0.23    &	1.3e-06	&	82.4
&	0.31    &	2.3e-08	&   82.1	
&	0.40	&	3.6e-10	&	82.1	\\
&	Heart	
&	0.26    &	3.0e-11	&   83.0
&   0.38	&	3.0e-11	&   83.0
&	0.50    &	3.0e-11	&   83.0	
&	0.62	&	3.0e-11	&	83.0	\\
&	Breast-Cancer	
&	0.76	&	7.5e-05	&	96.9	
&	1.22	&	4.4e-06 &	96.9	
&	1.65    &	4.4e-07	&	96.9	
&	2.08	&	5.5e-08	&	96.9	\\

\hline
\multirow{2}{*}{\bf APD2}
&	Ionosphere	
&	0.32	&	1.6e-06	&	93.9	
&	0.48	&	1.6e-06	&	93.9	
&	0.65	&	1.6e-06	&	93.9		
&	0.82	&	1.6e-06	&	93.9 \\
&	Sonar	
&	0.15	&	4.1e-06	&	82.4	
&	0.23	&	2.0e-07	&	82.1	
&	0.31	&	9.5e-09	&	82.1		
&	0.38	&	9.4e-10	&	82.1 \\
&	Heart	
&	0.23	&	4.5e-11	&	83.0	
&	0.34	&	3.3e-11	&	83.0	
&	0.45	&	3.1e-11	&	83.0		
&	0.57	&	3.1e-11	&	83.0 \\
&	Breast-Cancer	
&	0.82	&	4.9e-06	&	96.9	
&	1.23	&	7.9e-07	&	96.9	
&	1.66	&	2.4e-07	&	96.9		
&	2.08	&	9.3e-08	&	96.9 \\
\hline
\multirow{ 2}{*}{\bf APD2-}
&	Ionosphere	
&	0.33	&	1.6e-06	&	93.9	
&	0.49	&	1.6e-06	&	93.9	
&	0.64	&	1.6e-06	&	93.9	
&	0.80	&	1.6e-06	&	93.9\\
&	Sonar	
&	0.15	&	1.0e-06	&	82.1	
&	0.23	&	2.1e-08	&	82.1	
&	0.32	&	6.5e-11	&	82.1	
&	0.40	&	9.9e-12	&	82.1\\
{\bf restart}&	Heart	
&	0.23	&	3.0e-11	&	83.0	
&	0.35	&	3.0e-11	&	83.0	
&	0.47	&	3.0e-11	&	83.0	
&	0.58	&	3.0e-11	&	83.0\\
&	Breast-Cancer	
&	0.80	&	6.9e-07	&	96.9	
&	1.23	&	1.7e-08	&	96.9	
&	1.67	&	5.7e-10	&	96.9	
&	2.10	&	7.2e-11	&	96.9\\
\hline
\multirow{ 2}{*}{\bf Mirror-}
&	Ionosphere	
&	0.60	&	1.6e-06	&	93.9	
&	0.89	&	1.6e-06	&	93.9	
&	1.20	&	1.6e-06	&	93.9	
&	1.51	&	1.6e-06	&	93.9\\
&	Sonar	
&	0.31	&	5.5e-04	&	82.4	
&	0.45	&	1.4e-05	&	82.4	
&	0.61	&	6.9e-07	&	82.1	
&	0.76	&	2.1e-08	&	82.1\\
{\bf prox}&	Heart	
&	0.45	&	3.0e-11	&	83.0	
&	0.68	&	3.0e-11	&	83.0	
&	0.91	&	3.0e-11	&	83.0	
&	1.14	&	3.0e-11	&	83.0\\
&	Breast-Cancer	
&	1.48	&	2.0e-04	&	96.9	
&	2.35	&	1.4e-05	&	96.9	
&	3.23	&	1.6e-06	&	96.9	
&	4.07	&	2.5e-07	&	96.9\\
\hline
\end{tabular}%
}
\caption{$\ell_2$-norm soft margin: runtime (sec), relative error for $\cL(x_k,y_k)$ and TSA (\%) for {\bf APD1}, {\bf APD2}, {\bf APD2-restart} and {\bf Mirror-prox} at iteration $k\in\{1000, 1500, 2000, 2500\}$.}
\vspace*{-8mm}
\label{table:l2}
\end{table*}
\begin{figure}[h!]
\centering
\includegraphics[scale=0.16]{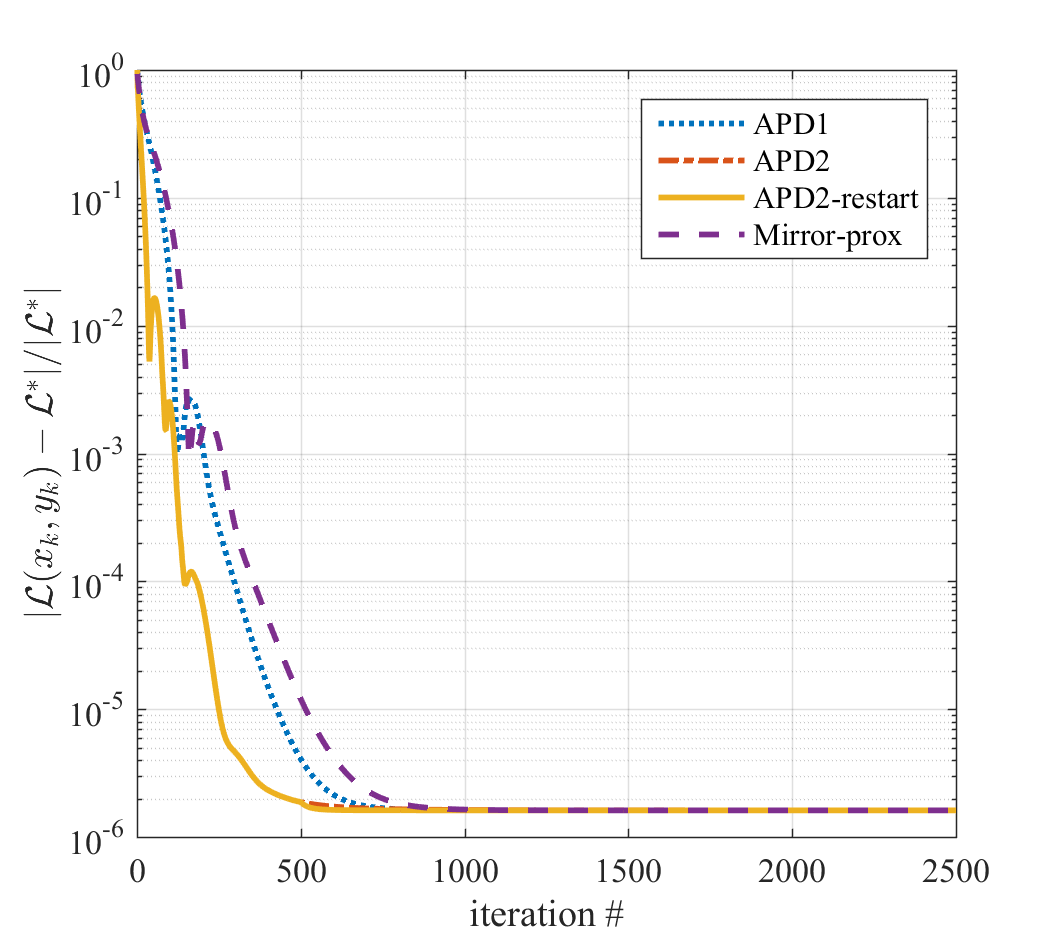}
\includegraphics[scale=0.16]{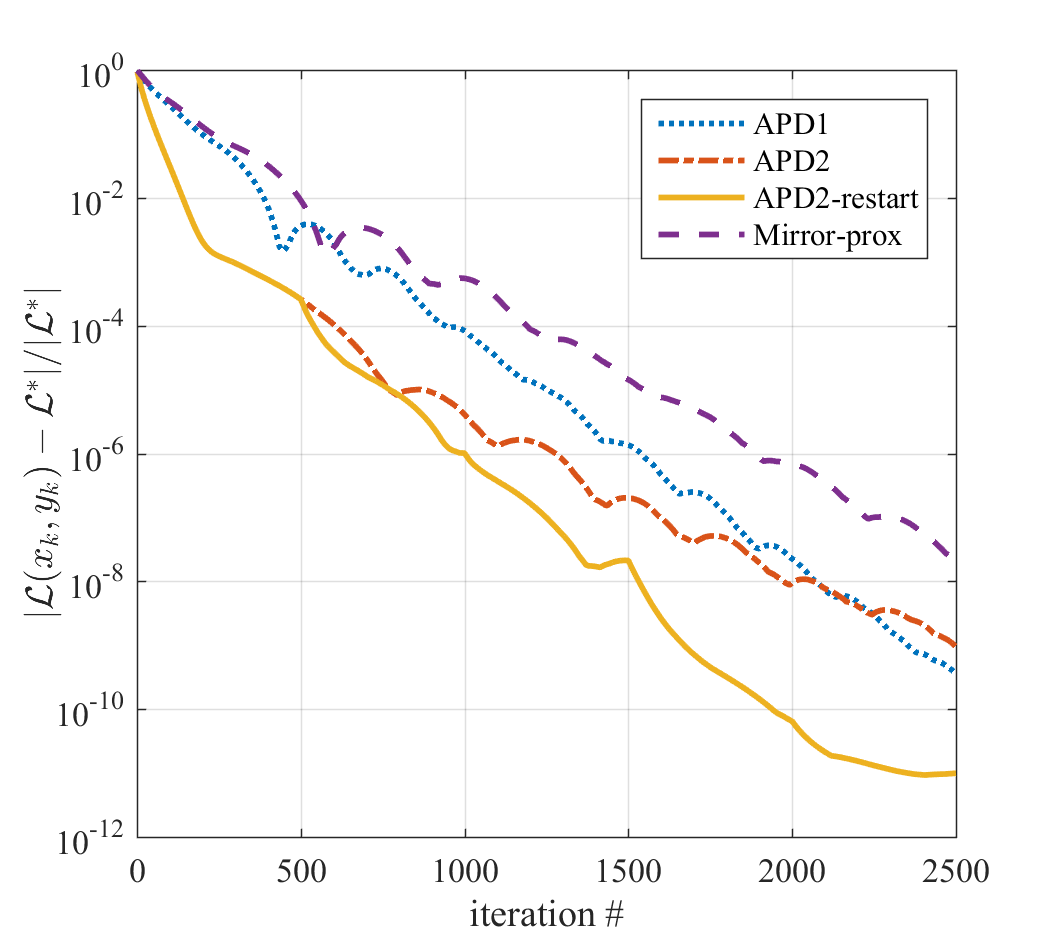}
\includegraphics[scale=0.16]{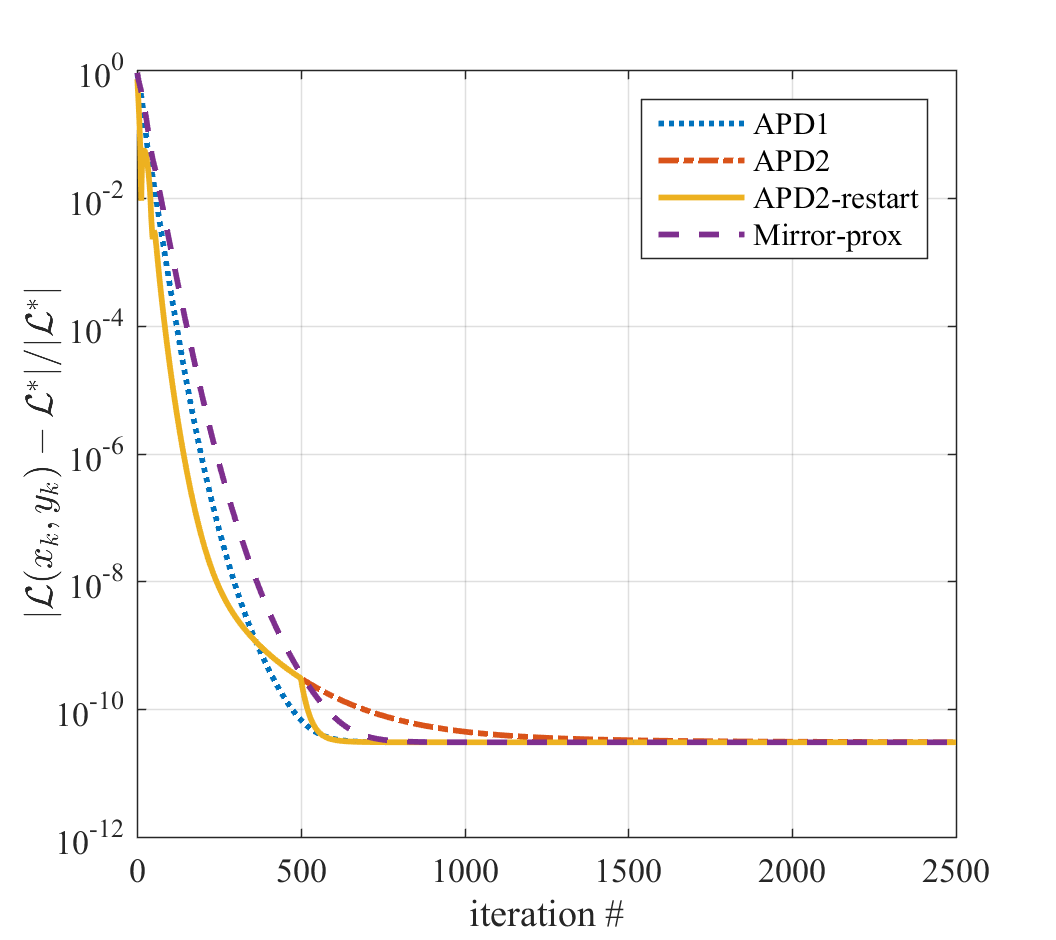}
\includegraphics[scale=0.16]{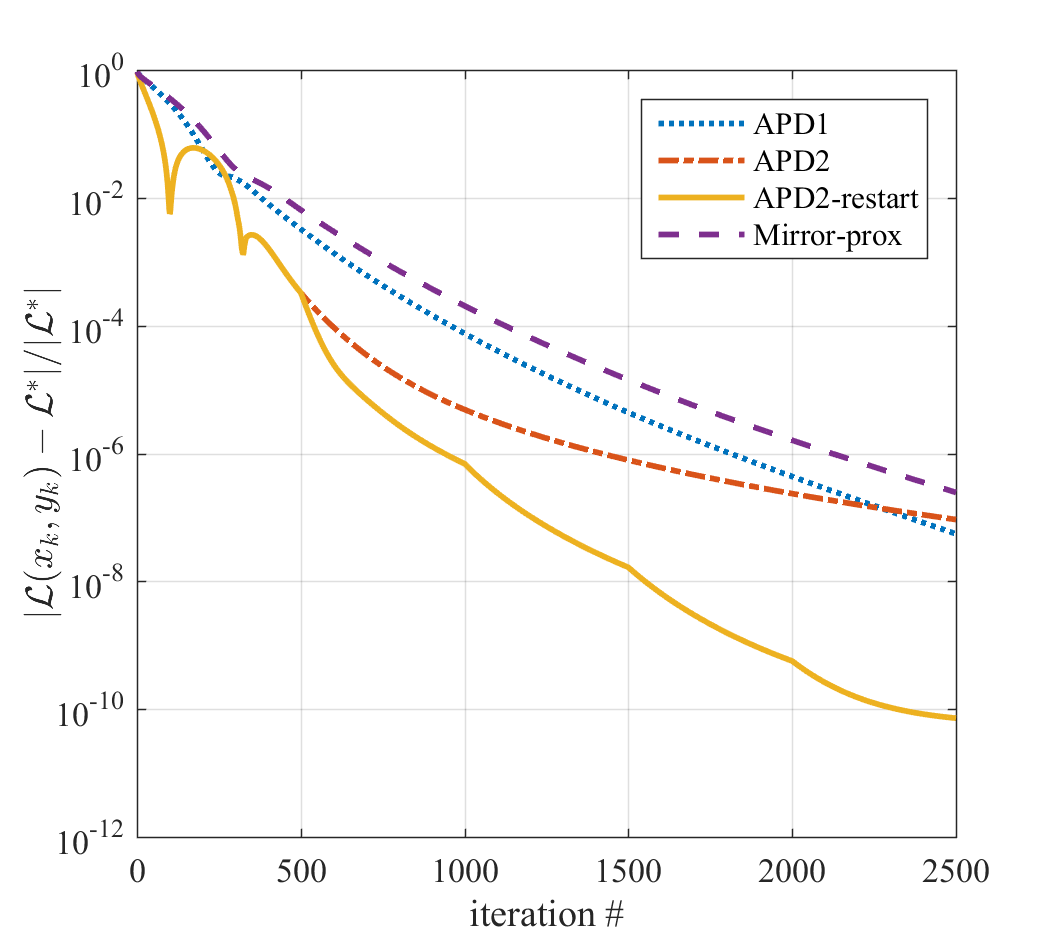}
\caption{$\ell_2$-norm soft margin: Comparison of {\bf APD1}, {\bf APD2}, {\bf APD2-restart} and {\bf Mirror-prox} in terms of rel. error for $\cL(x_k,y_k)$. From left to right: \texttt{Ionosphere}, \texttt{Sonar}, \texttt{Heart}, \texttt{Breast-Cancer}.}
\vspace*{-8mm}
\label{fig:l2-obj}
\end{figure}
\begin{figure}[h!]
\centering
\includegraphics[scale=0.16]{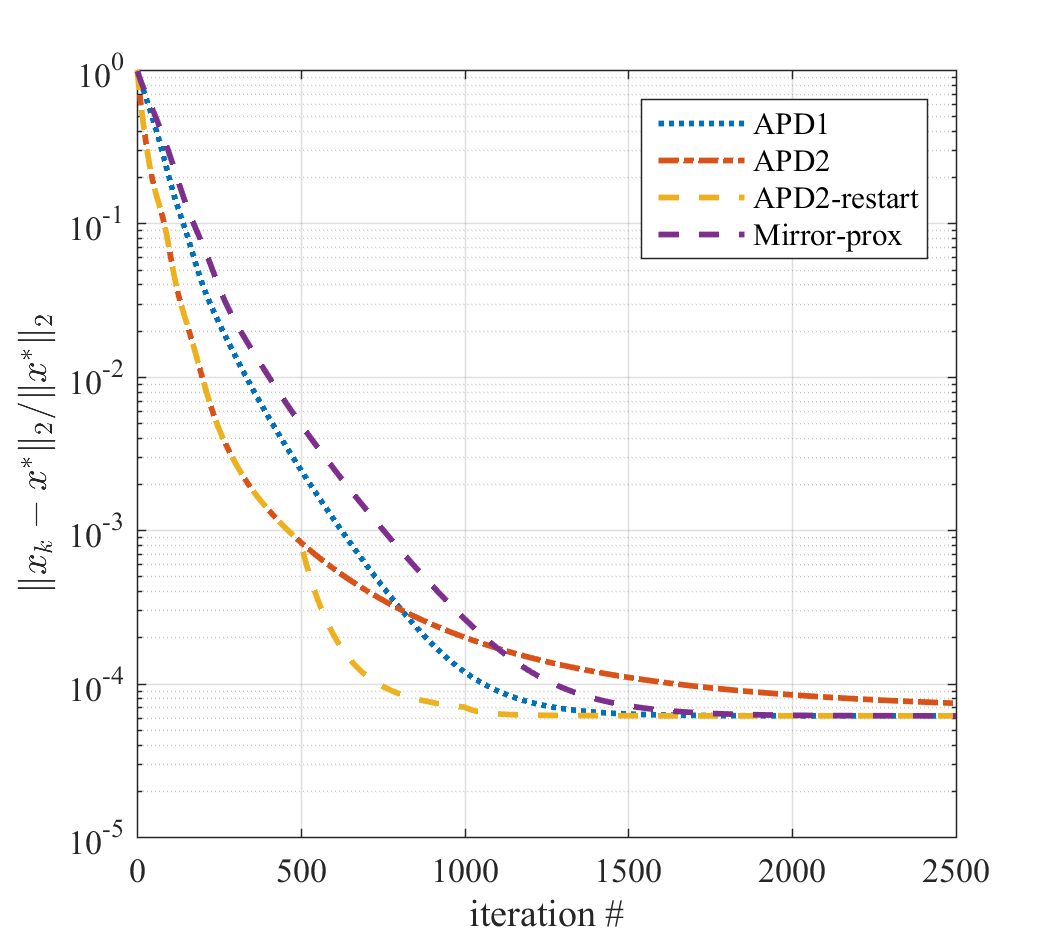}
\includegraphics[scale=0.16]{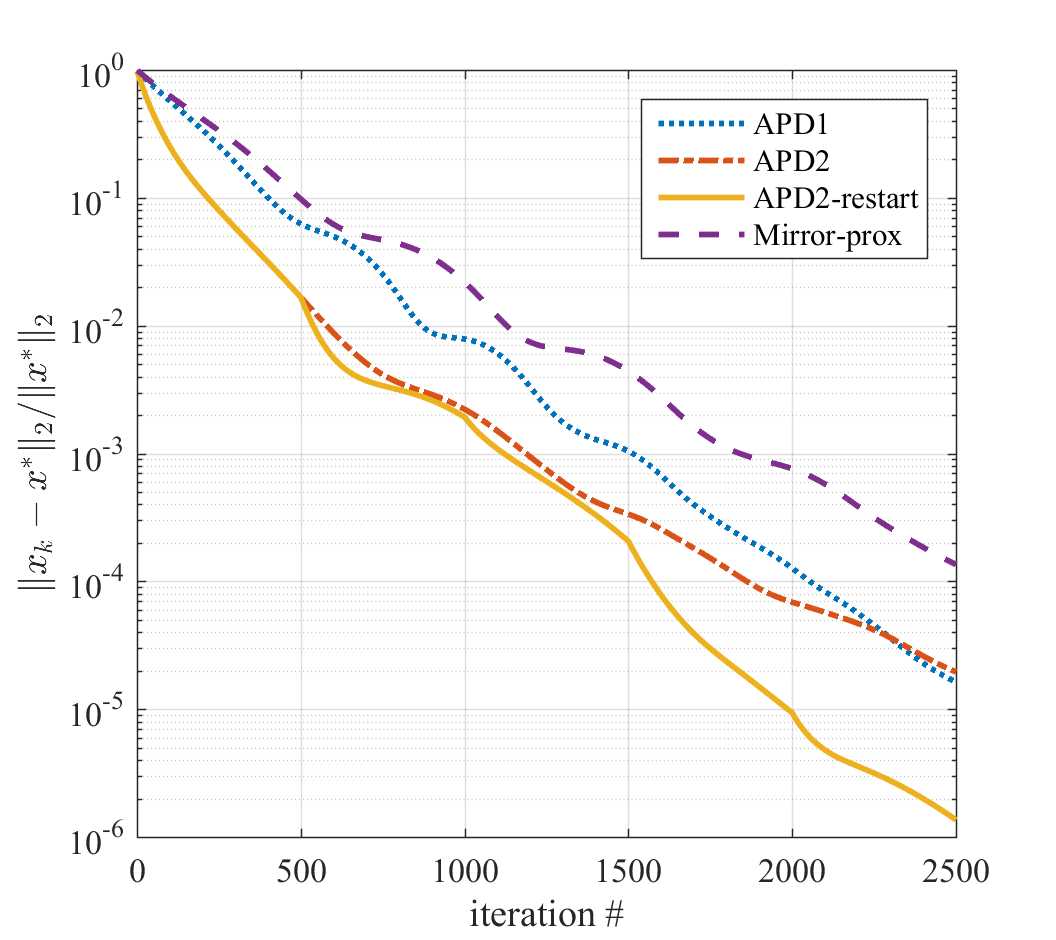}
\includegraphics[scale=0.16]{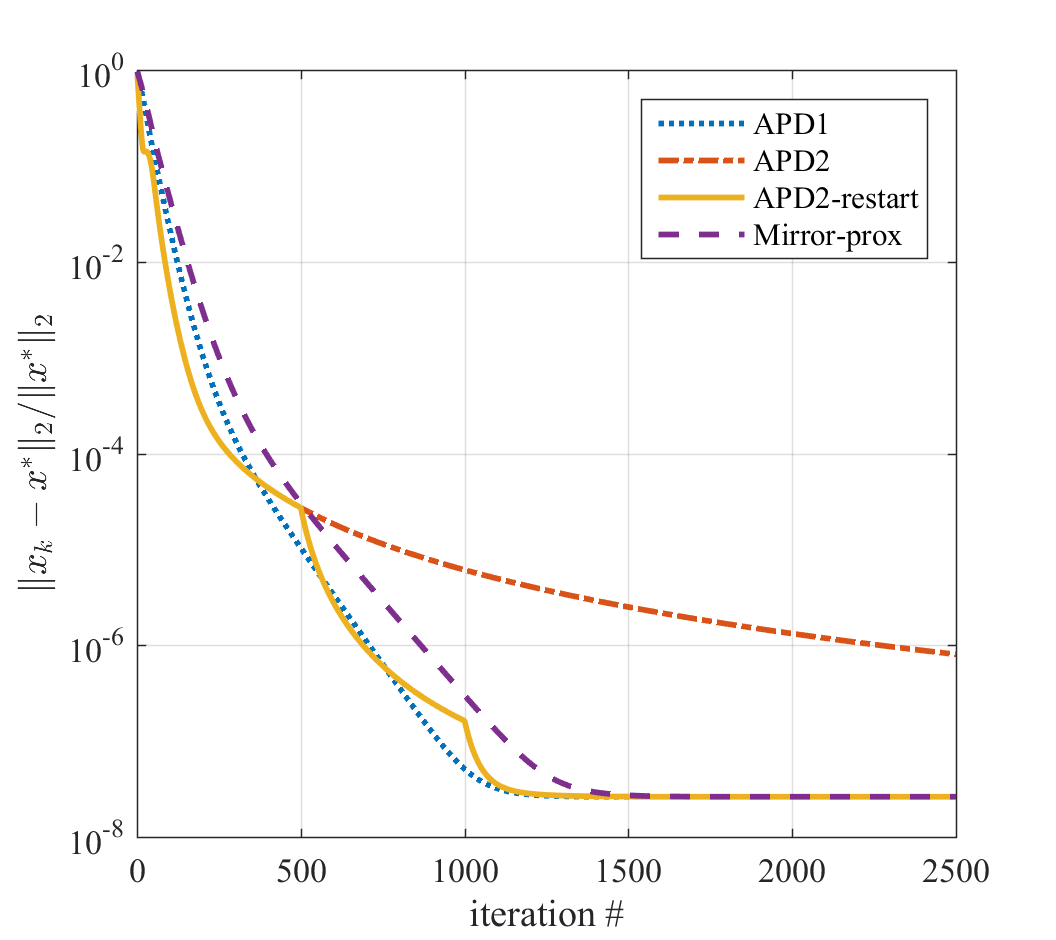}
\includegraphics[scale=0.16]{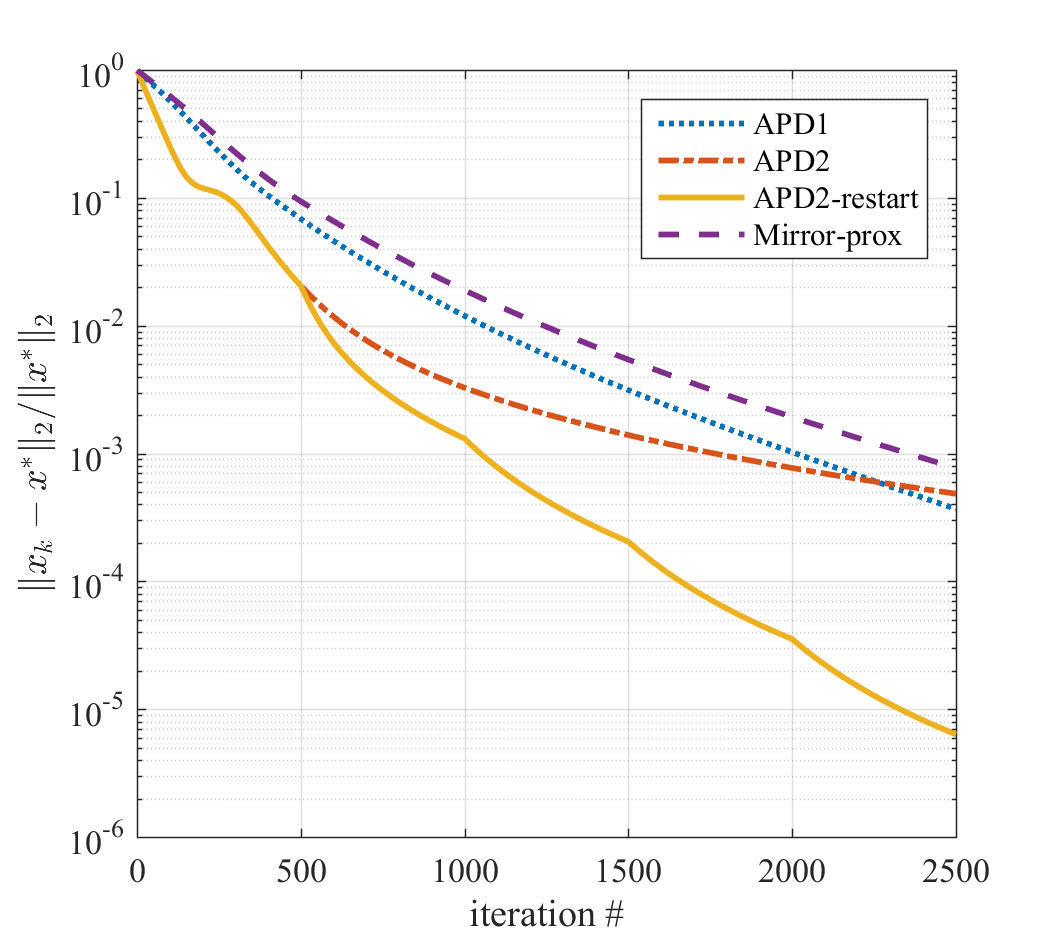}
\caption{$\ell_2$-norm soft margin: Comparison of {\bf APD1}, {\bf APD2}, {\bf APD2-restart} and {\bf Mirror-prox} in terms of rel. error for $x_k$. From left to right: \texttt{Ionosphere}, \texttt{Sonar}, \texttt{Heart}, \texttt{Breast-Cancer}.}
\vspace*{-8mm}
\label{fig:l2-sol}
\end{figure}
\subsubsection{ {APD} vs off-the-shelf interior point methods}
\label{sec:experiment-ipm}
In this section we compare time complexity of our methods against widely-used, open-source interior point method~(IPM) solvers {\bf Sedumi} v4.0 and {\bf SDPT3} v1.3. Here we used two data sets: \texttt{Breast-Cancer} available in UCI repository and {\texttt{SIDO0}}~\cite{guyon2008sido} (6339 observations, 4932 features -- we used half of the observations in data set). 
The goal is to investigate how the quality of iterates, measured in terms of relative solution error $\norm{x_k-x^*}_2/\norm{x^*}_2$, changes as run time progresses. Let $x^*$ be the solution of problem \eqref{eq:kernel-l2} which we computed using MOSEK in CVX with the best accuracy option.
To compare {\bf APD1}, {\bf APD2}, {\bf APD2-restart} and {\bf Mirror-prox} against the interior point methods {\bf Sedumi} and {\bf SDPT3}, the problem in~\eqref{eq:kernel-l2} is first solved by {\bf Sedumi} and {\bf SDPT3} using their default setting. Let $t_1$ and $t_2$ denote the run time of {\bf Sedumi} and {\bf SDPT3} in seconds, respectively. Next, the primal-dual methods {\bf APD1}, {\bf APD2}, {\bf APD2-restart} and {\bf Mirror-prox} were run with the same settings as  in $\ell_2$-norm soft margin experiment for $\max\{t_1,t_2\}$ seconds. 
The mean of relative solution error of the iterates over 10 replications are plotted against time (seconds) for each of these methods in Figure~\ref{fig:largedata}.
{\bf Sedumi} and {\bf SDPT3} are second-order methods and have much better theoretical convergence rates compared to the first-order primal dual methods proposed in this paper. However, for large-scale machine learning problems with dense data, the work per iteration of an IPM is significantly more than that of a first-order method; hence, if the objective is to attain low-to-medium level accuracy solutions, then first-order methods are better suited for this task, e.g., in Figure~\ref{fig:largedata-sido}, {\bf APD2-restart} iterates are more accurate than {\bf Sedumi} and {\bf SDPT3} for the first 2000 and 1000 seconds respectively.\vspace*{-2mm}
\begin{figure}[htbp]
\centering
\begin{subfigure}[b]{0.45\textwidth}
\includegraphics[width=5.8cm, height=4.5cm]{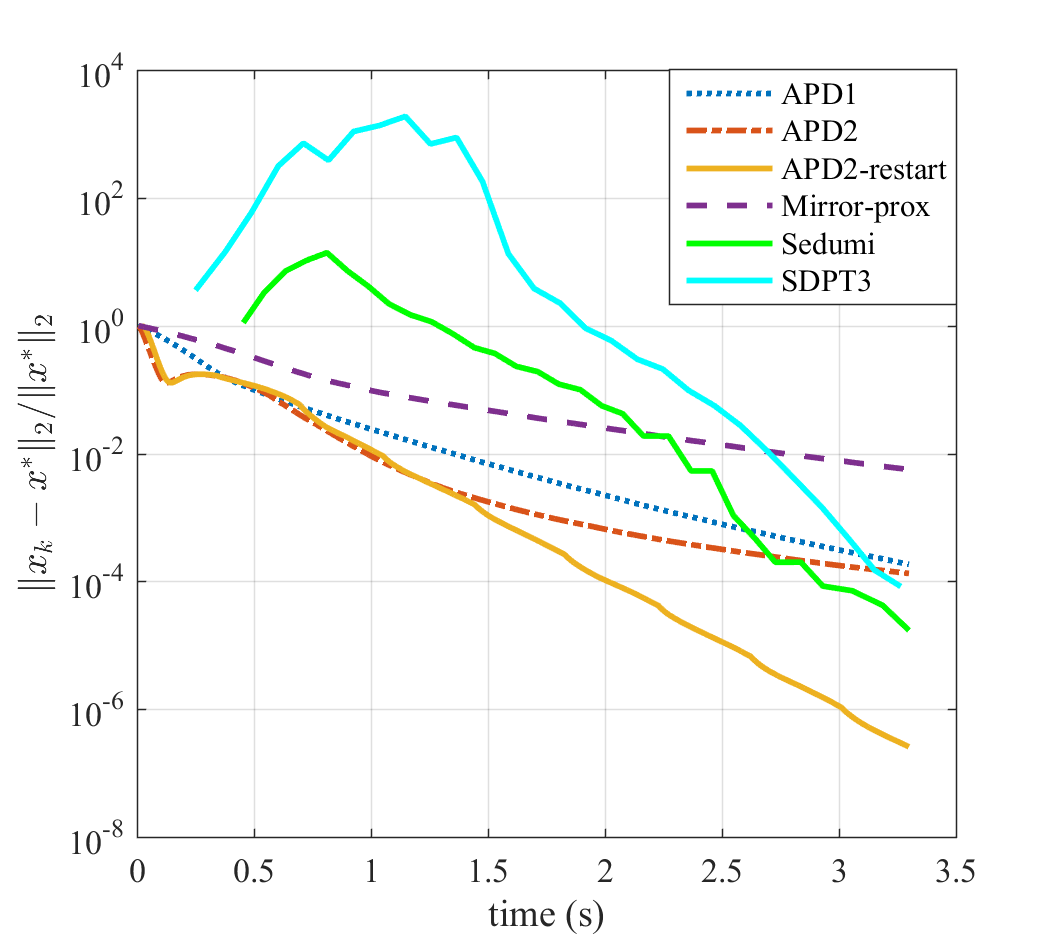}
\caption{{\texttt{Breast-Cancer}}}
\label{fig:largedata-bc}
\end{subfigure}
\begin{subfigure}[b]{0.45\textwidth}
\includegraphics[width=5.8cm, height=4.5cm]{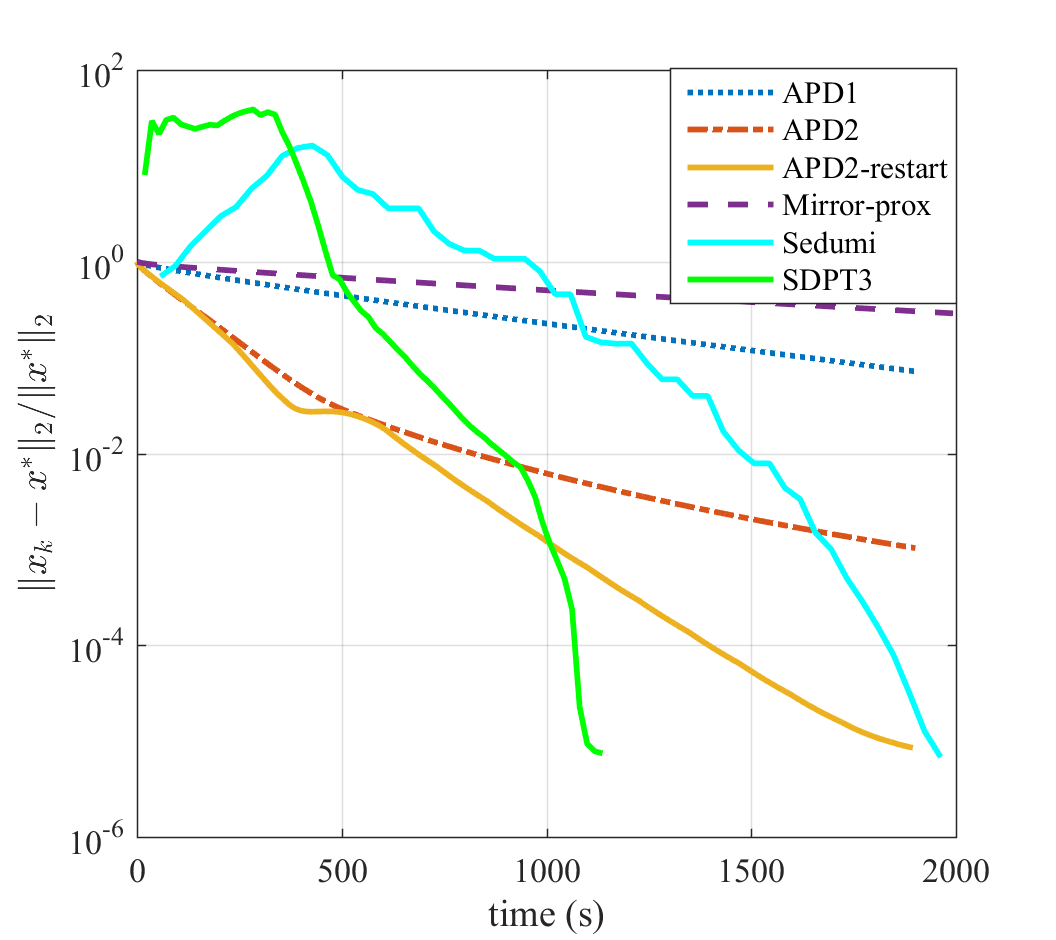}
\caption{{\texttt{SIDO0}}}
\label{fig:largedata-sido}
\end{subfigure}
\caption{Run time comparison of {\bf APD1}, {\bf APD2}, {\bf APD2-restart}, {\bf Mirror-prox}, {\bf Sedumi}, and {\bf SDPT3} in terms of relative error for $x_k$ on $\ell_2$ soft margin problems.}
\vspace*{-8mm}
\label{fig:largedata}
\end{figure}
\subsection{\sa{Quadratic Constrained Quadratic Programming (QCQP)}}\label{sec:qcqp} In this subsection, we compare APDB against linearized augmented Lagrangian method (LALM)~\cite{xu2017first} and proximal extrapolated gradient method~(PEGM) \cite{malitsky2018proximal} on randomly generated QCQP problems.

We consider a QCQP problem of the following form:\vspace*{-3mm}
{\small
\begin{subequations}
\begin{align}
\rho^*\triangleq\min_{x\in X} \quad &\rho(x)\triangleq {\frac{1}{2}x^\top A_0 x +b_0^\top x} \label{qcqp:obj}\\
\hbox{s.t.}\quad &G_j(x)\triangleq \frac{1}{2}x^\top A_j x +b_j^\top x-c_j\leq 0, \quad  j\in\{1,\hdots,m\},
\label{qcqp:const}
\end{align}
\end{subequations}}%
where $X=[-10,10]^n$, $\{A\}_{j=0}^m\subseteq \reals^{n\times n}$ are positive semidefinite matrices generated randomly; $\{b_j\}_{j=0}^m\subseteq \reals^n$ are generated randomly with elements drawn from standard Gaussian distribution; and $\{c_j\}_{j=1}^m\subseteq \reals$ are generated randomly with elements drawn from Uniform distribution over $[0,1]$.

We consider two different scenarios: QCQPs with merely convex $\rho(\cdot)$ and QCQPs with strongly convex $\rho(\cdot)$. For merely convex setup, we set $A_j=\Lambda_j^\top S_j \Lambda_j$ for any $j\in\{0,1,\hdots,m\}$ where $\Lambda_j\in\reals^{n\times n}$ is a random orthonormal matrix and $S_j\in\reals_+^{n\times n}$ is a diagonal matrix whose diagonal elements are generated uniformly at random from $[0,100]$ with $0$ included as the minimum element of $S_j$ for merely convex scenario. To generate $\Lambda_j$, we first generate a random matrix $\tilde{\Lambda}_j\in\reals^{n\times n}$ with each entry sampled from standard Gaussian distribution, then we called MATLAB function $\rm{orth}(\tilde{\Lambda}_j)$, which returns an orthonormal basis for the range of $\tilde{\Lambda}_j$. We generate the data for strongly convex problem instances similarly, except for $j=0$; to generate the objective function for the strongly convex case, the diagonal elements of $S_0$ are generated from $[1,101]$. To test, we generate 10 random i.i.d. problem instances for each scenario.

For this experiment, we set $n=10^3$ and $m=10$. For both merely convex and strongly convex setups, we implemented all the algorithms on 10 randomly generated QCQP instances, and we ran them until the termination condition $\max\{|\rho(x_k)-\rho^*|/|\rho^*|,~\frac{1}{n}\sum_{j=1}^n G_j(x_k)\}\leq \epsilon$ holds for $\epsilon=10^{-8}$.
The backtracking parameter is set to $\eta=0.7$ for all the methods we tested, i.e., APDB, LALM and PEGM. We use the technique suggested in Remark~\ref{rem:larger-step} to possibly select larger step-sizes for APDB. For PEGM, we implemented Algorithm~2 in \cite{malitsky2018proximal} where step-size $\lambda_k$ is increased at the beginning of each iteration, and for LALM we use Algorithm 1 in~\cite{xu2017first} with backtracking. Note that in \cite{xu2017first} an increase in step-sizes has not been considered; hence, the primal step-size at the beginning of each iteration is set to be the last primal step-size. Step-size initializations for PEGM and LALM are done in accordance with APDB, for which we set $\gamma_0=1$ and {$\bar{\tau}=10^{-3}$}. We will refer to APDB with $\mu>0$ and $\mu=0$ as APDB1 and APDB2 respectively.

To have a fair comparison among the methods, we plot the statistics againts the total number of gradient ($\grad_x\Phi$ and $\grad_y\Phi$) computations. Figure~\ref{fig:QCQP} illustrates illustrate the performance of the algorithms based on relative suboptimality and infeasibility error. The solid lines indicate the average statistics over 10 randomly generated replications and {the shaded region around each line shows the range statistic over all random instances}. We observe that APDB outperforms the other two competitive methods; moreover, APDB2 through exploiting the strong convexity ($\mu>0$) of the problem achieves $\epsilon$-accuracy faster than APDB1.

\begin{figure}
\centering
\includegraphics[scale=0.2]{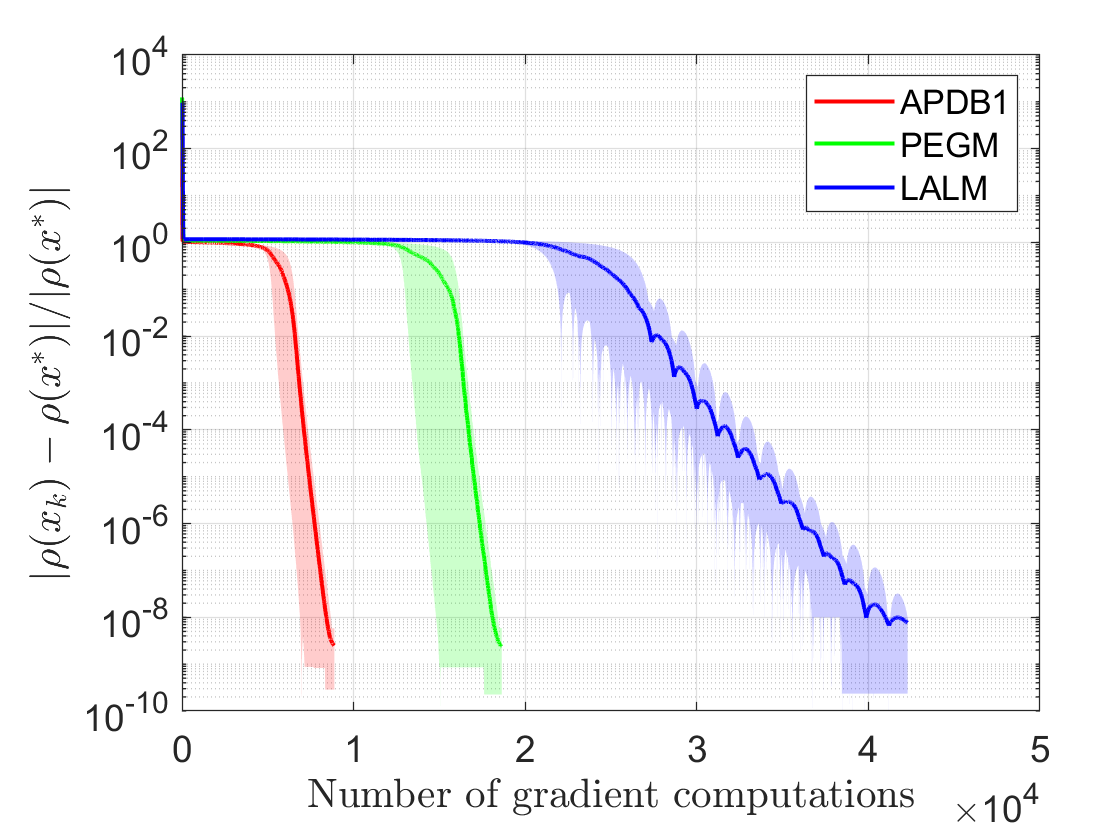}
\includegraphics[scale=0.2]{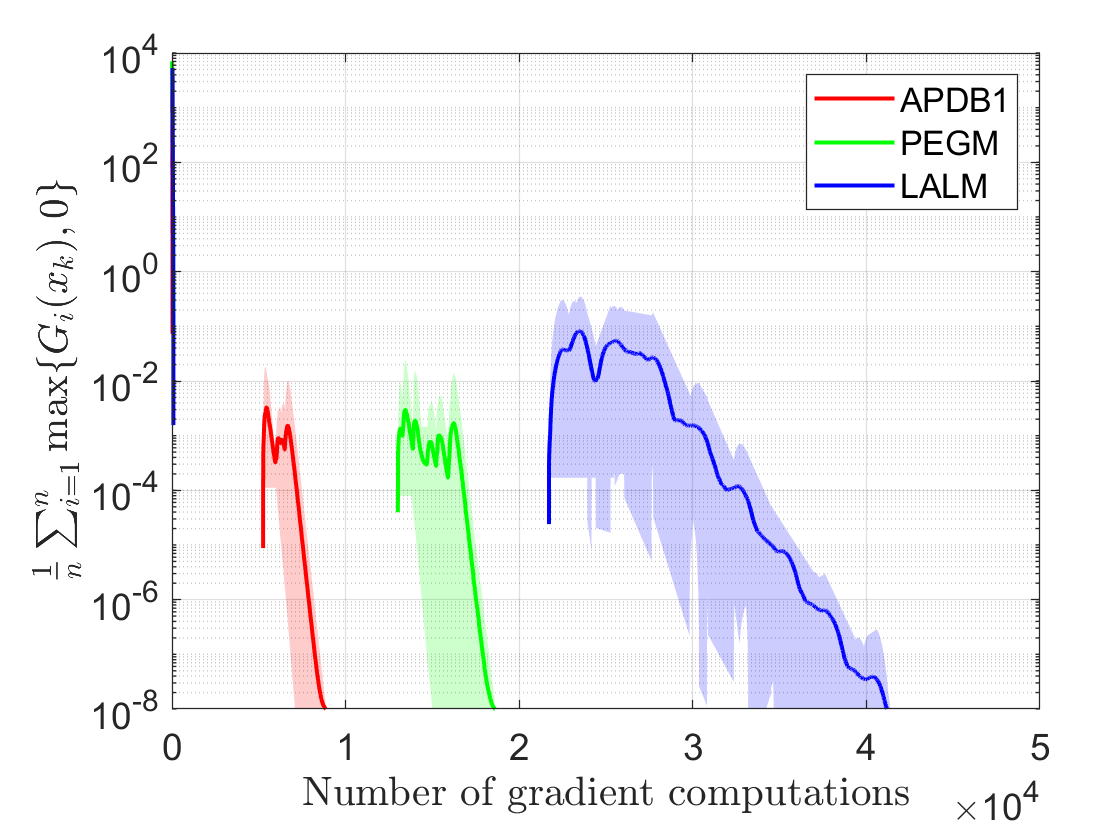}
\includegraphics[scale=0.2]{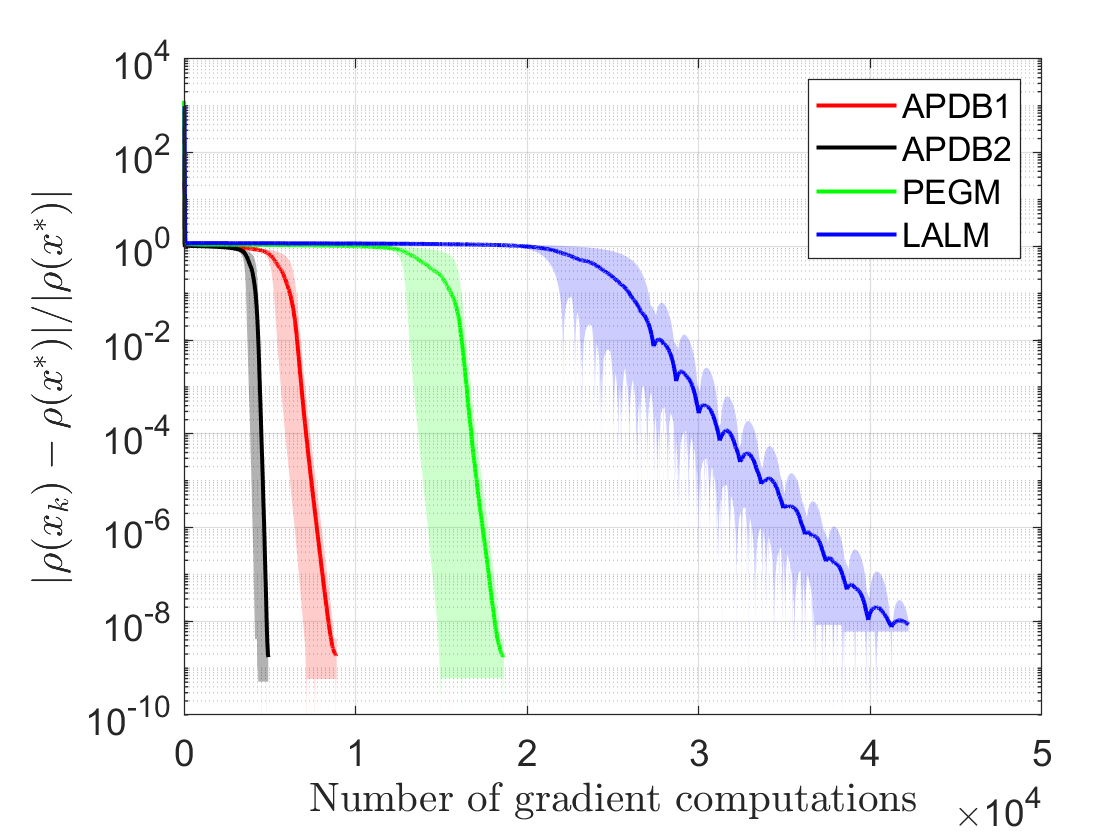}
\includegraphics[scale=0.2]{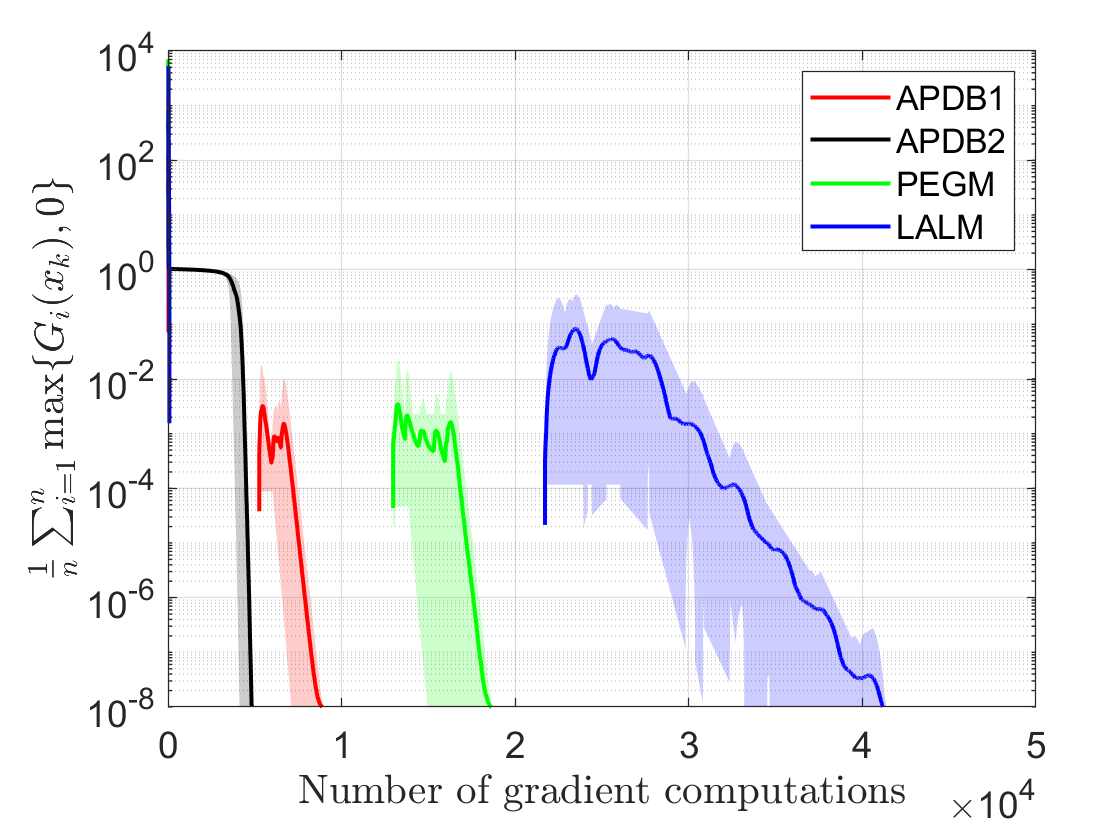}
\caption{Comparison of methods in terms of suboptimality (left) and infeasibility (right) for {merely convex (top row) and strongly convex (bottom row)} setting.}
\label{fig:QCQP}
\vspace*{-8mm}
\end{figure}

\subsection{Regression with Fairness Constraints} \label{sec:RFC}
\rev{In this experiment, we consider a 
regression problem \rev{with fairness constraints~\cite{komiyama2018nonconvex} to compute
a regressor such that its correlation with the sensitive attributes (e.g., race, gender, age) go to $0$ as the number of samples $N\to\infty$. For each $i\in\{1,\ldots,N\}$, the sample point data $(z^{(i)},w^{(i)},s^{(i)})$ consists of target attribute $z^{(i)}\in\reals$ we want to predict, non-sensitive attributes $w^{(i)}\in\reals^{d_w}$ and sensitive attributes $s^{(i)}\in\reals^{d_s}$.
Let $W\triangleq [w^{(1)},\hdots,w^{(N)}]^\top\in \reals^{N\times d_w}$, $S\triangleq [s^{(1)},\hdots,s^{(N)}]^\top\in \reals^{ N\times d_s}$, and let $z\triangleq[z^{(i)}]_{i=1}^N\in\reals^N$}. Suppose $s^{(i)}$ is highly correlated with $w^{(i)}$, i.e., there exists $B\in\reals^{d_s\times d_w}$ such that $w^{(i)}=B^\top s^{(i)}+\varepsilon^{(i)}$ and $\mathbb{E}[\varepsilon^{(i)}|s^{(i)}]=\mathbf{0}$ for all $i\geq 1$. Under this assumption, given i.i.d. $\{(w^{(i)},s^{(i)})\}_{i=1}^N$, one can estimate $B$ with $\hat{B}=(S^\top S)^{-1} S^\top W$ which converges in probability to $B$ as $N\to \infty$. Correlation between $\{w^{(i)}\}_{i=1}^N$ and $\{s^{(i)}\}_{i=1}^N$ can be removed by performing regression using $\{(u^{(i)},s^{(i)})\}_{i=1}^N$ where $[u^{(1)},\hdots,u^{(N)}]^\top =U\triangleq W-S\hat{B}$. Suppose the data is centered, i.e., $\mathbf{1}^\top S=\mathbf{0}^\top$ and $\mathbf{1}^\top W=\mathbf{0}^\top$, which also implies that $\mathbf{1}^\top U=\mathbf{0}^\top$.}

\rev{Consider $\hat{z}=S\alpha+U\beta$ as the estimator of $z$ for some $\alpha\in\reals^{d_s}$ and $\beta\in\reals^{d_w}$. Since $\{u^{(i)}\}_{i=1}^N$ and $\{s^{(i)}\}_{i=1}^N$ are uncorrelated, the best estimator of $\hat{z}$ (minimizing MSE) without using sensitive attributes $\{s^{(i)})\}_{i=1}^N$ is given by $\bar{z}=U\beta$. Given a fairness tolerance $\zeta\in[0,1]$, the goal is to compute $(\alpha,\beta)$ such that it minimizes the error $\norm{w-\hat{w}}^2$ subject to $\norm{\bar{w}-\hat{w}}^2/\norm{\hat{w}}^2\leq\zeta$, i.e., the ratio of variance in the estimator $\hat{y}$ explained by the sensitive attributes; thus, as $\zeta$ gets closer to $0$, the regressor becomes more fair, while as $\zeta$ gets closer to $1$ the predictive power of the regressor increases.
This problem is formulated as the following non-convex QCQP \cite{komiyama2018nonconvex}:
\begin{align}\label{fair-nonconvex}
    \min_{x=[\alpha^\top \beta^\top]^\top}&\rho_1(x)\triangleq \alpha^\top V_s \alpha + \beta^\top V_u\beta\rev{-2 q_s^\top\alpha-2 q_u^\top\beta}\\
    \hbox{s.t.}~ & G(x)\triangleq (1-\zeta)\alpha^\top V_s \alpha -\zeta \beta^\top V_u\beta\leq 0, \nonumber
\end{align}
where 
$V_s=\frac{1}{N}S^\top S$ and $V_u=\frac{1}{N}U^\top U$ are sample covariance matrices, \rev{$q_s\triangleq\frac{1}{N}S^\top y$, $q_u\triangleq\frac{1}{N}U^\top y$}, and 
$\zeta\in (0,1)$ is a user-defined fairness tolerance.}

\rev{Let $Q_1\triangleq \begin{bmatrix} V_s & 0 \\ 0 & V_u\end{bmatrix}$, $Q_2\triangleq \begin{bmatrix} V_s & 0 \\ 0 & 0\end{bmatrix}$ and \rev{$q=[q_s^\top q_u^\top]^\top$.}
It has been shown in \cite{komiyama2018nonconvex}[Theorem 4], the solution to the problem \eqref{fair-nonconvex} can be equivalently found by solving a convex QCQP which can be 
\rev{reformulated as:}
\begin{align}\label{fair-saddle}
    \min_{x}\rho_2(x),\quad \hbox{where}\quad\rho_2(x)\triangleq \max\{x^\top Q_1 x,~\frac{1}{\zeta}x^\top Q_2 x \}-2 \rev{q^\top x}.
\end{align}
\eqref{fair-saddle}
is a special case of \eqref{eq:original-problem}; indeed, $\min_x\max_{w\in\Delta_2} w_1 x^\top Q_1 x + \frac{1}{\zeta}w_2 x^\top Q_2 x -\rev{2q^\top x}$ where $\Delta_2\subset\reals^2$ denotes \rev{the 
unit simplex. We solve this SP problem using the proposed algorithm APDB
and compare it with PEGM 
on both real and synthetic datasets.}}

\paragraph{\bf Real Dataset} \rev{We consider the Community and Crime (\texttt{C\&C}) dataset available in UCI repository with 1994 observations, 101 \rev{attributes}, and the National Longitudinal Survey of Youth (\texttt{NLSY79}) dataset\footnote{https://www.bls.gov/nls/} with 6213 observations, 21 \rev{attributes}. For \texttt{C\&C} dataset, the target \rev{attribute $z$} is the normalized violent crime rate of each community and \rev{the sensitive attributes $s_1$, $s_2$} are the ratio of African American people and foreign-born people, respectively. For \texttt{NLSY79} dataset,  the target $z$ is the income of each person in 1990 and $s_1$, $s_2$ are the gender and age, respectively. \rev{After each dataset has been normalized by subtracting the mean and dividing by the standard deviation, we tested 
APDB and PEGM on 10 random replications such that each replication used 80\% of the data points randomly selected from the whole dataset}.}

\paragraph{\bf Synthetic Dataset} \rev{Following a similar setup as in \cite[Appendix B]{komiyama2018nonconvex}, we consider a scenario with $d_s=100$ and $d_w=900$, i.e.,  10\% of the attributes are sensitive, 
and $N=10^4$. 
\rev{The entries of both $W$ and $S$} are drawn from standard normal distribution \rev{$\cN(0,1)$}, and $z=S\ones_{d_s}+\frac{1}{100}W\ones_{d_w}+\eta$ where $\ones_m$ denotes an $m$-dimensional vector of ones, and $\eta\in\reals^N$ is
a random vector with elements drawn from $\cN(0,1)$.
We 
ran APDB and PEGM on 10 random replications.}

\rev{The parameters 
for APDB and PEGM algorithms are set
as in Section~\ref{sec:qcqp} other than $\bar{\tau}=1$.
In Figure \ref{fig:fair-regression}, the plots for three different statistics are displayed against the number of gradient calls, i.e., suboptimality of $x_k$ for \eqref{fair-saddle}, suboptimality of $x_k$ for \eqref{fair-nonconvex}, and infeasibility of $x_k$ for \eqref{fair-nonconvex}. The solid lines indicate the average statistic over 10 randomly generated replications and the shaded region around each \rev{solid} line shows the range statistic over all random instances.} 

\begin{figure}[h!]
\centering
\includegraphics[scale=0.15]{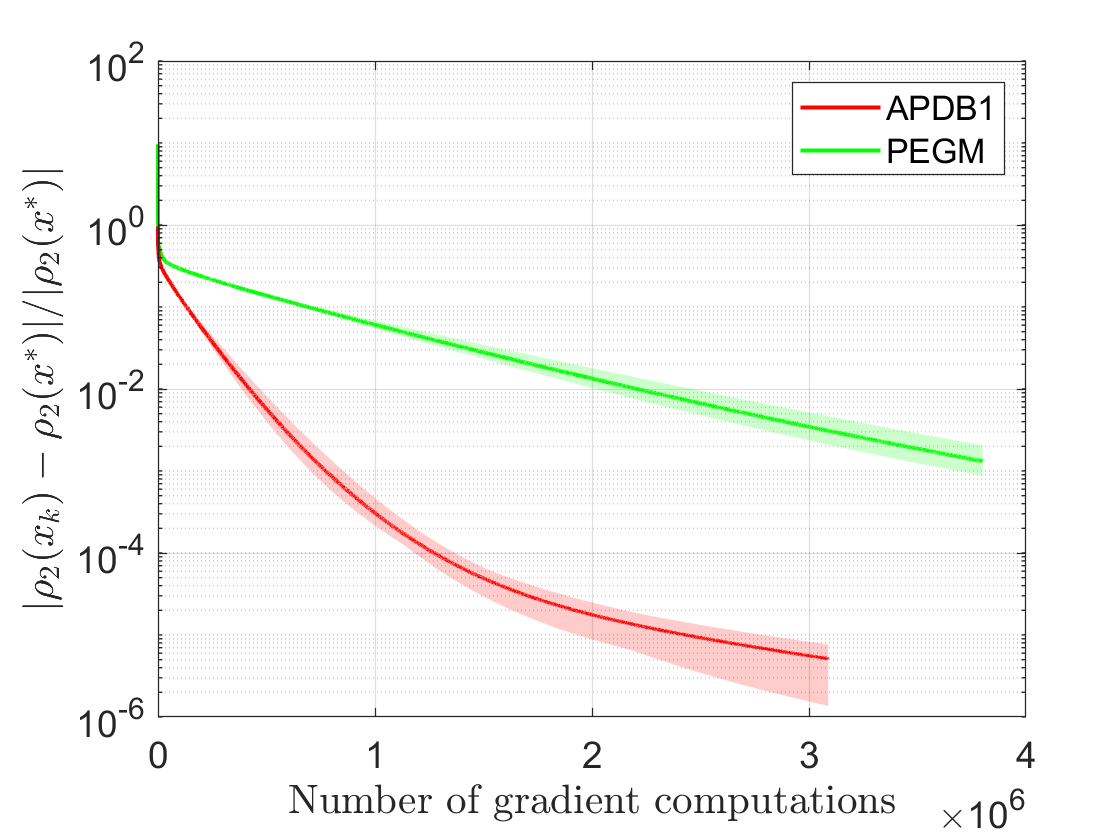}
\hspace{-5mm}
\includegraphics[scale=0.15]{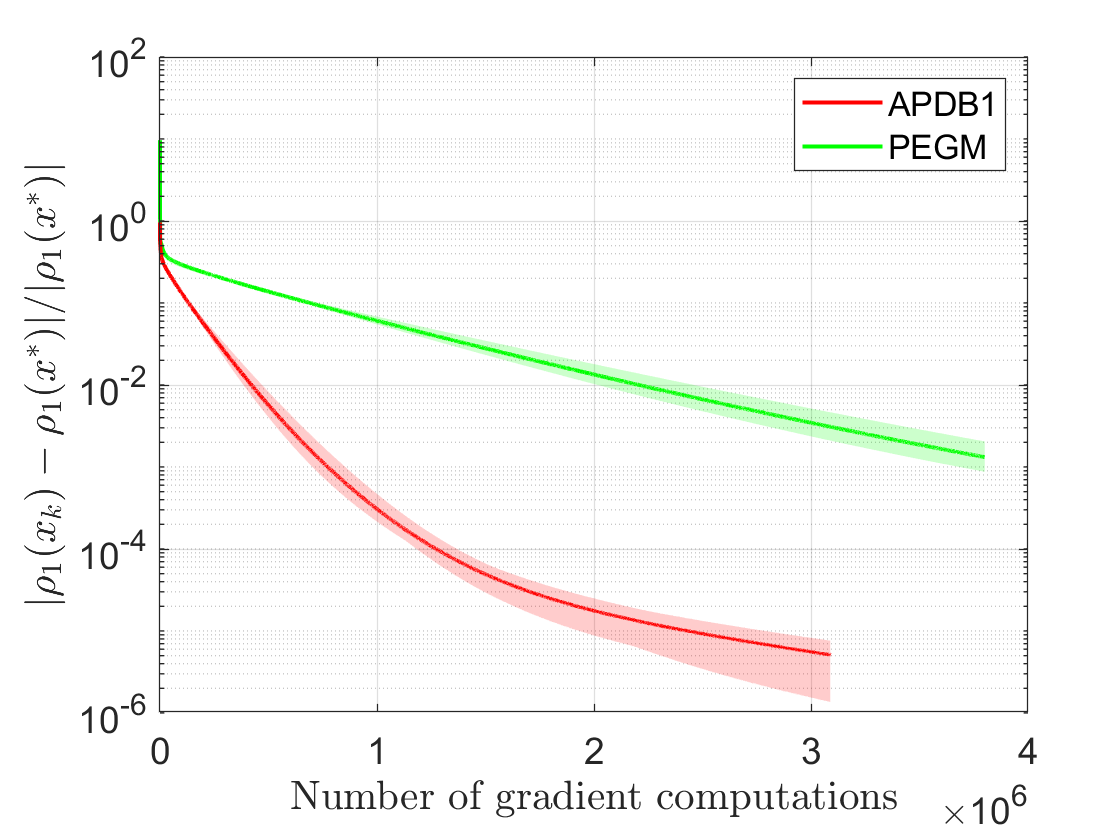}
\hspace{-5mm}
\includegraphics[scale=0.15]{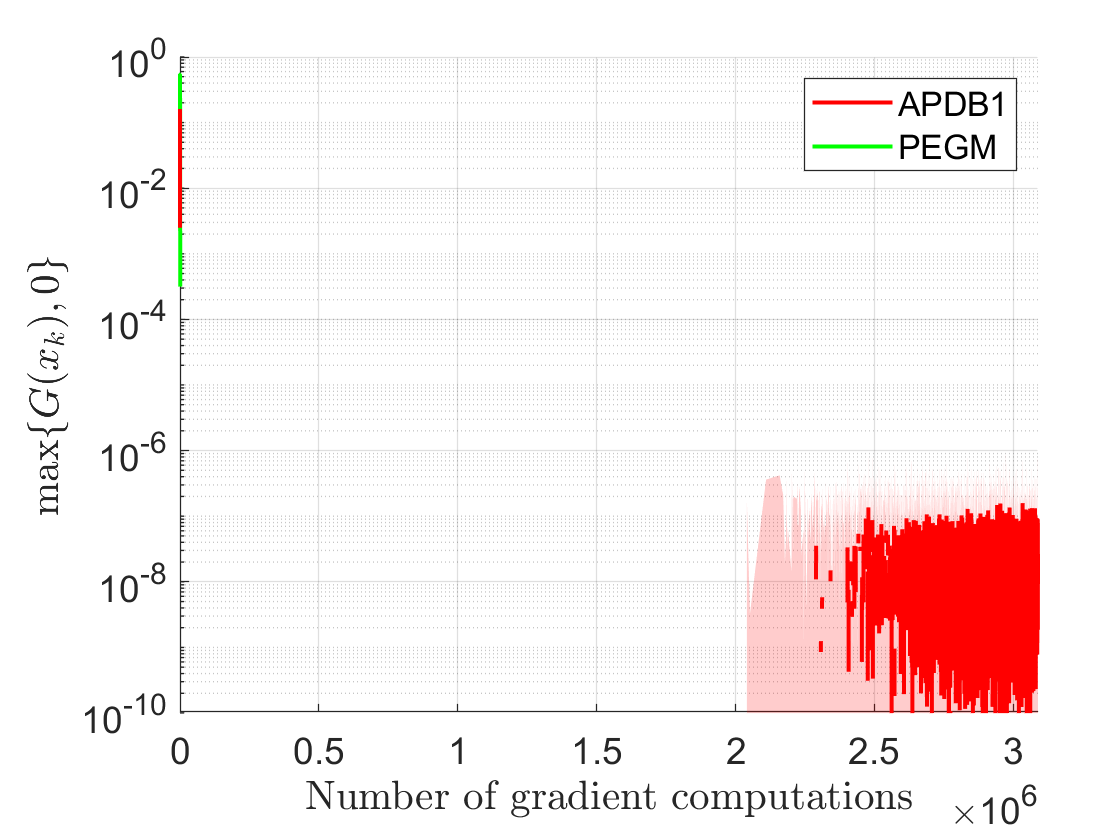}
\includegraphics[scale=0.15]{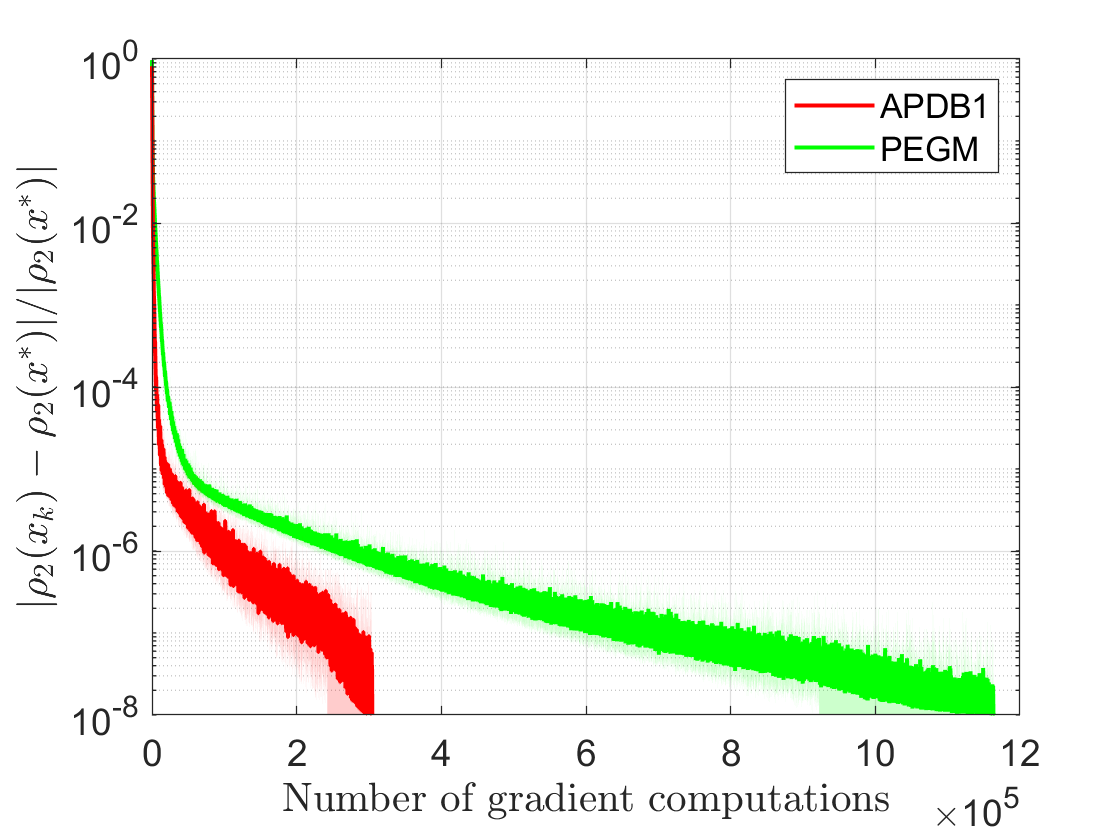}
\hspace{-5mm}
\includegraphics[scale=0.15]{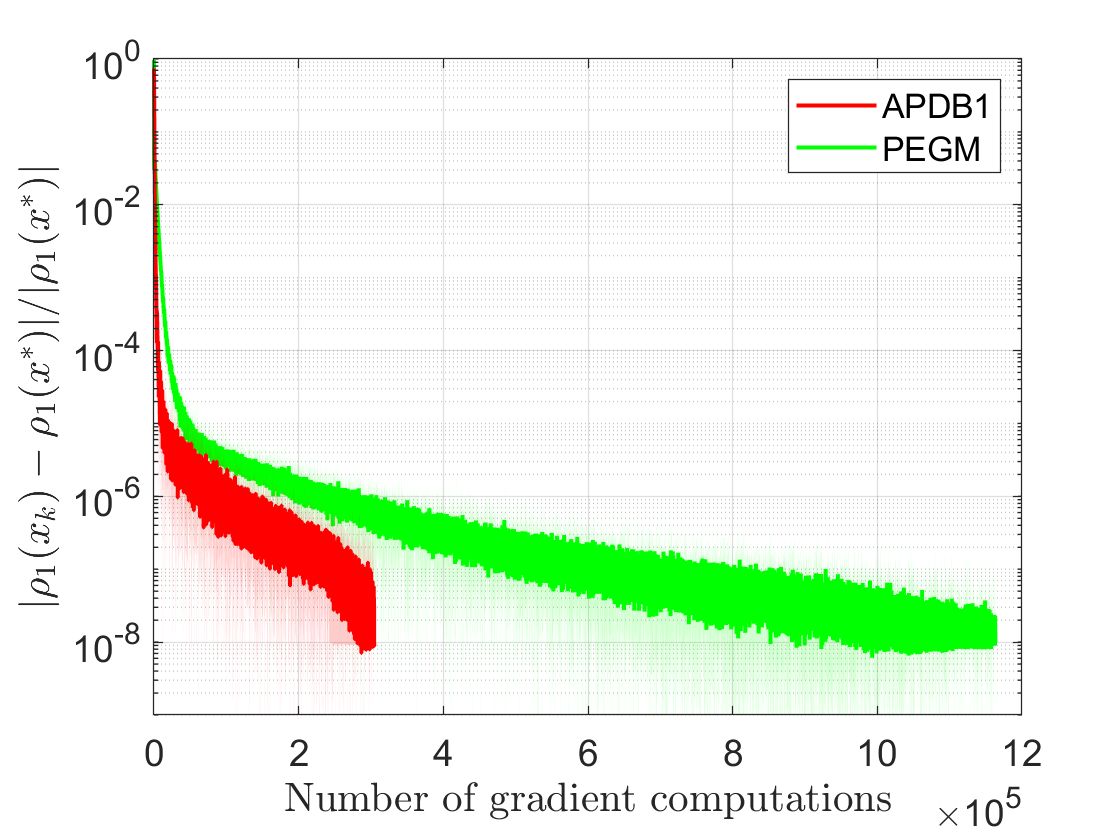}
\hspace{-5mm}
\includegraphics[scale=0.15]{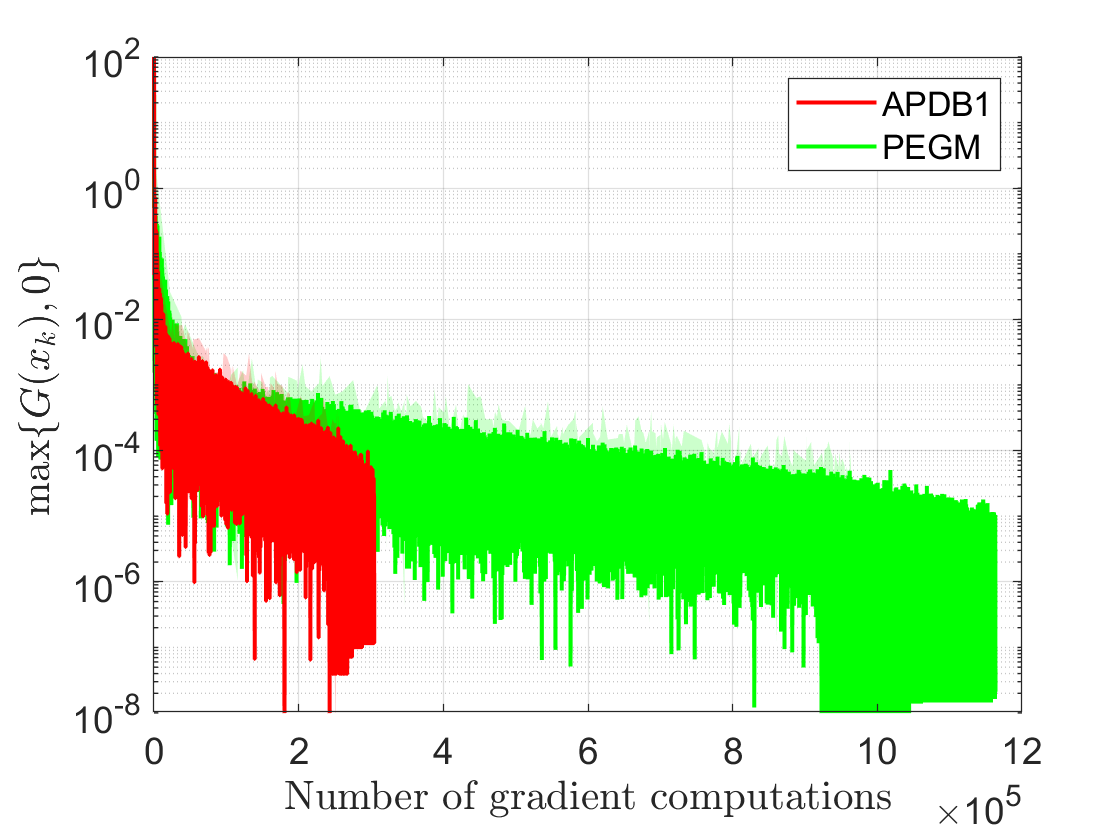}
\includegraphics[scale=0.15]{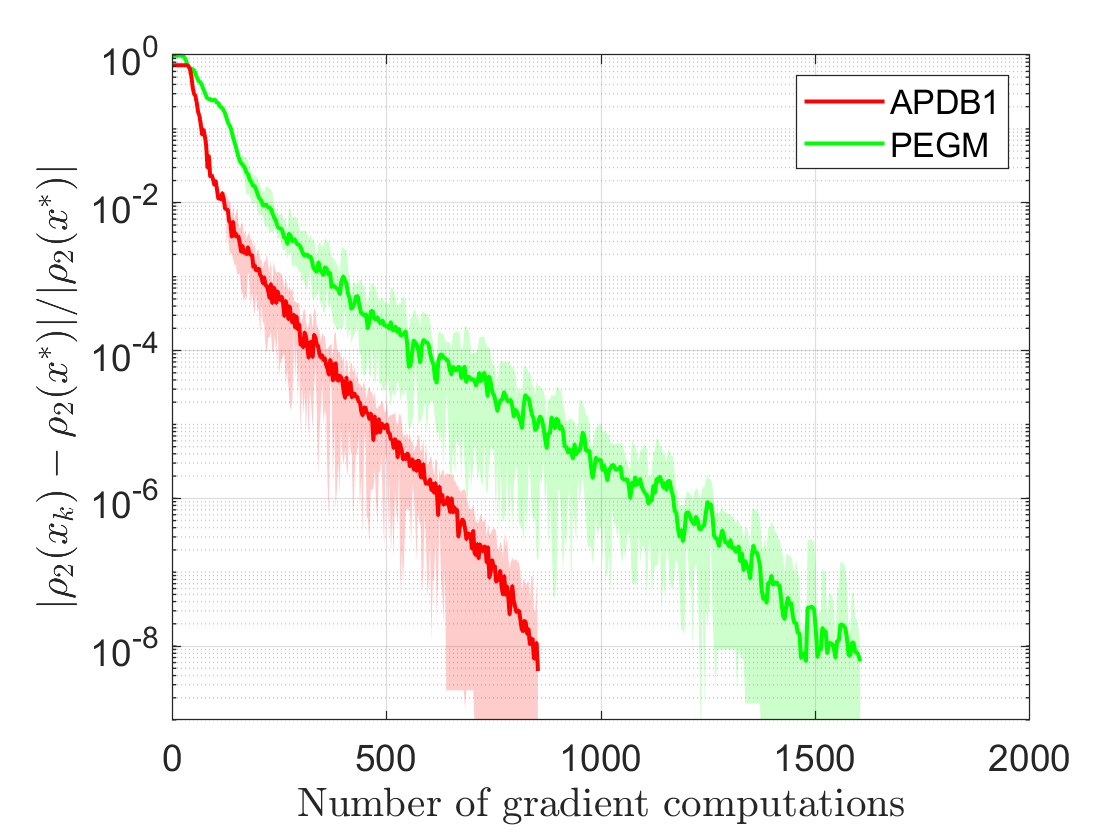}
\hspace{-5mm}
\includegraphics[scale=0.15]{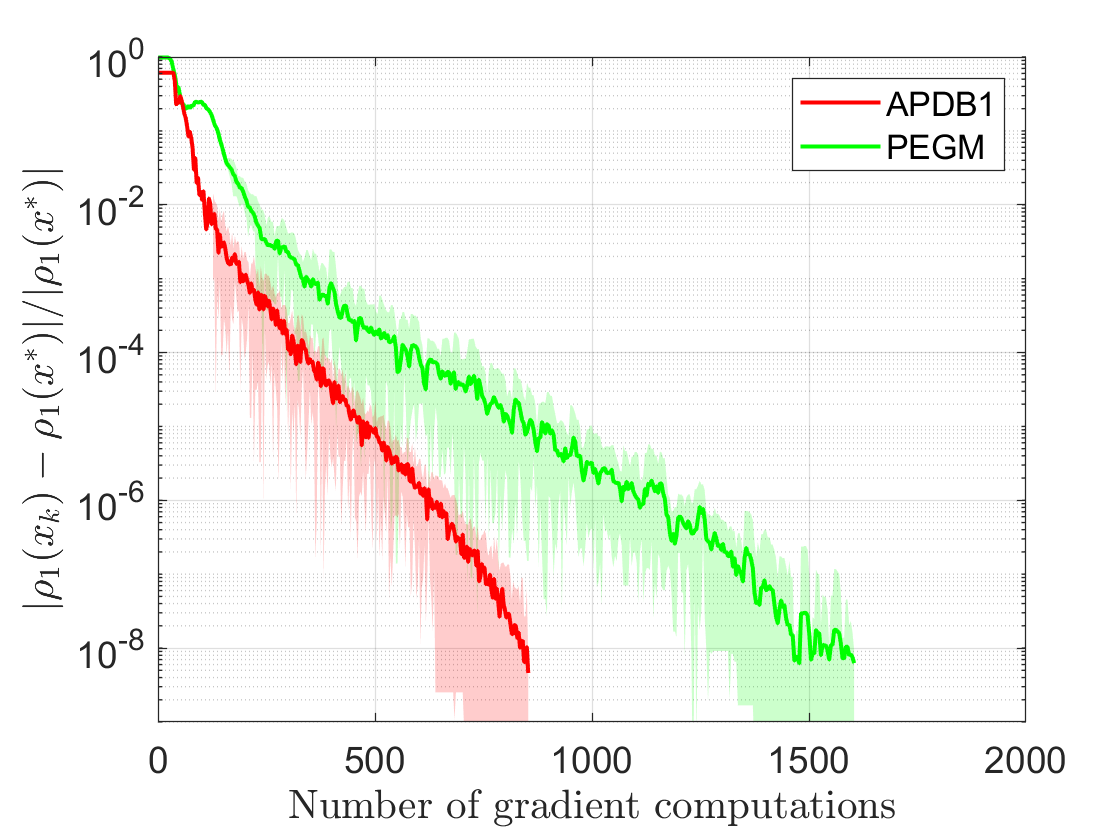}
\hspace{-5mm}
\includegraphics[scale=0.15]{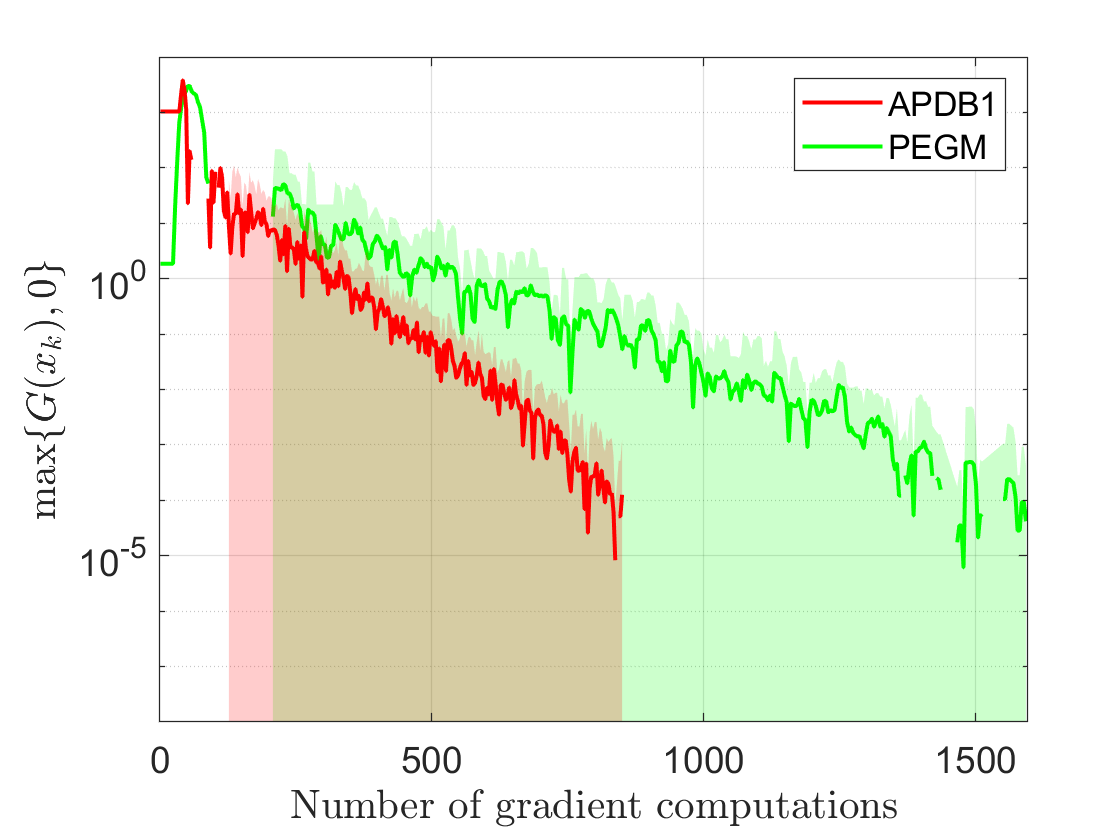}
\caption{\rev{Comparison of methods in terms of suboptimality of \eqref{fair-saddle} (left), suboptimality of \eqref{fair-nonconvex} (middle), and infeasibility of \eqref{fair-nonconvex} (right). The plots from top to bottom row: 
\texttt{C\&C}, \texttt{NSLY79}, and  \texttt{Synthetic}.}}
\vspace*{-2mm}
\label{fig:fair-regression}
\end{figure}

\section{\sa{Concluding remarks and future work}}\label{sec:conclude}
We proposed a primal-dual algorithm with a novel momentum term based on gradient extrapolation
to solve 
SP problems defined by a convex-concave function $\cL(x,y)=f(x)+\Phi(x,y)-h(y)$ with a general coupling term $\Phi(x,y)$ that is \emph{not} assumed to be bilinear. Assuming $\grad_y\Phi$ is Lipschitz and $\grad_x\Phi(\cdot,y)$ is Lipschitz for any fixed $y$, {we show that $\{x_k,y_k\}$ converges to a saddle point $(x^*,y^*)$} and derive error bounds in terms of {$\cL(\bar{x}_K,y^*)-\cL(x^*,\bar{y}_K)$} for the ergodic sequence -- without requiring primal-dual domains to be bounded; in particular, we show $\cO(1/K)$ rate when the problem is merely convex in $x$ using a constant step-size rule, {where $K$ denotes the number of gradient computations.} 
Assuming $\Phi$ is affine in $y$ 
and $f$ is strongly convex, we also obtained the ergodic convergence rate of $\cO(1/K^2)$. {Moreover, we introduced a backtracking scheme that guarantees the same results 
without requiring the Lipschitz constants $L_{xx}$, $L_{yx}$, and $L_{yy}$ to be known. To best of our knowledge, this is the first time $\cO(1/K^2)$ rate is shown for a single-loop method and the first time that a line-search method ensuring $\cO(1/K^2)$ rate is proposed for solving convex-concave 
SP problems when $\Phi$ is not bilinear.} 
Our new method captures a wide range of optimization problems, especially, those that can be cast as an SDP or a QCQP, or more generally conic convex problems with nonlinear constraints. 
{In the context of constrained optimization, 
using the proposed backtracking scheme, we demonstrated convergence results even when a dual bound is not available or easily computable.}
It has been illustrated in the numerical experiments that the APD method can compete against second-order methods when we aim at a low-to-medium-level accuracy. {Moreover, APD and APDB methods performs very well in practice in compare to the other competitive methods.}

As a future work, we are eager to investigate the stochastic variant of problem \eqref{eq:original-problem} where $\Phi(x,y)=\bE_\xi[\varphi(x,y;\xi)]$ and $\xi$ is a random variable. Based on our preliminary work on this setup, we have realized that one needs to require a more complicated set of assumptions on the step-size selection rule, and we will consider extending our APD framework to this stochastic setting. 
Indeed, we have already attempted to analyze randomized block-coordinate variants of APD in~\cite{hamedani2018iter,hamedani2019doubly}. In~\cite{hamedani2018iter}, we showed that if $\grad_x\Phi(\cdot,y)$ is coordinatewise Lipschitz for any fixed $y$ and Assumption~\ref{assum} holds, then the ergodic sequence generated by a variant of APD updating a randomly selected primal coordinate at each iteration together with a full dual variable update achieves the $\cO(m/k)$ convergence rate in a suitable error metric where $m$ denotes the number of coordinates for the primal variable, and $\{x_k,y_k\}$ converges to a random saddle point in an almost sure sense. 
Assuming that $\Phi(\cdot,y)$ is strongly convex for any $y$ with a uniform convexity modulus, and that $\Phi$ is linear in $y$, 
convergence rate of $\cO(m/k^2)$ is also shown. Moreover, in~\cite{hamedani2019doubly}, we considered $\min_x\max_y\Phi(x,y)$ where $\Phi(x,y)=\sum_{\ell=1}^p\Phi_{\ell}(x,y)$ for some $p\gg 1$, and both $x$ and $y$ are partitioned into block coordinates. We employ a block-coordinate primal-dual scheme in which randomly selected primal and dual blocks of variables are updated at every iteration. An ergodic convergence rate of $\cO(1/k)$ is obtained using variance reduction which is the first rate result to the best of our knowledge for the finite-sum non-bilinear saddle point problem, matching the best known rate for primal-dual schemes using full gradients. Using an acceleration for the setting where a single component gradient is utilized, a non-asymptotic rate of $\cO(1/\sqrt{k})$ is obtained. For both single and increasing sample size scenarios, almost sure convergence of the iterate sequence to a saddle point is shown. As a future work, we would like to incoorporate a stochastic line-search technique to our APD-based methods that use stochastic gradients or partial gradients for solving SP problems with $\Phi(x,y)=\bE_\xi[\varphi(x,y;\xi)]$.
\vspace*{-3mm}
\bibliographystyle{siam}
{\small
\bibliography{papers}}%
\section{Appendix}\label{append}
\begin{lemma}
\label{lem_app:prox}
Let $\cX$ be a finite dimensional normed vector space with norm $\norm{\cdot}_\cX$, $f:\cX\rightarrow\reals\cup\{+\infty\}$ be a closed convex function with convexity modulus $\mu\geq 0$ with respect to $\norm{\cdot}_\cX$, and $\bD:\cX\times\cX\rightarrow\reals_+$ be a Bregman distance function corresponding to a strictly convex function $\phi:\cX\rightarrow\reals$ that is differentiable on an open set containing $\dom f$. Given $\bar{x}\in\dom f$ and $t>0$, let
\begin{align}
\label{eq_app:prox}
x^+=\argmin_{x\in\cX} f(x)+t \bD(x,\bar{x}).
\end{align}
Then for all $x\in\cX$, the following inequality holds:
\begin{eqnarray}
f(x)+t\bD(x,\bar{x})\geq f(x^+) + t\bD(x^+,\bar{x})+t \bD(x,x^+)+\frac{\mu}{2}\norm{x-x^+}_\cX^2. \label{eq_app:bregman}
\end{eqnarray}
\end{lemma}
\begin{proof}
This result is a trivial extension of Property 1 in~\cite{Tseng08_1J}. The first-order optimality condition for \eqref{eq_app:prox} implies that
$0\in\partial f(x^+)+ t\grad_x \bD(x^+,\bar{x})$ -- here $\grad_x \bD$ denotes the partial gradient with respect to the first argument. Note that for any $x\in \dom f$, we have $\grad_x \bD(x,\bar{x})=\grad \phi(x)-\grad \phi(\bar{x})$. Hence,
$t(\grad \phi(\bar{x})-\grad \phi({x}^+))\in\partial f(x^+)$.
Using the convexity inequality for $f$, we get
\begin{eqnarray*}
f(x)\geq f(x^+)+t\fprod{\grad \phi(\bar{x})-\grad \phi({x}^+),~x-x^+}+\frac{\mu}{2}\norm{x-x^+}_\cX^2.
\end{eqnarray*}
The result in \eqref{eq_app:bregman} immediately follows from this inequality.
\end{proof}
\end{document}

%% file: ex_shared.tex
\usepackage{lipsum}
\usepackage{amsfonts}
\usepackage{graphicx}
\usepackage{epstopdf}
\usepackage{algorithmic}
\usepackage{amsmath,amssymb}
\usepackage{multirow}
\usepackage{subcaption}
\usepackage{enumerate}

\usepackage{breqn}
\usepackage{todonotes,comment}
\usepackage{etoolbox}
\makeatletter
\patchcmd{\@addmarginpar}{\ifodd\c@page}{\ifodd\c@page\@tempcnta\m@ne}{}{}
\makeatother
\reversemarginpar
\usepackage[normalem]{ulem} 
\newcommand\str{\bgroup\markoverwith
{\textcolor{red}{\rule[0.5ex]{2pt}{1.5pt}}}\ULon} 

\ifpdf
  \DeclareGraphicsExtensions{.eps,.pdf,.png,.jpg}
\else
  \DeclareGraphicsExtensions{.eps}
\fi

\newsiamremark{remark}{Remark}
\newsiamremark{hypothesis}{Hypothesis}
\crefname{hypothesis}{Hypothesis}{Hypotheses}
\newsiamthm{claim}{Claim}

\newcommand{\TheTitle}{A Primal-Dual Algorithm for General Convex-Concave Saddle Point Problems}
\newcommand{\TheShortTitle}{A Primal-Dual Algorithm for SP Problems}
\newcommand{\TheAuthors}{Erfan Yazdandoost Hamedani, and Necdet Serhat Aybat}

\headers{\TheShortTitle}{\TheAuthors}

\title{{\TheTitle}
}

\author{
  Erfan Yazdandoost Hamedani\thanks{Industrial \& Manufacturing Engineering Department, The Pennsylvania State University, PA
  (\email{evy5047@psu.edu}, \email{nsa10@psu.edu}).}
  \and
  Necdet Serhat Aybat\footnotemark[2]}

\usepackage{amsopn}
\DeclareMathOperator{\diag}{diag}


%% file: defs.tex
\def\grad{\nabla}

\def\ba{\mathbf{a}}
\def\bb{\mathbf{b}}

\def\be{\mathbf{e}}

\def\br{\mathbf{r}}

\def\bD{\mathbf{D}}
\def\bE{\mathbf{E}}

\def\cB{\mathcal{B}}

\def\cG{\mathcal{G}}

\def\cK{\mathcal{K}}
\def\cL{\mathcal{L}}
\def\cM{\mathcal{M}}
\def\cN{\mathcal{N}}
\def\cO{\mathcal{O}}
\def\cP{\mathcal{P}}

\def\cS{\mathcal{S}}
\def\cT{\mathcal{T}}

\def\cX{\mathcal{X}}
\def\cY{\mathcal{Y}}
\def\cZ{\mathcal{Z}}

\def\smskip{\smallskip}

\def\texitem#1{\par\smskip\noindent\hangindent 25pt
               \hbox to 25pt {\hss #1 ~}\ignorespaces}

\def\norm#1{\left\|#1\right\|}

\newcommand{\BEAS}{\begin{eqnarray*}}
\newcommand{\EEAS}{\end{eqnarray*}}
\newcommand{\BEA}{\begin{eqnarray}}
\newcommand{\EEA}{\end{eqnarray}}
\newcommand{\BEQ}{\begin{eqnarray}}
\newcommand{\EEQ}{\end{eqnarray}}
\newcommand{\BIT}{\begin{itemize}}
\newcommand{\EIT}{\end{itemize}}
\newcommand{\BNUM}{\begin{enumerate}}
\newcommand{\ENUM}{\end{enumerate}}

\newcommand{\BA}{\begin{array}}
\newcommand{\EA}{\end{array}}

\newcommand{\ones}{\mathbf 1}

\newcommand{\reals}{\mathbb{R}}

\newcommand{\dom}{\mathop{\bf dom}}

\newcommand{\intr}{\mathop{\bf int}}
\newcommand{\relint}{\mathop{\bf rel int}}

\newif\ifpagenumbering
\pagenumberingtrue

\pagenumberingfalse

\newsavebox{\theorembox}
\newsavebox{\lemmabox}
\newsavebox{\defnbox}
\newsavebox{\corollarybox}
\newsavebox{\propositionbox}
\newsavebox{\remarkbox}
\newsavebox{\assbox}
\savebox{\theorembox}{\noindent\bf Theorem}
\savebox{\lemmabox}{\noindent\bf Lemma}
\savebox{\defnbox}{\noindent\bf Definition}
\savebox{\corollarybox}{\noindent\bf Corollary}
\savebox{\propositionbox}{\noindent\bf Proposition}
\savebox{\remarkbox}{\noindent\bf Remark}
\savebox{\assbox}{\noindent\bf Assumption}
\newtheorem{assumption}{\usebox{\assbox}}
\newtheorem{defn}{\usebox{\defnbox}}

